\documentclass[oneside,reqno,a4paper,11pt]{amsart}
\usepackage{amsfonts}
\usepackage{amsmath,amsfonts,amssymb}
\usepackage{CJK,bbm,color,soul,supertabular,longtable,verbatim}
\usepackage[warning, easyold,math]{easyeqn}
\usepackage{graphicx}
\usepackage{color}
\usepackage{latexsym, amsfonts,mathrsfs, amsmath, amssymb, amscd, epsfig}
\usepackage{mathrsfs}
\usepackage{ esint }
\usepackage{amsthm}
\usepackage{etex}
\usepackage{epstopdf}

\newtheorem {theorem}{Theorem}[section]
\newtheorem {lemma}[theorem]{{\bf Lemma}}
\newtheorem {proposition}[theorem]{\bf Proposition}

\theoremstyle{remark}
\newtheorem {remark}{{\bf Remark}}[section]

\theoremstyle{problem}

\theoremstyle{definition}
\newtheorem {definition}{{\bf Definition}}[section]
\theoremstyle{plain} \numberwithin {equation}{section}

\begin{document}
\vspace{1cm}

\title[A free boundary problem]{A free boundary problem for semi-linear elliptic equation and its applications$^*$}
\author[Jianfeng Cheng,\\Lili Du]{Jianfeng Cheng$^{\lowercase {1,2}}$,\ \ Lili Du$^{\lowercase {2}}$}

\thanks{$^*$This work is supported in part by NSFC grant 11971331.}
\thanks{ E-Mail: jianfengcheng@126.com (J. Cheng), E-mail: dulili@scu.edu.cn (L. Du).}
 \maketitle
\begin{center}

 $^1$ The Institute of Mathematical Sciences,

  The Chinese University of Hong Kong, Hong Kong.

 $^2$ Department of Mathematics, Sichuan University,

          Chengdu 610064, P. R. China.

\end{center}

\begin{abstract} In this paper, we consider a free boundary problem of a semilinear nonhomogeneous elliptic equation with Bernoulli's type free boundary. The existence and regularity of the solution to the free boundary problem are established by use of the variational approach. In particular, we establish the Lipschitz continuity and non-degeneracy of a minimum, and regularity of the free boundary. As a direct and important application, the well-posedness results on the steady, incompressible inviscid jet and cavitational flow with general vorticity are also obtained in this paper.

\end{abstract}

\

\begin{center}
\begin{minipage}{5.5in}

\

\end{minipage}
\end{center}

\

\everymath{\displaystyle}
\newcommand {\eqdef }{\ensuremath {\stackrel {\mathrm {\Delta}}{=}}}

\tableofcontents

\def\Xint #1{\mathchoice
{\XXint \displaystyle \textstyle {#1}} %
{\XXint \textstyle \scriptstyle {#1}} %
{\XXint \scriptstyle \scriptscriptstyle {#1}} %
{\XXint \scriptscriptstyle \scriptscriptstyle {#1}} %
\!\int}
\def\XXint #1#2#3{{\setbox 0=\hbox {$#1{#2#3}{\int }$}
\vcenter {\hbox {$#2#3$}}\kern -.5\wd 0}}
\def\ddashint {\Xint =}
\def\dashint {\Xint -}
\def\clockint {\Xint \circlearrowright } 
\def\counterint {\Xint \rotcirclearrowleft } 
\def\rotcirclearrowleft {\mathpalette {\RotLSymbol { -30}}\circlearrowleft }
\def\RotLSymbol #1#2#3{\rotatebox [ origin =c ]{#1}{$#2#3$}}

\def\aint{\dashint}

\def\arraystretch{2}
\def\eps{\varepsilon}

\def\s#1{\mathbb{#1}} 
\def\t#1{\tilde{#1}} 
\def\b#1{\overline{#1}}
\def\N{\mathcal{N}} 
\def\M{\mathcal{M}} 
\def\R{{\mathbb{R}}}
\def\B{{\mathcal{B}}}
\def\BB{\mathfrak{B}}
\def\F{{\mathcal{F}}}
\def\G{{\mathcal{G}}}
\def\ba{\begin{array}}
\def\ea{\end{array}}
\def\be{\begin{equation}}
\def\ee{\end{equation}}

\def\bes{\begin{mysubequations}}
\def\ees{\end{mysubequations}}

\def\cz#1{\|#1\|_{C^{0,\alpha}}}
\def\ca#1{\|#1\|_{C^{1,\alpha}}}
\def\cb#1{\|#1\|_{C^{2,\alpha}}}
\def\psir{\left|\frac{\nabla\psi}{r}\right|^2}
\def\lb#1{\|#1\|_{L^2}}
\def\ha#1{\|#1\|_{H^1}}
\def\hb#1{\|#1\|_{H^2}}
\def\th{\theta}
\def\Th{\Theta}
\def\cin{\subset\subset}
\def\Ld{\Lambda}
\def\ld{\lambda}
\def\ol{{\Omega_L}}
\def\sla{{S_L^-}}
\def\slb{{S_L^+}}
\def\e{\varepsilon}
\def\C{\mathbf{C}} 
\def\cl#1{\overline{#1}}
\def\ra{\rightarrow}
\def\xra{\xrightarrow}
\def\g{\nabla}
\def\a{\alpha}
\def\b{\beta}
\def\d{\delta}
\def\th{\theta}
\def\fai{\varphi}
\def\O{\Omega}
\def\ol{{\Omega_L}}
\def\psirk{\left|\frac{\nabla\psi}{r+k}\right|^2}
\def\tO{\tilde{\Omega}}
\def\tu{\tilde{u}}
\def\tv{\tilde{v}}
\def\trho{\tilde{\rho}}
\def\W{\mathcal{W}}
\def\f{\frac}
\def\p{\partial}
\def\m{\omega}
\def\B{\mathcal{B}}
\def\H{\Theta}
\def\msS{\mathscr{S}}
\def\bq{\mathbf{q}}
\def\msE{\mathcal{E}}
\def\mfa{\mathfrak{a}}
\def\mfb{\mathfrak{b}}
\def\mfc{\mathfrak{c}}
\def\mfd{\mathfrak{d}}
\def\Div{\text{div}}
\def\Rot{\text{rot}}
\def\Curl{\text{curl}}
\def\mcL{\mathcal{L}}
\def\mcR{\mathcal{R}}
\def\f{\frac}
\def\p{\partial}
\def\o{\omega}
\def\h{_2^{\frac{1}{2}}}
\def\hh{_2^2}
\def\hhh{_2^{\frac{2}{3}}}
\def\k{_2^{\frac{3}{2}}}
\def\ii{\int_{0}^{t}\int}
\def\xiao{\leq}
\def\L{\Lambda}
\def\s{\sigma}
\def\q{\sqrt{\tau}}

\section{Introduction}

In this paper, we investigate the free boundary problem of a semilinear nonhomogeneous elliptic equation
\be\label{aa1}\left\{\ba{ll} -\Delta\psi=f(\psi)\ \ \ &\text{in}\ \ \O\cap\{\psi>0\},\\
 |\g\psi|=\ld(X)\ \ \ &\text{on}\ \ \O\cap\p\{\psi>0\},\\
 \psi=\Psi_0\ \ \ &\text{on}\ \ \p\O,\ea\right.\ee where $\O$ is a connected open and bounded domain in $\mathbb{R}^2$, $\p\O$ is a local Lipschitz graph. The given functions $\Psi_0\in C^{0,1}(\bar\O)$ with $\Psi_0\geq0$ and $\ld(X)\in C^{0,\beta}(\bar\O)$ with $0<\ld_1\leq\ld(X)\leq\ld_2<+\infty$, $f(\psi)$ is so-called {\it vorticity strength function} and $f\in C^{1,\beta}(\mathbb{R})$, $0<\beta<1$.

 The semilinear nonhomogeneous equation $-\Delta\psi=f(\psi)$ characterizes the steady, incompressible flow of an ideal fluid in $\mathbb{R}^2$ with non-zero variable vorticity. A wide class of vorticity distributions is considered. We are interesting in the existence, regularity and geometric properties of the solution $\psi$ and the free boundary $\O\cap\p\{\psi>0\}$.

 On the other hand, physical motivation for our study lies in the proof of minimizing the functional
 $$J(\psi)=\int_{\O}|\nabla\psi|^2+F(\psi)+\ld^2(X)I_{\{\psi>0\}}dX,$$ which is related closely to the jet flow problem with variable vorticity. Here, $F(t)=-2\int_0^tf(s)ds$, $I_E$ is the characteristic function of the set $E$. Here and after, denote $dX=dxdy$ for simplicity.

 In the remarkable paper \cite{AC1} by H. Alt and L. Caffarelli, some results on Lipschitz continuity, non-degeneracy lemma of a minimizer $\psi$, and the analyticity of the free boundary were obtained for the special case $f(\psi)=0$. The mathematical results of \cite{AC1} were used immediately in the study of jet flows \cite{ACF1,ACF2,ACF3,CDX1} and impinging jet flows \cite{CD,CDW} of inviscid, irrotational and incompressible fluid. For the nonhomogeneous problem ($f(\psi)\neq 0$), A. Friedman \cite{FA2} established the first result on the regularity of the solution and the free boundary for the linear nonhomogeneous case $\Delta\psi=P(x,y)$. Based on this result in \cite{FA2}, the well-posedness result of cavitational flow \cite{CDZ,FA2} and impinging jet flows \cite{CDW1} of inviscid and incompressible fluid with constant vorticity $(P(x,y)=P_0)$ were obtained. The similar results were extended to the quasilinear homogeneous case $$J(\psi)=\int_{\O} F(|\nabla\psi|^2)+\ld^2(X)I_{\{\psi>0\}} dX$$ in \cite{ACF8} and nonlinear homogeneous case $$J(\psi)=\int_{\O} a_{ij}(\psi)D_i\psi D_j\psi+\ld^2(X)I_{\{\psi>0\}} dX$$ in \cite{OY}.

 The first purpose of this paper is to establish the existence, regularity of the solution to the free boundary problem \eqref{aa1}, and extend the classical results in \cite{AC1} to the semilinear nonhomogeneous elliptic equation. In particular, the Lipschitz continuity, non-degeneracy lemma of the solution and regularity of the free boundary are obtained (please see Theorem \ref{lb5} for Lipschitz continuity and Theorem \ref{lc15} for the regularity of the free boundary).

 On the other hand, there is a large number of literatures on the regularity criteria of the free boundary for linear elliptic problem,
 \be\label{aa0}\left\{\ba{ll} \sum_{i,j=1}^na_{ij}(x)\p^2_{x_ix_j}\psi=f(x)\ \ \ &\text{in}\ \ \O\cap\{\psi>0\},\\
 |\g\psi|=\ld(x)\ \ \ &\text{on}\ \ \O\cap\p\{\psi>0\}.\ea\right.\ee For the Laplace operator and $f\equiv0$, L. Caffarelli showed in his pioneer work \cite{C1} that Lipschitz free boundary is $C^{1,\alpha}$-smooth, furthermore, he also showed in \cite{C2} that "flat" free boundary is Lipschitz. And higher regularity, such as $C^\infty$-smoothness and analyticity of the free boundary follow from the elegant work of D. Kinderlehrer and L. Nirenberg \cite{KN}. In the case of the homogeneous case $f\equiv0$, the regularity criteria results in spirit of works \cite{C1,C2} have been subsequently obtained for more general operators, such as concave fully nonlinear uniform elliptic operators of the form $F(D^2\psi)$ in \cite{W1,W2} and nonconcave fully nonlinear uniform elliptic operators of the form $F(D^2\psi,D\psi)$, and references \cite{CSF,FS1,FS2}. Moreover, Silva showed that the Lipschitz free boundary of the problem \eqref{aa0} in non-homogenous case $(f\neq0)$ is $C^{1,\alpha}$ in \cite{S}. Recently, Weiss and Zhang \cite{WZ} investigated the regularity of the free boundary problem of semilinear elliptic equation $-\Delta\psi=f(\psi)$, and showed the regularity criteria that the $W_{loc}^{1,1}$ free boundary is a graph implies the $C^{2,\alpha}$-regularity of the free boundary. However, we would like to emphasize that the results in this paper are not the conditional regularity of the free boundary and we do not give any apriori assumptions and topological property on the free boundary and the solution.

  An other related interesting problem is the semilinear Dirichlet problem with degenerate gradient on the free boundaries,
   \be\label{aa00}\left\{\ba{ll} \Delta\psi=f(x,\psi)\ \ \ &\text{in}\ \ \O\cap\{\psi>0\},\\
 \psi=|\g\psi|=0\ \ \ &\text{on}\ \ \O\cap\p\{\psi>0\}.\ea\right.\ee
The problem is used for modeling the distribution of a gas with
density, $\psi(x)$, in reaction with a porous catalyst pellet $\O$.
Alt and Phillips \cite{AP} investigated the regularity and the
geometry properties of the solution and the free boundaries with
some special structural conditions on the force term $f(x,\psi)$.
The degenerate gradient condition on the free boundaries implies
that desired optimal regularity of the solution is $C^{1,\alpha}$ in
$\O$. However, in this paper, the gradient of the solution is
non-zero on the free boundaries, thus the optimal regularity desired
here is only Lipschitz in $\O$. This is the main difference between
the free boundary problem \eqref{aa1} and the one \eqref{aa00}.

 The second purpose of this paper is to establish the well-posedness theory of the jet flows of inviscid, incompressible fluid with general vorticity issuing from a semi-infinitely long nozzle for symmetric case (see Figure \ref{f1}). It's a direct and important application of the mathematical theory of the free boundary problem \eqref{aa1}.

\begin{figure}[!h]
\includegraphics[width=100mm]{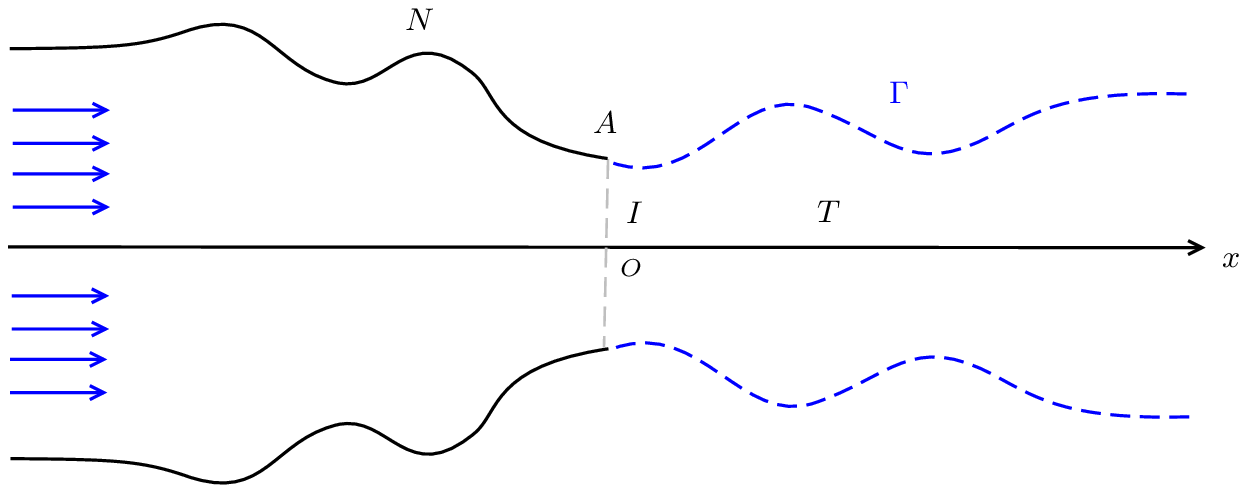}
\caption{Incompressible jet flow}\label{f1}
\end{figure}

Before we proceed with the bulk of this paper, we would like to
mention the previous results on the well-posedness of the free
boundary problem for incompressible inviscid jet flows.

 For the irrotational inviscid incompressible flow without surface tension, the existence of an axisymmetric jet flow \cite{ACF3}, an asymmetric jet flow \cite{ACF1} and impinging jet flow \cite{CD,CDW} has been established based on the fundamental work \cite{AC1}. The main advantage for two-dimensional irrotational flow is that the stream function $\psi$ solves the linear elliptic equation with Bernoulli's type free boundary. And then the conformal mapping and Green function approach for the linear elliptic equation work for the irrotational flow. Furthermore, for the simplified setting of constant vorticity, which corresponds to a constant vorticity strength $f(\psi)\equiv P_0$ in \eqref{aa1}, the well-posedness of axially symmetric cavitational flow in \cite{FA2}, symmetric impinging jet flow in \cite{CDW1} and symmetric cavitational flow in an infinitely long nozzle in \cite{CDZ} were established, based on the regularity of the free boundary problem with Poisson equation $\Delta\psi=P_0$. The assumption of an inviscid incompressible jet with non-zero constant vorticity provides us with the simplest case of a flow that is not irrotational and is attractive for its analytical tractability. However, the simplicity setting is not a mere mathematical convenience, as it is also physical relevant. Indeed, the vorticity remains invariant along the each streamline for inviscid fluid, therefore the fluid possesses constant vorticity as long as we impose the constant vorticity in the upstream.

 Generations of scientists working in fluid dynamics has recognized the importance of the vorticity. It has provided a powerful qualitative description for many of the important phenomena of fluid mechanics. From the mathematical point of view, the strong nonlinearity of the equations of vortex motion has made the analysis difficult.

 In the present paper, to establish the well-posedness of the incompressible jet issuing from a semi-infinitely long nozzle with general vorticity, we will impose the vorticity in the inlet of the nozzle, and then the vorticity strength function $f(\psi)$ is determined uniquely along the streamlines. To realize this idea, one of the key points is to show the well-posedness of the streamlines. A free boundary problem with semilinear nonhomogeneous elliptic equation as \eqref{aa1} is formulated. An associated variational problem is shown to possess a minimizer, which yields a solution of semilinear free boundary problem.

\section{Regularity and non-degeneracy of the solutions}

To investigate the free boundary of the semilinear elliptic equation
\eqref{aa1}, based on the works \cite{AC1,ACF8,FA2} by Alt,
Caffarelli and Friedman, we study the variational problem in this
section, and establish the non-degeneracy lemma and the regularity
of the solution.

\subsection{The variational problem}

 Let $\O\subset\mathbb{R}^2$
be a bounded and connected open domain and $\p\O$
is a locally Lipschitz graph, $I_{E}$ is the characteristic function
of a set $E\subset\mathbb{R}^2$ and $f(t)\in C^{1,\beta}((-\infty,+\infty))$, which
satisfies that \be\label{c0}\text{$0\leq f(t)\leq \L$ for any $t\leq 0$ and $-\L\leq
f'(t)\leq 0$ for any $t\in\mathbb{R}$},\ee where $\L\geq0$ is a constant.  Denote
$$F(t)=-2\int_0^tf(s)ds.$$
It is easy to check that \be\label{c1}F(0)=0,\ -2\L\leq F'(t)\leq 0
\ \text{for any $t\leq0$ and}\ 0\leq F''(t)\leq 2\L \ \text{for any
$t\in\mathbb{R}$}.\ee Define a function $\Psi\in C^0(\bar{\O})\cap
C^{2,\alpha}(\O)$ and $\Psi$ satisfies
\be\label{c001}\Delta\Psi+f(\Psi)\leq0\ \ \text{in $\O$ and $\Psi>0$
in $\O$}.\ee

Consider an energy functional
$$J(\psi)=\int_{\O}\left(|\nabla\psi|^2+F(\psi)+\ld^2(X)I_{\{\psi>0\}}\right)dX,$$  and an admissible
set
$$K=\{\phi\in H^1(\O)\mid~\phi\leq\Psi\ \text{a.e. in $\O$},~~\phi=\Psi_0\ \text{on}~ S\},$$ where $S$ is a given subset of $\p\O$ with $\mathcal{H}^1(S)>0$,
$\Psi_0\in H^1(\O)$ and $0\leq\Psi_0\leq\Psi$ a.e. on $S$,
$\ld(X)\in C^{0,\beta}(\bar\O)$ and satisfies
\be\label{c2}0<\ld_1\leq\ld(X)\leq\ld_2<+\infty\ \ \text{for any
$X\in\bar\O$}.\ee

{\bf The variational problem $(P)$:} Find a $\psi\in K$, such that
$$J(\psi)=\min_{\phi\in K} J(\phi).$$
\begin{remark} For the special case $F(t)\equiv0$, the variational problem have been studied by Alt and Caffarelli in
\cite{AC1}. They established the Lipschitz continuity and
non-degeneracy of the minimizer $\psi$, and obtained the regularity of the free
boundary. Moreover, Friedman in \cite{FA2} investigated the variational problem for $F(t)=-2P(x,y)t$ with $P(x,y)\geq 0$ and $P(x,y)\in C^{0,1}(\bar\O)$, and established the regularity of the free boundary.

\end{remark}

As the first step, the existence of the minimizer $\psi$ to the variational problem
$(P)$ can be obtained by using the similar arguments in Lemma 1.3 in
\cite{AC1}, we omit it here.

\subsection{Lipschitz continuity of the minimizer}
In this subsection, we will obtain the Lipschitz continuity of the minimizer $\psi$.

We first focus on the H\"older continuity of the minimizer $\psi$ in $\O$, and
show that the minimizer $\psi$ satisfies the semilinear elliptic
equation in $\O\cap\{\psi>0\}$.
\begin{lemma}\label{lb3} (1) $\psi\in C^\alpha(\O)$ for some
$\alpha\in(0,1)$ and $\psi\geq 0$ in $\O$.\\
(2) $\psi$ satisfies that
\be\label{c5}\text{$\Delta\psi+f(\psi)\geq0$ in $\O$}\ee  in the weak
sense and \be\label{c6}\text{$\Delta\psi+f(\psi)=0$ in
$\O\cap\{\psi>0\}$.}\ee Furthermore, $\psi\in C^{2,\alpha}(D)$ for
any compact subset $D$ of $\O\cap\{\psi>0\}$.
\end{lemma}

\begin{proof}

(1) Let $B_r\subset\O$ and define a function $\phi$ as follows
$$\Delta\phi+f(\phi)=0\ \ \ \text{in}~~B_r,\ \ \text{and} \ \ \ \phi=\psi\ \ \text{outside}\ \
B_{r}.$$ The maximum principle gives that $\phi\leq\Psi$ in $\O$,
and thus $\phi\in K$. Then we have
\be\label{cc1}\ba{rl}0\leq&J(\phi)-J(\psi)\\
=&\int_{B_r} -|\nabla(\phi-\psi)|^2+2
\nabla(\phi-\psi)\cdot\nabla\phi +F(\phi)-F(\psi)dX\\
&+\int_{B_r}\ld^2(X)(I_{\{\phi>0\}}-I_{\{\psi>0\}})dX\\
\leq&\int_{B_r} -|\nabla(\phi-\psi)|^2
dX+\ld_2^2\int_{B_r}I_{\{\phi>0\}} dX,\ea\ee which implies that
\be\label{cc2}\ba{rl}\int_{B_r} |\nabla(\phi-
\psi)|^2dX\leq\ld_2^2\int_{B_r}I_{\{\phi>0\}} dX\leq\ld_2^2\pi
r^2.\ea\ee With the aid of gradient $L^2$-estimate in \eqref{cc2},
we can now use the method of Morrey in Theorem 5.3.6 in \cite{MB} to
deduce the H\"{o}lder continuity of the minimizer.

Set $\psi^\e=\psi-\e\min\{\psi,0\}$ for any $\e\in(0, 1)$. It is
clear that $\psi^\e\in K$ and $\psi^\e\geq\psi$ in $\O$. Furthermore,
$\psi>0$ if and only if $\psi^\e>0$ in $\O$, and one has
\be\label{c3}\ba{rl}0\leq&J(\psi^\e)-J(\psi)\\
 \leq&\int_{\O}((1-\e)^2-1)|\nabla\min\{\psi,0\}|^2-\e
 F'((1-\e)\min\{\psi,0\})\min\{\psi,0\}
dX.\ea\ee Since $F'(t)\leq 0$ for any $t\leq0$ in \eqref{c1}, it follows from \eqref{c3} that
$$\int_{\O}|\nabla\min\{\psi,0\}|^2dX\leq 0,$$
which implies that
$$\text{$\psi(x,y)\geq 0$  in $\O$}.$$

(2) For any $\xi\in C_0^{\infty}(\O)$ with $\xi\geq 0$ in $\O$, it
follows from the statement (1) that $\psi-\e\xi\in K$ and
$\{\psi-\e\xi>0\}\subset\{\psi>0\}$ for any $\e>0$, then we have
$$\ba{rl}0\leq&\lim_{\e\rightarrow 0}\left(\f1\e\int_{\O}
|\nabla\psi-\e\g\xi|^2-|\nabla \psi|^2+F(\psi-\e\xi)-F(\psi)
dX\right)\\
=&-\int_{\O}2\nabla\psi\cdot\nabla\xi+F'(\psi)\xi dX,\ea$$ which
gives \eqref{c5}.

By virtue of the continuity of $\psi$, we have that $\O\cap\{\psi>0\}$
is open. For any $\xi\in C_0^{\infty}(\O\cap\{\psi>0\})$, it
is easy to check that $\psi(x,y)-\e\xi(x,y)\leq \Psi(x,y)$ in $\O$
and $\{\psi-\e\xi>0\}\subset\{\psi>0\}$ for small $|\e|>0$. Obviously,
$\psi-\e\xi\in K$ and we have
$$\ba{rl}0\leq-\e\int_{\O}2\nabla\psi\cdot\nabla\xi+F'(\psi-\e\xi)\xi
dX+o(\e),\ea$$ which implies that
\be\label{c06}0=\int_{\O}2\nabla\psi\cdot\nabla\xi+F'(\psi-\e\xi)\xi
dX \ \ \ \text{for any $\e$}.\ee Taking $\e\rightarrow 0$ in \eqref{c06} yields
\eqref{c6}.

It follows from the Schauder interior estimate in \cite{GT} that
$\psi\in C^{2,\alpha}(D)$ for any compact subset $D$ of
$\O\cap\{\psi>0\}$.

\end{proof}

Denote
$$\Gamma=\O\cap\p\{\psi>0\}\ \ \text{ the free boundary of $\psi$}.$$

The dynamic boundary condition on the free boundary $\Gamma$ can be verified in the following.
\begin{proposition}\label{lb2} If $\ld(X)\in C^{1,\beta}(\bar\O)$, for any minimizer $\psi$, we have \be\label{b005}\lim_{\e\downarrow
0}\int_{\O\cap\p\{\psi>\e\}}\left(|\nabla
\psi|^2-\ld^2(X)-F(\psi)\right)\eta\cdot\nu_\e dS=0, \ee for any 2-vector
$\eta\in \left(C_0^1(\O)\right)^2$, where $\nu_\e$ is the outer normal vector to $\O\cap\p\{\psi>\e\}$ with $\e>0$. Similarly, if a segment $l\subset
\p\O\cap\p\{\psi>0\}$ is $C^{1,\alpha}$ and $\psi=0$ on $l$, we have
 \be\label{b006}|\nabla
\psi|\geq\ld(X)\ \ \text{on $l$}. \ee
\end{proposition}
\begin{proof}We define a diffeomorphism
$$Y=\xi_\delta(X):\O\rightarrow\O$$ by $\xi_\delta(X)=X+\delta\eta(X)$ for any
$\eta(X)\in \left(C_0^1(\O)\right)^2$, where $\delta$ is a real
number and $|\delta|>0$ is suitable small.

Denote
$$\text{$\psi_\delta(\xi_\delta(X))=\psi(X)$}.$$ It's easy to
verify that $\psi_\delta\in K$ and
$$(D\xi_\delta(X))^{-1}=(E+\delta\nabla\cdot\eta E-\delta D \eta)(\text{det} D\xi_\delta)^{-1}\quad \text{and}\quad \text{det} D\xi_\delta
=1+\delta\nabla\cdot\eta +o(\delta),$$ where $E$ is the identity
matrix. Then we have \be\label{b007}\ba{rl}0\leq&\int_{\O}\left(|\nabla\psi_\delta|^2+F(\psi_\delta)+\ld^2I_{\{\psi_\delta>0\}}\right)dY-\int_{\O}\left(|\nabla\psi|^2+F(\psi)+\ld^2I_{\{\psi>0\}}\right)dX\\
=&\int_{\O}\left(|\nabla\psi(D\xi_\delta)^{-1}|^2+F(\psi)+\ld^2(\xi_\delta(X))I_{\{\psi>0\}}\right)\text{det}
D\xi_\delta
dX\\
&-\int_{\O}\left(|\nabla\psi|^2+F(\psi)+\ld^2I_{\{\psi>0\}}\right)dX\\
=&\delta\int_{\O\cap\{\psi>0\}}\left((|\nabla\psi|^2+F(\psi)+\ld^2)\nabla\cdot\eta-2\nabla\psi\cdot
D \eta\cdot\nabla\psi+\g\ld^2\cdot\eta\right)dX+o(\delta).\ea\ee

Due to the arbitrariness of $\delta$, the linear term of
\eqref{b007} in $\delta$ has to vanish, and this gives that
\be\label{b008}\ba{rl}0=&\int_{\O\cap\{\psi>0\}}\left((|\nabla\psi|^2+F(\psi)+\ld^2)\nabla\cdot\eta-2\nabla\psi\cdot
D \eta\cdot\nabla\psi+\g\ld^2\cdot\eta\right)dX\\
=&
\int_{\O\cap\{\psi>0\}}\nabla\cdot\left((|\nabla\psi|^2+\ld^2+F(\psi))\eta-2(\eta\cdot\nabla\psi)\nabla\psi\right)dX\\
=&\lim_{\e\downarrow
0}\int_{\O\cap\p\{\psi>\e\}}\left((|\nabla\psi|^2+\ld^2+F(\psi))\eta-2(\eta\cdot\nabla\psi)\nabla\psi\right)\cdot\nu_\e dS\\
=&\lim_{\e\downarrow
0}\int_{\O\cap\p\{\psi>\e\}}\left(\ld^2(X)-|\nabla\psi|^2+F(\psi)\right)\eta\cdot\nu_\e
dS.\ea\ee

Hence, we obtain the conclusion \eqref{b005}. To prove \eqref{b006}, define $\xi_\delta(X)=X+\delta\eta(X)$ for
any $\eta(X)\in \left(C^1(\O\cap\p\{\psi>0\})\right)^2$ with
$\eta=0$ on $(\O\cap\p\{\psi>0\})\setminus l$ and $\eta\cdot\nu\leq
0$ on $l$, where $\delta>0$ is suitable small and $\nu$ is the outer
normal vector of $l$.

Along the similar arguments in \eqref{b007} and \eqref{b008}, we
have
\be\label{b009}\ba{rl}0\leq&\int_{\O\cap\{\psi>0\}}\left((|\nabla\psi|^2+F(\psi)+\ld^2)\nabla\cdot\eta-2\nabla\psi\cdot
D \eta\cdot\nabla\psi+\g\ld^2(X)\cdot\eta\right)dX\\
=&
\int_{\O\cap\{\psi>0\}}\nabla\cdot\left((|\nabla\psi|^2+\ld^2+F(\psi))\eta-2(\eta\cdot\nabla\psi)\nabla\psi\right)dX\\
=&\int_{l}\left(\ld^2(X)-|\nabla\psi|^2+F(\psi)\right)\eta\cdot\nu
dS.\ea\ee Since $l\in C^{1,\alpha}$, $\psi$ is also $C^{1}$ up to
$l$. Therefore, we complete the proof of \eqref{b006} by using
\eqref{b009}.

\end{proof}
\begin{remark} Proposition \ref{lb2} implies that if the free boundary
$\Gamma$ is $C^1$ and $\psi$ is $C^1$ in $\{\psi>0\}$
uniformly up to the free boundary, then
$$|\g\psi(X)|=\ld(X)\ \ \text{on the free boundary $\Gamma$}.$$
\end{remark}

Next, we will give the estimate for the Lipschitz norm for $\psi$ near the free boundary, this is the principal part in establishing the Lipschitz continuity of the minimizer $\psi$.
Denote $$\O_+=\O\cap\{\psi>0\}\ \ \text{and}\ \
\O_0=\O\cap\{\psi=0\}.$$
\begin{lemma}\label{lb4} Suppose $X_0=(x_0,y_0)\in\O$ with
$d(X_0)<\f12\min\{\text{dist} (X_0,\p\O),1\}$, where
$d(X_0)=\text{dist}(X_0,\O_0)$. Then there exists a positive constant $C=C(\L)$ depending only on $\L$, such that
$$\psi(X_0)\leq C(\ld_2+\L) d(X_0),$$ where $\L$ is defined in \eqref{c0}.
\end{lemma}

\begin{proof} We assume that
\be\label{c7}\psi(X_0)>Md(X_0)>0,\ee and derive an upper bound on
$M$. Denote $d=d(X_0)$ for simplicity. By scaling
$\psi_d(X)=\f{\psi(X_0+d(X-X_0))}{d}$ and
$\ld_d(X)=\ld(X_0+d(X-X_0))$ with $X\in B_1(X_0)$, we have
$$\text{$\Delta\psi_d+df(d\psi_d)=0 \ $ in $B_1(X_0)$.}$$

On the other hand, it is easy to check that
$$|df(d\psi_d)|\leq d|f(0)|+d^2\max_{\xi\geq0}|f'(\xi)|\psi_d\leq d\L+d^2\L\psi_d.$$
Thanks to Harnack's inequality in \cite{T} and \eqref{c7}, we have
\be\label{c07}\inf_{B_{\f34}(X_0)}\psi_d\geq cM-Cd\L\geq cM-C\L,\ee
where the constants $c$ and $C$ depend on $\L$.

Let $Y$ be a point in $\p B_1(X_0)\cap \{\psi_d=0\}$. Define a
function $\phi$ as follows
$$\Delta\phi+df(d\phi)=0\ \ \ \text{in}~~B_1(Y),\ \ \text{and} \ \ \ \phi=\psi_d\ \ \text{on}\ \ \p B_{1}(Y).$$

Since $\psi_d$ is a minimizer, by virtue of the convexity of $F(t)$,
one has
$$\ba{rl}0\leq&\int_{B_1(Y)} |\nabla\phi|^2-|\nabla \psi_d|^2+F(d\phi)-F(d\psi_d)
+\ld_d^2(I_{\{\phi>0\}}-I_{\{\psi_d>0\}})dX\\
\leq&\int_{B_1(Y)} -|\nabla(\phi-\psi_d)|^2+2\nabla(\phi- \psi_d)\cdot\g\phi dX\\
&+\int_{B_1(Y)}dF'(d\phi)(\phi-\psi_d)
+\ld_d^2I_{\{\psi_d=0\}}dX\\
\leq&-\int_{B_1(Y)} |\nabla(\phi- \psi_d)|^2
dX+\ld_2^2\int_{B_1(Y)}I_{\{\psi_d=0\}} dX,\ea$$ and thus \be\label{c9}\ba{rl}\int_{B_1(Y)} |\nabla(\phi- \psi_d)|^2
dX\leq\ld_2^2\int_{B_1(Y)}I_{\{\psi_d=0\}} dX.\ea\ee
It follows from the statement (2) in Lemma \ref{lb3} that $\Delta\psi_d+
df(d\psi_d)\geq0$ in $B_1(Y)$, and the maximum principle gives that
\be\label{c8}\phi\geq \psi_d\ \ \text{in}\ \ B_1(Y).\ee

 Then it follows from \eqref{c07} and \eqref{c8} that
$$\phi(X)\geq \psi_d(X)\geq cM-C\L>0\ \ \text{in}\ \ B_{\f34}(X_0)\cap
B_1(Y).$$ Using Harnack's inequality in \cite{T} again, we have
\be\label{c10}\phi(X)\geq cM-C\L=C_0>0\ \ \text{in}\ \
B_{\f12}(Y).\ee

We take $Y=0$ for simplicity and introduce a function
$$\varphi(X)=C_0\left(e^{-\beta|X|^2}-e^{-\beta}\right).$$ Then a direct computation gives that
$$\Delta\varphi+df(d\varphi)\geq C_0e^{-\beta |X|^2}(4\beta(\beta|X|^2-1)-d^2\L)+C_0d^2\L e^{-\beta}+df(0)>0,$$ for $\f12<|X|<1$, and sufficiently large $\beta>0$. This gives that
$$\Delta(\varphi-\phi)+df(d\varphi)-df(d\phi)>0\ \ \text{in}\ \ B_1(0)\setminus
B_{\f12}(0),$$ provided that $\beta$ is sufficiently large. It
follows from the maximum principle that
$$\phi(X)>\varphi(X)=C_0\left(e^{-\beta|X|^2}-e^{-\beta}\right)\geq cC_0(1-|X|)\ \ \text{in}\ \ B_1(0)\setminus
B_{\f12}(0),$$ which implies that
$$\phi(X)> (cM-C\L)(1-|X|)\ \ \text{in}\ \ B_1(0)\setminus
B_{\f12}(0).$$ This together with \eqref{c10} gives that
\be\label{c11}\phi(X)\geq (cM-C\L)(1-|X|)\ \ \text{in}\ \ B_1(0).\ee

With the aid of \eqref{c9} and \eqref{c11}, it follows from the
similar arguments in Lemma 3.2 in \cite{AC1} and Lemma 2.2 in
\cite{ACF8} that
$$(cM-C\L)^2\leq C\max_{X\in \O} \ld^2(X)=C\ld_2^2,$$ that is
$$M\leq C(\ld_2+\L),$$ where the constant $C=C(\L)$ depends only on $\L$.

\end{proof}

With the aid of Lemma \ref{lb4}, we will obtain the Lipschitz
continuity of $\psi$, which plays an essential part for constructing a Radon measure. To obtain the Lipschitz continuity, we should use the method of Harnack's inequality in \cite{T}.
\begin{theorem}\label{lb5} (1). The minimizer $\psi$ is Lipschitz continuous in $\O$, namely, $\psi\in C^{0,1}(\O)$. \\
(2). For any compact
subset $D$ of $\O$ which containing a free boundary point, the
Lipschitz coefficient of $\psi$ in $D$ is estimated by
$C(\ld_2+\L)$, and the constant $C$ depends only on $\L$, $D$ and $\O$.
\end{theorem}

\begin{proof} (1). Suppose that
$d(X_0)<\f13\min\{\text{dist} (X_0,\p\O),1\}$, where
$d(X_0)=\text{dist}(X_0,\O_0)$. Set
$$\t\psi(\t X)=\f{\psi(X_0+d(X_0)\t X)}{d(X_0)}.$$
By virtue of Lemma \ref{lb4}, one has
$$\t\psi(\t X)\leq C(\ld_2+\L)\ \ \text{in}\ \ B_1(0),$$
where the constant $C=C(\L)$ depends only on $\L$.

Since $\Delta\t\psi+d(X_0)f(d(X_0)\t\psi)=0$ in $B_1(0)$,  it follows
from the elliptic estimate for $\t\psi$ that
$$\ba{rl}|\nabla\t\psi(0)|\leq& C(\|\t\psi\|_{C^0(B_1(0))}+\|f(d(X_0)\t\psi)\|_{L^\infty(B_1(0))} d(X_0))\\
\leq& C(1+d^2(X_0)|f(0)|+d^2(X_0)\L)\|\t\psi\|_{C^0(B_1(0))}\\
\leq &C(\ld_2+\L),\ea$$ which implies that
$$|\nabla\psi(X_0)|=|\g\t\psi(0)|\leq C(\ld_2+\L).$$
 Thus, for any compact subset $D$ of $\O$, $|\nabla\psi(X)|$
 is bounded in $D\cap\{\psi>0\}\cap U$, where $U$ is a small
 neighborhood of the free boundary. It follows from
 Lemma \ref{lb3} that $\psi\in C^{2,\alpha}$ in
 $\O\cap\{\psi>0\}$, this together with $\g\psi=0$ a.e. in $\O\cap\{\psi=0\}$ gives that $\psi\in
 C^{0,1}(\O)$.

(2). Consider any connected domains $D\subset\subset G\subset\subset\O$
and $D$ contains at least one free boundary point. Let
$r_0=\f15\min\{\text{dist} (G,\p\O),1\}$ and $X\in G\cap \O_+$. Since
$D$ contains some free boundary points, and we have that $G$ is not
contained in $\O_+$, then we can find some finite points $X_1, X_2,. . . , X_m$ in
$G$ ($m$ depends only on $G$ and $\O$), such that
$$X_0=X,\ \ X_k\in B_{r_0}(X_{k-1}) \ \text{for}\ \ k=1,...,m,$$ and
$$ B_{2r_0}(X_{k})\subset \O_+\ \text{for}\ \ k=1,...,m-1, \text{and $B_{2r_0}(X_m)$ is not contained in $\O_+$}.$$
It follows from Lemma \ref{lb3} and Lemma \ref{lb4} that
$$\psi(X_m)\leq C(\ld_2+\L)r_0\ \ \text{and} \ \ \Delta\psi+f(\psi)=0\ \ \text{in}\ \  B_{2r_0}(X_{k})\ \text{for}\ \ k=1,...,m-1.$$
By virtue of Harnack's inequality in \cite{T}, we have
$$\psi(X_{k-1})\leq C\psi(X_k)+C\L r_0\ \text{for}\ \ k=1,...,m-1,$$and $$\psi(X_{m-1})\leq C\inf_{X\in B_{2r_0}(X_{m-1})}\psi(X)+C\L r_0\leq C\psi(X_m)+C\L r_0.$$
Inductively, we have \be\label{c01}\psi(X)=\psi(X_0)\leq C(\ld_2+\L)
\ \ \text{for all $X\in G$}.\ee

For any $X\in D\cap \O_+$ and $d(X)=\text{dist}(X,\O_0)$. We consider the
following two cases.

{\bf Case 1.} $d(X)\geq\f12\min\{dist(D,\p G),1\}$, by using
elliptic estimate and \eqref{c01}, one has
$$|\g\psi(X)|\leq C\left(\sup_{G}\psi+\sup_{G}|f(\psi)|\right)\leq C(\ld_2+\L),$$
where $C$ is a constant depending only on $\L$, $D, G$ and $\O$.

{\bf Case 2.} $d(X)<\f12\min\{dist(D,\p G),1\}$, it follows from
Lemma \ref{lb4} that
$$|\g\psi(X)|\leq C(\ld_2+\L),$$
where $C$ is a constant depending only on $\L$.

Hence, we complete the proof of Theorem \ref{lb5}.

\end{proof}

\subsection{Non-degeneracy of the minimizer}

 As a consequence of Theorem \ref{lb5}, we will obtain the non-degeneracy lemma of the minimizer $\psi$ in this subsection.
\begin{lemma}\label{lb6}Let $\psi$ be a minimizer, for any compact subset $D$ of $\O$, there exists a positive constant $C^*$ (depending only on $\L, D$ and $\O$), such that for any disc
$B_r(X_0)\subset D$, \be\label{c12}\f{1}{r}\fint_{\p B_r(X_0)}\psi
dS\geq C^*(\lambda_2+\L)\ee implies that
$$\psi(x,y)>0\ \ ~~~~\text{in} \ ~~B_r(X_0).$$
\end{lemma}

\begin{proof} Suppose that there exists a free boundary point $X'\in
B_r(X_0)\cap \O_0$. It follows from Theorem \ref{lb5} that $|\g\psi|\leq
C(\ld_2+\L)$, where $C$ is a constant depending only on $\L, D$
and $\O$.
Then we have
$$\psi(X)=\psi(X)-\psi(X')\leq C(\ld_2+\L)r \ \ \text{in}\ \ \overline{B}_r, $$
which contradicts to \eqref{c12}, provided that $C^*$ is
sufficiently large.
\end{proof}

The non-degeneracy lemma will be stated in the following.

\begin{lemma}\label{lb7} For any $\kappa\in(0,1)$,
there exists a positive constant $c^*_\kappa=c^*(\kappa)$, such that
for any disc
$B_r(X_0)\subset \O$ with $r<\f1\L\min\{c_{\kappa}^*\ld_1,1\}$,
\be\label{c13}\f{1}{r}\fint_{\p B_r(X_0)}\psi dS\leq c_\kappa^*\ld_1
\ee implies that
$$\psi(x,y)\equiv0~~~~\text{in}~~B_{\kappa r}(X_0).$$
\end{lemma}

\begin{proof} Without loss of generality, we assume that $X_0=0$. Set
$\psi_{r}(X)=\f{\psi(rX)}{r}$ and $\ld_r(X)=\ld(rX)$, we have
$$\Delta\psi_r+rf(r\psi_r)=0\ \ \
\text{in}~~\O_r\cap\{\psi_r>0\},$$  where $\O_r=\{X\mid rX\in\O\}$. Denote
$$\e=\sup_{X\in B_{\sqrt{\kappa}}(0)}\psi_r(X).$$

Since $\Delta\psi_r\geq -rf(r\psi_r)\geq -rf(0)\geq -\L r$ in $\O$, it follows from the maximum principle that
$$\psi_r(X)\leq \L r\f{1-|X|^2}{4}+\f{1-|X|^2}{2\pi}\int_{\p
B_1(0)}\f{\psi_r(Y)}{|X-Y|^2} dS_Y$$ for any $X\in B_{\sqrt{\kappa}}(0)$, which implies that
\be\label{c14}\e=\sup_{X\in B_{\sqrt{\kappa}}(0)}\psi_r(X)\leq\L
r+C(\kappa)\fint_{\p B_1(0)}\psi_r(Y) dS_Y.\ee

Define a function $\phi$ solving the following problem
$$\left\{\ba{ll}&\Delta\phi+rf(r\phi)=0~~~~\text{in}\ \ \ B_{\sqrt{\kappa}}(0)\setminus B_{\kappa}(0),\\
&\phi=0~~~~\text{in}\ \ \ B_{\kappa}(0),\ \ \ \ \
\phi=\e~~~~\text{outside}\ \ \ B_{\sqrt{\kappa}}(0).\ea\right.$$

It is easy to check that
$$\ba{rl}0\leq
&\int_{B_{\sqrt{\kappa}}(0)}|\nabla\min\{\psi_r,\phi\}|^2+F(r\min\{\psi_r,\phi\})+\ld_r^2I_{\{\min\{\psi_r,\phi\}>0\}}dX\\
&-\int_{B_{\sqrt{\kappa}}(0)}|\nabla\psi_r|^2+F(r\psi_r)+\ld_r^2I_{\{\psi_r>0\}}dX,\ea$$
which yields that
\be\label{c15}\ba{rl}&\int_{B_{\kappa}(0)}|\nabla\psi_r|^2+F(r\psi_r)+\ld_r^2I_{\{\psi_r>0\}}dX\\
\leq&\int_{B_{\sqrt{\kappa}}(0)\setminus
B_{\kappa}(0)}|\nabla\min\{\psi_r,\phi\}|^2+F(r\min\{\psi_r,\phi\})-|\nabla\psi_r|^2-F(r\psi_r)dX\\
\leq&\int_{B_{\sqrt{\kappa}}(0)\setminus B_{\kappa}(0)}-|\nabla\min\{\phi-\psi_r,0\}|^2+2\nabla\phi\cdot\nabla\min\{\phi-\psi_r,0\}\\
&+rF'(r\phi)\min\{\phi-\psi_r,0\} dX\\
\leq&2\int_{\p B_{\kappa}(0)}\min\{\phi-\psi_r,0\}\f{\p\phi}{\p\nu}
dS,\ea\ee where $\nu$ is the outer normal vector, we have used the fact
$$\int_{B_{\sqrt{\kappa}}(0)\setminus
B_{\kappa}(0)}I_{\{\min\{\psi_r,\phi\}>0\}}-I_{\{\psi_r>0\}}dX\leq
0,$$due to the fact $\{\min\{\psi_r,\phi\}>0\}\subset\{\psi_r>0\}$.

On the other hand, by virtue of the convexity of $F(t)$, we have
that $$F(r\psi_r)\geq F'(0)r\psi_r\geq-2\L r\psi_r,$$ which gives
that
\be\label{c150}\ba{rl}&\int_{B_{\kappa}(0)}|\nabla\psi_r|^2-2\L
r\psi_r+\ld_r^2I_{\{\psi_r>0\}}dX
\leq\int_{B_{\kappa}(0)}|\nabla\psi_r|^2+F(r\psi_r)+\ld_r^2I_{\{\psi_r>0\}}dX.\ea\ee

The standard elliptic estimates give that \be\label{c16}\sup_{X\in\p
B_{\kappa}(0)}|\nabla\phi(X)|\leq C\left(\sup_{X\in
B_{\sqrt{\kappa}}(0)}\phi(X)+r\|f(r\phi)\|_{L^\infty(B_{\sqrt{\kappa}}(0))}
\right)\leq C(\L,\kappa)(\e+\L r).\ee Furthermore, with the aid of
the trace theorem, it follows from \eqref{c15} - \eqref{c16} that
\be\label{c17}\ba{rl}&\int_{B_{\kappa}(0)}|\nabla\psi_r|^2+\ld_r^2I_{\{\psi_r>0\}}dX\\
\leq&2\int_{\p B_{\kappa}(0)}\min\{\phi-\psi_r,0\}\f{\p\phi}{\p\nu}
dS+2\int_{B_{\kappa}(0)}\L r\psi_rdX\\
\leq &C(\L,\kappa)\e_0\left(\int_{\p B_{\kappa}(0)}\psi_r
dS+\int_{B_{\kappa}(0)}\psi_rdX\right)\\
\leq&C(\L,\kappa)\e_0\left(\int_{ B_{\kappa}(0)}\psi_rI_{\{\psi_r>0\}}dX+C\int_{ B_{\kappa}(0)}|\nabla\psi_r|I_{\{\psi_r>0\}}dX\right)\\
\leq&C(\L,\kappa)\e_0\left(\f{\e_0}{\ld_1^2}\int_{
B_{\kappa}(0)}I_{\{\psi_r>0\}}dX+\f{C}{\ld_1}\int_{
B_{\kappa}(0)}|\nabla\psi_r|^2+\ld_r^2I_{\{\psi_r>0\}}dX\right)\\
\leq&C(\L,\kappa)\left(\f{\e^2_0}{\ld_1^2}+\f{\e_0}{\ld_1}\right)\int_{
B_{\kappa}(0)}|\nabla\psi_r|^2+\ld_r^2I_{\{\psi_r>0\}}dX,\ea\ee
where $\e_0=\e+\L r$. By virtue of \eqref{c13} and \eqref{c14}, one
has
$$\f{\e_0}{\ld_1}=\f{\e+\L r}{\ld_1}\leq C(\L,\kappa)\left(\f1{\ld_1}\fint_{\p B_1(0)}\psi_r dS+\f{\L r}{\ld_1}\right)<C(\L,\kappa)c_\kappa^*.$$
This implies that $\f{\e_0}{\ld_1}$ is
small enough, provided that $c_\kappa^*$ is small enough. This together with \eqref{c17} implies that
 $$\int_{
B_{\kappa}(0)}|\nabla\psi_r|^2+\ld_r^2I_{\{\psi_r>0\}}dX=0,\ \
$$ for sufficiently small $c_\kappa^*$. This implies
that $\psi=0$ in $B_\kappa(0)$, if $c_\kappa^*$ is small enough.

\end{proof}

\begin{remark}\label{re1} It should be noted that Lemma \ref{lb6} remains valid, provided that $B_r(X_0)$ is not
contained in $\O$ and $\psi=0$ on $B_r(X_0)\cap\p\O$.

\end{remark}

Finally, we give the density estimate of the free boundary point, which gives that $\mathcal{L}^2(D\cap\p\{\psi>0\})=0$ for any
compact subset $D$ of $\O$.

\begin{lemma}\label{lb8} Let $G$ be a compact subset of $\O$,
there exists a positive constant $c\in(0,1)$, such that for any disc
$B_r=B_r(X_0)\subset G$ with $X_0\in\Gamma$ and small $r$,
\be\label{c18}c<\f{\mathcal{L}^2(B_r\cap\{\psi>0\})}{\mathcal{L}^2(B_r)}<1-c,\ee
where $\mathcal{L}^2$ is the two-dimensional Lebesgue measure.
\end{lemma}

\begin{proof} {\bf Step 1.} Without loss of generality, we assume that $X_0=0$. It follows from Lemma \ref{lb7} that there exists a point $Y\in \p B_r$, such that
$\psi(Y)> cr$. Set $\varphi(X)=\psi(X)+\f{\L|X|^2}{4}$, one has
$$\Delta\varphi=\Delta\psi+\L\geq f(0)-f(\psi)\geq 0\ \ \ \text{in $B_r$},$$ which implies that
$$\f{\L r}{2\kappa}+\f1{\kappa r}\fint_{\p B_{\kappa r}(Y)}\psi(Z) dS_Z\geq\f1{\kappa r}\fint_{\p B_{\kappa r}(Y)}\varphi(Z) dS_Z\geq\f{\varphi(Y)}{\kappa
r}\geq\f{c}{\kappa},$$ for small $\kappa>0$, that is
$$\f1{\kappa r}\fint_{\p B_{\kappa r}(Y)}\psi(Z) dS_Z\geq\f{c-\L r}{2\kappa}>\f{c}{4\kappa}\ \ \text{if $r$ is small enough}.$$
By virtue of non-degeneracy Lemma \ref{lb7}, one has $$\text{$\psi>0$ in $B_{\kappa r}(Y)$,}$$ this gives
the left-hand side of \eqref{c18}.

{\bf Step 2.} Define a function $\phi$ solving the following problem
$$\Delta\phi+f(\phi)=0~~~~\text{in}~~B_r,\ \ \phi=\psi~~~~~~\text{outside}~~B_r.$$
The maximum
principle gives that $\psi\leq\phi\leq\Psi$ in $B_r$, and thus $\phi\in K$, where $\Psi$ is defined in \eqref{c001}.

Then we have
$$\ba{rl}0\leq J(\phi)-J(\psi)
\leq\int_{B_r} -|\nabla(\phi-\psi)|^2
dX+\ld_2^2\int_{B_r}I_{\{\psi=0\}} dX,\ea$$ which together
with Poincar\'e's inequality and H\"{o}lder inequality implies
that \be\label{c19}\ba{rl}\ld_2^2\int_{B_r}I_{\{\psi=0\}}
dX\geq&\int_{B_r} |\nabla(\phi-\psi)|^2
dX
\geq\f{c}{r^2}\int_{B_r} |\phi-\psi|^2 dX.\ea\ee

For any point $Y\in B_{\kappa r}$, since $\Delta\phi\leq M_0\L-f(0)$ with $M_0=\sup_{X\in\O}\Psi(X)$, one has
\be\label{cc20}\ba{rl}\phi(Y)\geq&\f{r^2-|Y|^2}{2\pi r}\int_{\p B_{r}}
\f{\psi}{|X-Y|^2}dS_X+(M_0\L-f(0))\f{|Y|^2-r^2}{4}\\
\geq&(1-C\kappa)\fint_{\p B_{r}}
\psi dS_X-Cr^2.\ea\ee
It follows from \eqref{cc20}, Lemma \ref{lb4} and
Lemma \ref{lb7} that
\be\label{c20}\ba{rl}\phi(Y)-\psi(Y)\geq& \phi(Y)-C\kappa r \\
\geq&(1-C\kappa)\fint_{\p B_{r}}
\psi dS_X-C\kappa r-Cr^2\\
\geq& c\ld_1 r,\ea\ee provided that $\kappa$ and $r$ are small.

Combining \eqref{c19} and \eqref{c20}, we have
\be\label{c21}\ba{rl}\ld_2^2\int_{B_r}I_{\{\psi=0\}} dX
\geq\f{c}{r^2}\int_{B_r} |\phi-\psi|^2
dX
\geq\f{c}{r^2}\int_{B_{\kappa r}} |\phi-\psi|^2
dX
\geq&c_\kappa\ld_1^2 r^2,\ea\ee for small $r>0$ and small $\kappa>0$. This gives the right-hand side of
\eqref{c18}.

\end{proof}

\section{Regularity of the free boundary}

Based on the Lipschitz continuity and non-degeneracy lemma of the minimizer $\psi$ in Section 2, we will establish the regularity of the free boundary in this section.

\subsection{Measure estimate of the free boundary} In this subsection, the main objective is to show that the set $\O\cap\{\psi>0\}$ is finite perimeter locally in $\O$.

Set $\mu=\Delta\psi+f(\psi)$ and $\mu_0=\Delta\psi+f(\psi)I_{\{\psi>0\}}$. First, we will show that $\mu_0$ is a Radon measure supported on the free boundary
$\Gamma$.

\begin{lemma}\label{lb90} $\mu_0=\Delta\psi+f(\psi)I_{\{\psi>0\}}$ is a positive Radon measure with support in $\O\cap\p\{\psi>0\}$. Moreover, $\mu=\Delta\psi+f(\psi)$ is a Radon measure supported on
$\O\cap\{\psi=0\}$ and its singular point is contained in the free boundary.

\end{lemma}

\begin{proof} For any $\xi\in C_0^{\infty}(\O)$ with $\xi\geq0$, denote $k(t)=\max\{\min\{2-t,1\},0\}$ for any $t>0$. Then we have
\be\label{cc23}\ba{rl}
 \int_{\O}\g\psi\cdot\g(\xi k(t\psi)) dX
 =&\int_{\O}\nabla\psi\cdot\nabla\xi dX+\int_{\O\cap\{\psi\geq\f1{t}\}}\nabla\psi\cdot\nabla(\xi k(t\psi)-\xi)dX\\
  =&\int_{\O}\nabla\psi\cdot\nabla\xi dX -\int_{\O\cap\{\psi>0\}}f(\psi)\xi dX\\
  &+\int_{\O\cap\{\f1t\leq\psi\leq\f2{t}\}}(2-\psi t)f(\psi)\xi dX+\int_{\O\cap\{\psi\leq\f2{t}\}}f(\psi)\xi dX\\
  \geq&\int_{\O}\nabla\psi\cdot\nabla\xi dX-\int_{\O\cap\{\psi>0\}}f(\psi)\xi dX,\ea
 \ee for sufficiently large $t>0$, where we have used the fact that $f(\psi)\geq\f{f(0)}2\geq0$ in $\O\cap\left\{\psi\leq\f2t\right\}$ for large $t>0$.

On the other hand, one has
\be\label{cc24}\ba{rl}
 \int_{\O}\g\psi\cdot\g(\xi k(t\psi)) dX
 =&\int_{\O\cap\{\f1t\leq\psi\leq\f2{t}\}}(2-\psi t)\nabla\psi\cdot\nabla\xi-\xi t|\g\psi|^2dX\\
  &+\int_{\O\cap\{0<\psi\leq\f1{t}\}}\nabla\psi\cdot\nabla\xi dX\\
  \leq&\int_{\O\cap\{0<\psi\leq\f2{t}\}}|\g\psi||\g\xi|dX.\ea
 \ee
It follows from \eqref{cc23} and \eqref{cc24} that
\be\label{cc25}\ba{rl}
 \int_{\O}\nabla\psi\cdot\nabla\xi dX-\int_{\O\cap\{\psi>0\}}f(\psi)\xi dX
 \leq\int_{\O\cap\{0<\psi\leq\f2{t}\}}|\g\psi||\g\xi|dX.\ea
 \ee
Taking $t\rightarrow+\infty$ in \eqref{cc25}, we conclude that $\Delta\psi+f(\psi)I_{\{\psi>0\}}\geq0$ in the sense of distributions. Consequently, there exists a positive Radon
measure $\mu_0$ supported on $\O\cap\p\{\psi>0\}$, such that $\mu_0=\Delta\psi+f(\psi)I_{\{\psi>0\}}$.

Recalling $f(0)\geq0$, there exists a positive Radon
measure $\mu$ supported on $\O\cap\{\psi=0\}$, such that $\mu=\mu_0+f(0)I_{\{\psi=0\}}=\Delta\psi+f(\psi)$.

\end{proof}

Next, we will give the estimate of the Radon measure $\mu$.
\begin{lemma}\label{lb9} Let $G$ be a compact subset of $\O$.
There exist some positive constants $r_0, c$ and $C$, such that for any disc
$B_r=B_r(X_0)\subset G$ with $X_0\in\Gamma$ and $r<r_0$,
\be\label{c30}cr\leq\int_{B_r} d\mu\leq Cr.\ee Furthermore,
\be\label{c22}cr\leq\int_{B_r\cap\p\{\psi>0\}} d\mu=\int_{B_r} d\mu_0\leq Cr.\ee

\end{lemma}

\begin{proof} For any $\e>0$, taking a test function
$d_{\e,B_r}(X)=\min\left\{\f{\text{dist}(X,\mathbb{R}^2\setminus
B_r)}{\e},1\right\}$ for a set $B_r\subset G$, it is easy to check that
$d_{\e,B_r}(X)$ converges to $I_{B_r}$ as $\e\rightarrow 0$ and
\be\label{c23}\ba{rl}
 \int_{ \O} d_{\e,B_r} d\mu=&\int_{\O}\nabla\psi\cdot\nabla d_{\e,B_r}+f(\psi)d_{\e,B_r}dX.\ea
 \ee

Taking $\e\rightarrow0$ in \eqref{c23}, it follows from Lemma
\ref{lb5} that \be\label{c24}\int_{B_r} d\mu=\int_{\p
B_r}\nabla\psi\cdot\nu d\mathcal{H}^1+\int_{B_r}f(\psi) dX\leq
Cr+\L\pi r^2\leq Cr,\ee where $\mathcal{H}^1$ is the one-dimensional Hausdorff measure on $\mathbb{R}^2$ and $\nu$ is the outer normal vector.

Let $Y\in B_r$ and $G_Y(X)>0$ be the Green
function for Laplacian in $B_r$ with pole $Y$. If $\psi(Y)>0$, then
the pole $Y$ is outside the support of the Radon measure $\mu$, and thus
\be\label{c25}\int_{B_r}G_Y(X)
d\mu=-\psi(Y)-\int_{B_r}f(\psi)G_Y(X)dX+\int_{\p
B_r}\psi\p_{-\nu}G_Y(X) d\mathcal{H}^1.\ee

Thanks to the non-degeneracy Lemma \ref{lb7}, there exists a point
$Y\in\p B_{\kappa r}$, such that $\psi(Y)\geq c^*_\kappa\ld r>0$ for
all small $\kappa>0$. Recalling that $\psi$ is Lipschitz continuous, we have
\be\label{c26}\psi(Y)=\psi(Y)-\psi(X_0)\leq C|Y-X_0|=C\kappa r\ \ \text{and}\ \
\psi>0\ \ \text{in}\ \ B_{\mathfrak{c}(\kappa)r}(Y),\ee for some small constant
$\mathfrak{c}(\kappa)$. Then it follows from \eqref{c25} and \eqref{c26} that
\be\label{c27}\int_{B_r}G_Y(X) d\mu\geq -C\kappa r-C\L
r^2+c(1-\kappa)\fint_{\p B_r}\psi d\mathcal{H}^1\geq c_0 r,\ \
(c_0>0),\ee provided that $\kappa$ and $r$ are small enough.

On the other hand, we have \be\label{c28}\int_{B_r}G_Y(X)
d\mu=\int_{B_r\cap\{\psi=0\}}G_Y(X)
d\mu\leq\left(\sup_{B_r\cap\{\psi=0\}}G_Y(X)\right)\int_{B_r}
d\mu\leq C_k\int_{B_r} d\mu,\ee where we have used the fact that
$\text{dist}(Y, B_r\cap\{\psi=0\})\geq \mathfrak{c}(\kappa)r$.
Combining \eqref{c27} and \eqref{c28}, one has
$$\int_{B_r} d\mu\geq cr,$$ which together with \eqref{c24} gives
\eqref{c30}.

It is clear that
$$\int_{B_r\cap\p\{\psi>0\}} d\mu\leq \int_{B_r} d\mu\leq
Cr,$$ this gives the right-hand side of \eqref{c22}. On the other
hand, we have
$$\ba{rl}\int_{B_r} d\mu=&\int_{B_r\cap\p\{\psi>0\}} d\mu+\int_{B_r\cap\{\psi=0\}} d\mu\\
\leq&\int_{B_r\cap\p\{\psi>0\}}
d\mu+\int_{B_r\cap\{\psi=0\}} f(0) dX\\
\leq& \int_{B_r\cap\p\{\psi>0\}} d\mu+\L\pi r^2,\ea$$ which implies
that
$$\int_{B_r\cap\p\{\psi>0\}} d\mu\geq\int_{B_r} d\mu-\L\pi r^2\geq cr-\L \pi r^2\geq cr\ \ \text{for small
$r$},$$ this gives the left-hand side of \eqref{c22}.

\end{proof}

We next introduce the representation theorem as follows, which implies that $\O_+=\O\cap\{\psi>0\}$ has finite perimeter.

\begin{proposition}\label{lc1}(Representation Theorem) Let $\psi$ be a minimizer, there holds that:

(1). $\mathcal{H}^1(D\cap\p\{\psi>0\})<+\infty$  for any compact
subset $D$ of $\O$.

(2). There exists a Borel function $\mathcal{B}_{\psi}$, such that
$$\Delta\psi+f(\psi)I_{\{\psi>0\}}=\mathcal{B}_{\psi}\mathcal{H}^1\lfloor_{\p\{\psi>0\}},$$ that is, for any $\xi\in C_0^\infty(\O)$,
$$-\int_{\O}\nabla\psi\cdot\nabla\xi dX+\int_{\O\cap\{\psi>0\}}f(\psi)\xi dX=
 \int_{\O\cap\p\{\psi>0\}}\mathcal{B}_{\psi}\xi d\mathcal{H}^1.$$

(3). For any compact subset $D$ of $\O$, $B_r(X_0)\subset D$ with
$X_0\in\Gamma$ and small $r>0$, $$0<c\leq \mathcal{B}_{\psi}(X)\leq
C<+\infty\ \ \text{and}\ \
cr\leq\mathcal{H}^1(B_r(X_0)\cap\p\{\psi>0\})\leq Cr,$$ for some
positive constants $c,C$ independent of $r$ and $X_0$.

\end{proposition}

\begin{proof}  Thanks to the fact \eqref{c22}, and along the similar arguments in the proof of
Theorem 4.5 in \cite{AC1}, we can obtain that
$$\f1C\mu_0(G)\leq \mathcal{H}^1(G)\leq C\mu_0(G)\ \ \text{for any compact subset $G$ of
$D\cap\p\{\psi>0\}$},$$ Denote $\mu_0=\Delta\psi+f(\psi)I_{\{\psi>0\}}$, this gives the assertion (1).

Therefore, $\mu_0$ is absolutely continuous with respect to
$\mathcal{H}^1\lfloor_{\p\{\psi>0\}}$. It follows from Theorem 2.5.8
in \cite{FH} that there exists a Radon-Nikodym derivative
$\mathcal{B}_\psi=\f{d\mu_0}{d(\mathcal{H}^1\lfloor_{\p\{\psi>0\}})}$,
namely,
$$\mu_0(G)=\int_{G}\mathcal{B}_\psi
d(\mathcal{H}^1\lfloor_{\p\{\psi>0\}})\ \ \text{for any
$G\subset\O$},$$ which together with \eqref{c22} gives the assertion
(3).

 Then we have
$$\int_{\O\cap\{\psi=0\}}\xi d\mu+\int_{\O}\xi d\mu_0=\int_{\O}\xi d\mu=-\int_{\O}\nabla\psi\cdot\g\xi dX+\int_{\O}f(\psi)\xi
dX,$$ for any $\xi\in C_0^\infty(\O)$, which gives that
$$-\int_{\O}\nabla\psi\cdot\nabla\xi dX+\int_{\O\cap\{\psi>0\}}f(\psi)\xi dX=
 \int_{\O\cap\p\{\psi>0\}}B_\psi\xi d\mathcal{H}^1,$$ due to $\int_{\O\cap\{\psi=0\}}\xi d\mu=\int_{\O\cap\{\psi=0\}}f(0)\xi
 dX$. This yields the assertion (2).

\end{proof}

By virtue of (1) in Proposition \ref{lc1}, it follows from Theorem 1 in
$\S5.11$ in \cite{EV2} that the set $\O_+=\O\cap\{\psi>0\}$ has finite
perimeter, that is, $\mathcal{V}_{\psi}=-\nabla I_{\O_+}$ is Borel
measure and the total variation $|\mathcal{V}_{\psi}|$ is a Radon
measure. Denote the reduced boundary of the set
$\O\cap\{\psi>0\}$ as
$$\p_{red}\{\psi>0\}=\{X\in\O\cap\p\{\psi>0\}\mid |\nu_{\psi}|=1\},$$ where $\nu_{\psi}$ is the unique unit vector with
$$\int_{B_r(X)}|I_{\{\psi>0\}}-I_{\{Y\mid (Y-X)\cdot\nu_{\psi}<0\}}|dY=o(r^2)\ \ \text{as $r\rightarrow0$, }$$if such a vector
exists,
 and $\nu_{\psi}=0$ otherwise. Furthermore, it follows from
 Theorem 4.5.6 in \cite{FH} that
$$\mathcal{V}_{\psi}=\nu_{\psi}\mathcal{H}^1\lfloor\ \p_{red}\{\psi>0\}.$$

\subsection{ Blow-up limits} In this subsection, we study some properties of the so-called blow-up limits.

For any $X_0\in\O$, take two sequences $\{X_n\}$ and $\{\rho_n\}$
with $X_n\in\O$ and $\rho_n>0$, such that $\psi(X_n)=0$,
$X_n\rightarrow X_0$ and $\rho_n\rightarrow0$ as
$n\rightarrow0$. We call the sequence of functions defined by
$$\psi_n(X)=\f{\psi(X_n+\rho_nX)}{\rho_n} \ \ \text{and}\ \ \ld_n(X)=\ld(X_n+\rho_n X),$$
 as the blow-up sequence with respect to $B_{\rho_n}(X_n)$, $X\in\mathbb{R}^2$. It follows
 from Lemma \ref{lb5} that $|\g\psi_n(X)|\leq C$ in any compact set of
 $\mathbb{R}^2$, provided that $n$ is large enough. Since $\psi_n(0)=0$, there exists a blow-up limit
$\psi_0:\mathbb{R}^2\rightarrow\mathbb{R}$, such that for a
subsequence $\{\psi_n\}$, \be\label{c020}\psi_n\rightarrow\psi_0\ \
\text{in $C_{loc}^{0,\alpha}(\mathbb{R}^2)$ for any
$\alpha\in(0,1)$},\ee and $$\nabla\psi_n\rightarrow\nabla\psi_0\ \
\text{weakly star in $L_{loc}^{\infty}(\mathbb{R}^2)$.}$$ Recalling
the definition of the {\it Hausdorff distance $d(E,F)$} between two
sets $E$ and $F$ as the infimum of the numbers $\e$, such that
$$F\subset\bigcup_{X\in E}B_\e(X)\ \ \text{and}\ \ E\subset\bigcup_{X\in F}B_\e(X).$$

Next, we will give the several convergences of the blow-up sequence, the idea borrows from the similar arguments for Laplace equation in Lemma 3.6 in Chapter 3 in \cite{FA1}.
\begin{lemma}\label{lc2} The following properties hold:

(1) $\p{\{\psi_n>0\}}\rightarrow\p{\{\psi_0>0\}}\ \ \text{locally in
Hausdorff distance.}$

(2) $I_{\{\psi_n>0\}}\rightarrow I_{\{\psi_0>0\}}\ \ \text{ in
$L_{loc}^1(\mathbb{R}^2)$.}$

(3) If $X_n\in\p\{\psi>0\}$, then $0\in\p\{\psi_0>0\}$.

(4) $\nabla\psi_n\rightarrow\nabla\psi_0\ \ \text{a.e. in
$\mathbb{R}^2$}.$

(5) If $\psi(X_n)=0$ and $X_n\rightarrow X_0\in \O$, then every blow-up limit $\psi_0$ with respect to $B_{\rho_n}(X_n)$ is an absolute
minimum for the functional $J_0(\phi)$ in $B_r=B_r(0)$ for any $r>0$, namely,
\be\label{c42}J_0(\psi_0)=\min J_0(\phi)\ \ \text{for any $\phi\in
H^1(B_r)$ and $\phi=\psi_0$ on $\p B_r$},\ee where
$J_0(\phi)=\int_{B_r}|\g\phi|^2+\ld^2(X_0)I_{\{\phi>0\}}dX$.
\end{lemma}
\begin{proof} (1).
For any $X_0\in \O$, if $X_0\notin\p\{\psi_0>0\}$, then there exists
a small $r>0$ such that $B_r(X_0)\cap\p\{\psi_0>0\}=\varnothing$
with $B_r(X_0)\subset\O$. We next claim that
\be\label{c021}\text{$B_{\f
r{4}}(X_0)\cap\p\{\psi_n>0\}=\varnothing$ for sufficiently large
$n$.}\ee In fact, it follows from \eqref{c020} that the claim
\eqref{c021} is true, if $\psi_0>0$ in $B_r(X_0)$.

If $\psi_0\equiv 0$ in $B_r(X_0)$, for any fixed small $\e>0$, there exists a $N=N(\e)$, such that $\psi_n<\e$ in
$B_{\f r2}(X_0)$ for any $n>N$, and one has
$$\f2r\fint_{\partial B_{\f r2}(X_0)}\psi_n dS<\f{2\e}r\leq c_{\f12}^*\ld_1 \ \text{for sufficiently large $n$},$$
which together with the non-degeneracy Lemma \ref{lb7} implies that $\psi_n\equiv 0 \ \text{in} \ B_{\f r4}(X_0)$ for $n>N$, this gives the claim \eqref{c021}.

Reversely, for any $X_0\in \O$, if $X_0\notin\p\{\psi_n>0\}$ for a
subsequence $\{\psi_n\}$, then
$B_r(X_0)\cap\p\{\psi_n>0\}=\varnothing$ for small $r>0$. Next, we
claim that \be\label{c022}\text{$B_{\f
r4}(X_0)\cap\p\{\psi_0>0\}=\varnothing$}.\ee If $\psi_n>0$ in
$B_r(X_0)$, we have
$$\Delta\psi_n+\rho_nf(\rho_n\psi_n)=0 \ \ \text{in}\ \ B_r(X_0),$$ which implies that $$\Delta\psi_0= 0 \ \text{in} \
B_{\f r2}(X_0),\ \psi_0\geq 0\ \text{in} \ B_{\f r2}(X_0).$$ The
strong maximum principle yields that
$$\text{either}\ \ \psi_0\equiv 0\ \ \text{or}\ \ \psi_0>0\ \text{ in}\ B_{\f r2}(X_0),$$
which gives the claim \eqref{c022}.

It is easy to check that the claim \eqref{c022} holds, in the case of
$\psi_n\equiv 0$ in $B_r(X_0)$.

 Hence, we obtain the convergence of the free boundary in the Hausdorff distance.

(2). For any $Y_0\in\p\{\psi_0>0\}$, it follows from the results in
Step 1 that there exists a sequence $\{Y_n\}$ with
$Y_n\in\p\{\psi_n>0\}$, such that $Y_n\rightarrow Y_0$. Since
$\rho_nY_n+X_n\in\p\{\psi>0\}$, by using Lemma \ref{lb6} and Lemma
\ref{lb7} for $\psi_n$, we have
\be\label{c023}c_{\f12}^*\ld_1\leq\f1r\fint_{\partial
B_r(Y_n)}\psi_n(X)dS_X=\f1{r\rho_n}\fint_{\partial
B_{r\rho_n}(\rho_nY_n+X_n)}\psi(Z)dS_Z\leq C^*(\ld_2+\L),\ee for any
$r>0$, provided that $n$ is sufficiently large. Then taking
$n\rightarrow+\infty$ in \eqref{c023} gives that
\be\label{c027}c_{\f12}^*\ld_1\leq\f1r\fint_{\partial
B_r(Y_0)}\psi_0dS\leq C^*(\ld_2+\L)\ \ \text{for}\ \
Y_0\in\p\{\psi_0>0\},\ee which together with Theorem 4.5 in
\cite{AC1} imply that
\be\label{c024}\mathcal{H}^1(\p\{\psi_0>0\}\cap D)<+\infty\ \
\text{for any compact subset $D$ of $\mathbb{R}^2$}.\ee Here,
$\mathcal{H}^1$ is the one-dimensional Hausdorff measure on
$\mathbb{R}^2$. Consequently,
$$\mathcal{L}^2(\p\{\psi_0>0\}\cap D)=0,$$ where
$\mathcal{L}^2$ is the two-dimensional Lebesgue measure on
$\mathbb{R}^2$.

Let $O_{\e_n}$ be an $\e_n$-neighborhood of $\p\{\psi_0>0\}$, such
that \be\label{c025}\p\{\psi_n>0\}\subset O_{\e_n}\ \ \text{and}\ \
\mathcal{L}^2(D\cap O_{\e_n})\rightarrow 0\ \quad \text{as}\
\e_n\rightarrow 0.\ee

Hence, it follows from the results in Step 1 that
$$\int_{D}\left|I_{\{\psi_n>0\}}-I_{\{\psi_0>0\}}\right|dX\leq\int_{D\cap
O_{\e_n}}1dX=\mathcal{L}^2(D\cap O_{\e_n}),$$  for sufficiently
large $n$, which together with \eqref{c025} gives that
$$\text{$I_{\{\psi_n>0\}}\rightarrow I_{\{\psi_0>0\}}\ $ in
$\ L^1(D)$.}$$

(3). If $X_n$ is a free boundary point of the minimizer $\psi$, it follows from
Lemma \ref{lb6} and Lemma \ref{lb7} that
\be\label{c026}c_{\f12}^*\ld_1\leq\f1r\fint_{\partial
B_r(0)}\psi_n(X)dS_X=\f1{r\rho_n}\fint_{\partial
B_{r\rho_n}(X_n)}\psi(Z)dS_Z\leq C^*(\ld_2+\L),\ee for any $r>0$,
provided that $n$ is sufficiently large. Then taking
$n\rightarrow+\infty$ in \eqref{c026}, one has
$$c_{\f12}^*\ld_1\leq\f1r\fint_{\partial B_r(0)}\psi_0dS\leq
C^*(\ld_2+\L)\ \ \text{for any $r>0$},$$ which gives that
$0\in\p\{\psi_0>0\}$.

(4).  Let $D$ be any compact subset of $\{\psi_0>0\}$, thanks to the
standard elliptic estimates for $\psi_n$, one has
\be\label{c028}\nabla\psi_n\rightarrow\nabla\psi_0\ \
\text{uniformly in $D$}.\ee Next, we claim that
\be\label{c029}\nabla\psi_n\rightarrow\nabla\psi_0\ \ \text{a.e. in
$\{\psi_0=0\}$}.\ee

Since $\{\psi_0=0\}$ is $\mathcal{L}^2$-measurable, it follows from
Corollary 3 of Section 1.7 in \cite{EV2} that
$$\lim_{r\rightarrow 0}\f{\mathcal{L}^2(B_r(X)\cap\{\psi_0=0\})}{\mathcal{L}^2(B_r(X))}=1\ \ \text{for $\mathcal{L}^2$ a.e. $X\in\{\psi_0=0\}$}.$$
Denote $$\mathcal{S}=\left\{X\in\{\psi_0=0\}\mid \lim_{r\rightarrow
0}\f{\mathcal{L}^2(B_r(X)\cap\{\psi_0=0\})}{\mathcal{L}^2(B_r(X))}=1\right\}.$$
We next show that \be\label{c290}\text{$\psi_0(X_0+X)=o(|X|)\ $ for
any $X_0\in\mathcal{S}$}.\ee In fact, suppose not, we assume that $\psi_0(Y)>kr$ for some
$Y\in B_r(X_0)$ with $r\rightarrow 0$ and $k>0$. With the aid of
\eqref{c027}, it follows from Theorem 4.3 and Remark 4.4 in
\cite{AC1} that $\psi_0$ is Lipschitz continuous, which implies that
$$\psi_0(X)>\f k2 r\ \ \text{in $B_{\e kr}(Y)$ for some small $\e>0$}.$$
This gives that $\{\psi_0>0\}$ has positive density at $X_0$, which
contradicts to $X_0\in\mathcal{S}$. Thus, we obtain the fact \eqref{c290}.

With the aid of \eqref{c020} and \eqref{c290}, for any $\e>0$, we
have
$$\f{\psi_n}{r}<\e \ \ \text{in $B_r(X_0)$ for small $r$},$$
provided that $n$ is sufficiently large, that is $n>N(\e,r)$. It
follows from the non-degeneracy Lemma \ref{lb7} that $\psi_n\equiv
0$ in $B_{\f r2}(X_0)$, which implies that $\psi_0\equiv 0$ in
$B_{\f r{2}}(X_0)$, and thus $\mathcal{S}$ is open.
Furthermore, one has
$$\text{$\psi_n\equiv \psi_0$ in any compact subset of $\mathcal{S}$, provided that $n$ is sufficiently
large.}$$ This completes the proof of the claim \eqref{c029}.

Since $\mathcal{L}^2(\p\{\psi_0>0\})=0$, it follows from
\eqref{c028} and \eqref{c029} that $\nabla\psi_n\rightarrow
\nabla\psi_0$ a.e. in $\mathbb{R}^2$.

(5). For any $\phi\in H^1(B_r)$ and $\phi=\psi_0$ on $\p B_r$ with $B_r=B_r(0)$, it
suffices to show that \be\label{c030}J_0(\psi_0)\leq J_0(\phi).\ee
Taking $\eta_\e(X)=\min\left\{\f1\e dist(X,\mathbb{R}^2\setminus
B_r),1\right\}$, it is easy to see that $\eta_\e\in C_0^{0,1}(B_r)$
and $0\leq\eta_\e\leq1$. Set
$$\phi_n=\phi+(1-\eta_\e)(\psi_n-\psi_0).$$ It is easy to check that
$$\phi_n=\psi_n\ \text{outside}\ \ B_r, \ \ \text{and}\ \ \{\phi_n>0\}\subset\{\phi>0\}\cup\{\eta_\e<1\}.$$
Then we have \be\label{c031}\ba{rl}&\int_{B_r}|\nabla
\psi_n|^2+F(\rho_n\psi_n)+\ld_n^2I_{\{\psi_n>0\}}
dX\\
\leq&\int_{B_r}|\nabla
\phi_n|^2+F(\rho_n\phi_n)+\ld_n^2I_{\{\phi_n>0\}}
dX\\
\leq&\int_{B_r}|\nabla
\phi_n|^2+F(\rho_n\phi_n)+\ld_n^2I_{\{\phi_n>0\}}
dX+\int_{B_r}\ld_n^2I_{\{\eta_\e<1\}} dX.\ea\ee

By virtue of the results in the statements (2) and (4), taking
$n\rightarrow+\infty$ in \eqref{c031} gives that
\be\label{c032}\ba{rl}&\int_{B_r}|\nabla\psi_0|^2+\ld^2(X_0)I_{\{\psi_0>0\}}
dX\\
\leq&\int_{B_r}|\nabla\phi|^2+\ld^2(X_0)I_{\{\phi>0\}}
dX+\ld^2(X_0)\int_{B_r}I_{\{\eta_\e<1\}} dX.\ea\ee

Taking $\e\rightarrow 0$ in \eqref{c032}, we complete the proof of
the claim \eqref{c030}.

\end{proof}

\subsection{Linear growth near the free boundary}
In this subsection, we will obtain the gradient estimate of $\psi$
near the free boundary, and show that $\psi(X)$ should grow
linearly away from the free boundary. Namely,
\be\label{cc0}\psi(X+X_0)=\ld(X_0)\max\{-X\cdot\nu(X_0),0\}+o(|X|),\ee
for $X\rightarrow 0$, where $X_0\in\O\cap\p\{\psi>0\}$ and $\nu(X_0)$ is the
unit vector.

\begin{lemma}\label{lc5} For any
compact subset $D$ of $\O$, there exist some positive constants $C$ and
$0<\gamma<1$ depending only on $\L$, $D$ and $\O$, such that
for any disc 
$B_r(X_0)\subset B_R(X_0)\subset D$ with $X_0\in\O\cap\p\{\psi>0\}$, then
\be\label{c46}\sup_{X\in B_r(X_0)} |\nabla\psi(X)|\leq \sup_{X\in
B_R(X_0)}\ld(X)+C \left(\f rR\right)^\gamma,\ee  Furthermore, there exists a $\alpha_0\in(0,1)$, such that
\be\label{c460} |\nabla\psi(X)|\leq \ld(X)+C r^{\alpha_0},\ \
\alpha_0\in(0,1),\ee for any disc $B_r(X)\subset D$ touching the free
boundary with small $r>0$.
\end{lemma}

\begin{proof} Denote $B_r=B_r(X_0)$ for simplicity, and one has
$$\Delta\psi+f(\psi)=0\ \ \text{in}\ \ \O\cap\{\psi>0\}.$$

{\bf Step 1.} In this step, we will show that
\be\label{c47}\limsup_{X\rightarrow X_0,
\psi(X)>0}|\g\psi(X)|=\ld(X_0),\ \ \ X_0\in\O\cap\p\{\psi>0\}.\ee

Denote $\kappa=\limsup_{X\rightarrow X_0, \psi(X)>0}|\g\psi(X)|$, it
suffices to prove that $\kappa=\ld(X_0)$. In view of definition of
$\kappa$, there exists a sequence $\{Z_n\}$ with $\psi(Z_n)>0$ and
$|\g\psi(Z_n)|\rightarrow \kappa$. Let $Y_n\in\Gamma$ be the
nearest point to $Z_n$ and denote $\rho_n=|Y_n-Z_n|$. Let $\psi_0$ be a
blow-up limit of a sequence $\f{\psi(Y_n+\rho_nX)}{\rho_n}$ with respect to $B_{\rho_n}(Y_n)$, without loss of
generality, we assume that $$\f{Z_n-Y_n}{\rho_n}\rightarrow-e_2,\ \ \ e_2=(0,1).$$
By virtue of the statement (5) in Lemma \ref{lc2}, we have that $\psi_0$ is a
minimizer. It follows from the similar arguments in Lemma 2.2 and
Lemma 2.4 in \cite{AC1} that $\psi_0$ is subharmonic in $\mathbb{R}^2$ and
$\Delta\psi_0=0$ in $\{\psi_0>0\}$. Furthermore, $\psi_0(0)=0$ and
\be\label{c48}|\g\psi_0|\leq \kappa\ \ \text{in $\{\psi_0>0\}$,
$|\g\psi_0(-e_2)|=\kappa$ and $B_1(-e_2)\subset\{\psi_0>0\}$},\ee
this gives that $\kappa>0$. Define $\phi_0=\f{\p\psi_0}{\p\nu_0}$, where
$\nu_0=-\f{\g\psi_0(-e_2)}{|\g\psi_0(-e_2)|}$. It is easy to check
that $\phi_0$ is harmonic in $\{\psi_0>0\}$. Moreover, it follows from
\eqref{c48} that
$$\phi_0=-\f{\g\psi_0\cdot\g\psi_0(-e_2)}{|\g\psi_0(-e_2)|}\geq-\kappa\ \text{in $B_1(-e_2)$ and $\phi_0(-e_2)=-\kappa$}.$$
The strong maximum principle gives that
$$\phi_0\equiv-\kappa\ \ \text{in}\ \ B_1(-e_2),$$ which together with $\psi_0(0)=0$ implies that $$\psi_0(X)=-\kappa X\cdot\nu_0\ \ \text{in}\ \ B_1(-e_2).$$
Since $\psi_0>0$ in $B_1(-e_2)$, we have that $\nu_0=e_2$. Due to the uniqueness of the
Cauchy problem for the Laplace equation, one has
$$\psi_0=-\kappa y\ \ \text{in}\ \ \{y\leq 0\}.$$

Next, we claim that \be\label{c49}\psi_0=0\ \ \text{in some strip}\
\ \{0<y<\e_0\}.\ee Suppose that the claim \eqref{c49} is not true.
Define $$l=\limsup_{y\downarrow 0,\psi_0(x,y)>0}\f{\p\psi_0(x,y)}{\p
y}.$$ Since $\psi_0$ is a local minimizer, it follows from Corollary
3.3 in \cite{AC1} that $\psi_0$ is Lipschitz continuous, and thus
$l<\infty$. Suppose $l>0$, and let
$$\f{\p\psi_0(x_n,y_n)}{\p y}\rightarrow l\ \ \text{as}\ \ y_n\downarrow 0.$$

Choose a blow-up sequence $\f{\psi_0(x_n+y_nx,y_ny)}{y_n}$ with respect to $B_{y_n}(x_n,0)$, let $\psi_{00}$ be the blow-up limit. Using above arguments again, we have
$$\psi_{00}(x,y)=ly\ \ \text{in}\ \ \{y>0\},\ \ \text{and}\ \ \psi_{00}(x,y)=-\kappa y\ \ \text{in}\ \ \{y<0\}.$$
Since $\psi_{00}$ is a minimizer for \eqref{c42} and $(x,0)$ is a free
boundary point of $\psi_{00}$, we can show that the set $\{\psi_{00}=0\}$ has density
zero at any point $(x,0)$, which contradicts to Lemma 3.7 in
\cite{AC1}.

Therefore, we obtain that $l=0$ and $\psi_0(x,y)=o(y)$ as
$y\downarrow 0$. For any $\e>0$, one has
$$\f1r\fint_{\p B_r(X_0)}\psi_0dS<\e, \ \  \text{where $X_0=(x_0,y_0)$ and $r=y_0$},$$
provided that $y_0>0$ is small enough. It follows from the
non-degeneracy Lemma 3.4 in \cite{AC1} that $\psi_0=0$ in some strip
$\{0<y<\e_0\}$. Thus, the proof of the claim
\eqref{c49} is done.

Finally, noting that $\psi_0$ is a minimizer to the variational problem
\eqref{c42}, by means of Theorem 2.5 in \cite{AC1}, we can conclude
that $|\g\psi_0|=\ld(X_0)$ on the free boundary of $\psi_0$, and
thus $\ld(X_0)=\kappa$.

{\bf Step 2.} In this step, we will show that there exists a constant $C>0$, such that
\be\label{c460}\sup_{X\in B_r(X_0)} |\nabla\psi(X)|\leq \sup_{X\in
B_R(X_0)}\ld(X)+C \left(\f rR\right)^\gamma,\ee for any disc
$B_r(X_0)\subset B_R(X_0)\subset D$. Denote $B_r=B_r(X_0)$ and
$B_R=B_R(X_0)$ for simplicity.

 Define
$\mathcal{Q}_\e(Y)=\max\{|\nabla\psi(Y)|^2-\ld_0^2-\e,0\}$ for any
$\e>0$, where $\ld_0=\sup_{X\in B_R}\ld(X)$. It is easy to check
that
$$\Delta\mathcal{Q}_\e=2|D^2\psi|^2-2f'(\psi)|\g\psi|^2\geq 0\ \ \text{in $\O\cap\{\psi>0\}$}.$$
And thus $\mathcal{Q}_\e$ is subharmonic function in
$\O\cap\{\psi>0\}$. It follows from \eqref{c47} that
$\mathcal{Q}_\e=0$ in a small neighborhood of the free boundary
$\Gamma$. We extend $\mathcal{Q}_\e$ by $0$ and set
$$P_\e(r)=\sup_{Y\in B_r} \mathcal{Q}_\e(Y)\ \ \text{for any $r>0$}.$$ Then $P_\e(r)-\mathcal{Q}_\e$ is superharmonic in $B_r$.
Furthermore, $P_\e(r)-\mathcal{Q}_\e\geq 0$ in $B_r$ and
$P_\e(r)-\mathcal{Q}_\e=P_\e(r)$ in $B_r\cap\{\psi=0\}$. It follows
from Theorem 8.26 in \cite{GT} that
\be\label{c50}\ba{rl}\inf_{Y\in B_{\f r2}}(P_\e(r)-\mathcal{Q}_\e(Y))\geq& cr^{-2}\|P_\e(r)-\mathcal{Q}_\e\|_{L^1(B_r)}\\
\geq & cr^{-2}\|P_\e(r)-\mathcal{Q}_\e\|_{L^1(B_r\cap\{\psi=0\})}\\
\geq& cP_\e(r),\ea\ee where we have used the fact
$\mathcal{L}^2(B_r\cap\{\psi=0\})\geq cr^2$ in Lemma \ref{lb8}.
Taking $\e\rightarrow 0$ in \eqref{c50}, we have
$$\inf_{Y\in B_{\f r2}}(P_0(r)-\mathcal{Q}_0(Y))
\geq cP_0(r),\ \ \ 0<c<1,$$ which implies that $$\sup_{Y\in B_{\f
r2}}\mathcal{Q}_0(Y) \leq (1-c)P_0(r).$$ Thus we have $$P_0\left(\f
r2\right) \leq (1-c)P_0(r).$$

It follows from Lemma 8.23 in \cite{GT} that
$$P_0(r)
\leq C\left(\f r R\right)^\gamma P_0(R),$$ which gives that
$$\sup_{X\in B_r} |\nabla\psi(X)|\leq\ld_0+ C \left(\f rR\right)^\gamma,\ \ \gamma\in(0,1),$$ for small
$r>0$.

{\bf Step 3.} For any disc $B_r(X)$ touching the free boundary,
there exists a free boundary point $X_0\in B_r(X)$, such that
$B_r(X)\subset B_{2r}(X_0)$. Since $\ld(X)\in C^{0,\beta}(\O)$, the gradient estimate \eqref{c46} gives that
\be\label{c460}|\g\psi(X)|\leq\sup_{Y\in B_{2r}(X_0)}
|\nabla\psi(Y)|\leq \sup_{Y\in
B_{2\sqrt{r}}(X_0)}\ld(Y)+Cr^{\f\beta2}\leq \ld(X)+
Cr^{\alpha_0},\ee where $\alpha_0=\f12\min\{\gamma,\beta\}$,
provided that $r$ is small.

\end{proof}

To obtain the linear growth \eqref{cc0} of $\psi$ near the free
boundary, we next show that
\begin{lemma}\label{lc6} For any compact subset $D$ of $\O$ and disc 
$B_r(X_0)\subset D$ with $X_0\in\O\cap\p\{\psi>0\}$, then
\be\label{c510}\fint_{B_r(X_0)\cap\{\psi>0\}}
\max\{\ld^2(Y)-|\nabla\psi(Y)|^2,0\} dY\rightarrow0\ \ \text{as
$r\rightarrow0$}.\ee
\end{lemma}

\begin{proof} Without loss of generality, we assume that $X_0=0$, and denote $B_r=B_r(0)$ for simplicity. For any $\xi\in C_0^\infty(\O)$ with $\xi\geq 0$, define a function
$$\psi_\e=\max\{\psi-\e\xi,0\}\ \ \text{for any $\e>0$}.$$ Since $\psi_\e\in
K$, we have that $J(\psi)\leq J(\psi_\e)$, namely,
$$\ba{rl}0\geq&\int_{\O} |\nabla\psi|^2-|\nabla
\psi_\e|^2+F(\psi)-F(\psi_\e)
dX+\int_{\O}\ld^2(I_{\{\psi>0\}}-I_{\{\psi_\e>0\}})dX\\
\geq&\int_{\O}
-|\nabla\min\{\psi,\e\xi\}|^2+2\nabla\psi\cdot\nabla\min\{\psi,\e\xi\}-
F'(\psi)\min\{\psi,\e\xi\}\\
&-\L|\min\{\psi,\e\xi\}|^2
dX+\int_{\O \cap\{0<\psi\leq \e\xi\}}\ld^2 dX\\
=&-\int_{\O} |\nabla\min\{\psi,\e\xi\}|^2+\L|\min\{\psi,\e\xi\}|^2
dX+\int_{\O \cap\{0<\psi\leq \e\xi\}}\ld^2 dX.\ea$$ This gives
that \be\label{c51}\int_{\O \cap\{0<\psi\leq
\e\xi\}}\ld^2-|\nabla\psi|^2 dX\leq
\e^2\int_{\O\cap\{\psi>\e\xi\}}
|\nabla\xi|^2dX+\L\int_{\O}|\min\{\psi,\e\xi\}|^2 dX.\ee

Taking $B_r\subset B_\rho=B_\rho(0)\subset B_R$, the Lipschitz continuity of
$\psi$ gives that $$\text{ $\psi(X)=\psi(X)-\psi(0)\leq Cr$ in $B_r$.}$$ Taking
$\e=Cr$ and

$$\xi(X)=\left\{\ba{ll}0~~&\text{for}\ \ X\in \O\setminus B_\rho,\\
\f{\ln\rho-\ln|X|}{\ln\rho-\ln r}~~~&\text{for}~~X\in B_\rho\setminus B_r,\\
1~~~&\text{for}~B_r.\ea\right.$$ It follows from \eqref{c51} that
\be\label{c52}\ba{rl}\int_{B_\rho \cap\{0<\psi\leq
\e\xi\}}\ld^2-|\nabla\psi|^2 dX\leq&
\e^2\int_{B_\rho\cap\{\psi>\e\xi\}}
|\nabla\xi|^2dX+\L\int_{B_\rho}|\min\{\psi,\e\xi\}|^2 dX\\
\leq& Cr^2\int_{B_\rho\cap\{\psi>\e\xi\}} |\nabla\xi|^2dX+
C\rho^2r^2\\
\leq& \f{Cr^2}{\ln\rho-\ln r}+ C\rho^2r^2.\ea\ee Since $\psi\leq
Cr=\e\xi$ in $B_r$, it is easy to check that
\be\label{c53}\ba{rl}\int_{B_\rho \cap\{0<\psi\leq
\e\xi\}}\max\{\ld^2-|\nabla\psi|^2,0\} dX\geq&
\int_{B_r\cap\{0<\psi\leq\e\xi\}}\max\{\ld^2-|\nabla\psi|^2,0\}
dX\\
=&\int_{B_r\cap\{\psi>0\}}\max\{\ld^2-|\nabla\psi|^2,0\}
dX.\ea\ee By virtue of \eqref{c46},
\eqref{c52} and \eqref{c53}, we have
$$\ba{rl}&\int_{B_\rho \cap\{0<\psi\leq
\e\xi\}}\max\{\ld^2-|\nabla\psi|^2,0\} dX\\
=& \int_{B_\rho  \cap\{0<\psi\leq
\e\xi\}}\max\{|\nabla\psi|^2-\ld^2,0\} dX+ \int_{B_\rho  \cap\{0<\psi\leq
\e\xi\}}\ld^2-|\nabla\psi|^2 dX\\
\leq& \int_{B_\rho  \cap\{0<\psi\leq
\e\xi\}}\max\{|\nabla\psi|^2-\ld^2,0\} dX+ \f{Cr^2}{\ln\rho-\ln r}+
C\rho^2r^2\\
\leq&\int_{B_\rho\cap\{\psi>0\}}\max\{|\nabla\psi|^2-\ld^2,0\} dX+
\f{Cr^2}{\ln\rho-\ln r}+ C\rho^2r^2\\
\leq&C\rho^2\left(\left(\f{\rho}{R}\right)^\alpha+\ld_R\right)+\f{Cr^2}{\ln\rho-\ln
r}+ C\rho^2r^2,\ea $$ where $\ld_R=\sup_{X,Y\in
B_R(X_0)}|\ld(X)-\ld(Y)|$, which gives that
\be\label{c54}\ba{rl}\fint_{B_r
\cap\{\psi>0\}}\max\{\ld^2-|\nabla\psi|^2,0\} dX\leq
C\left(\f{\rho}{r}\right)^2\left(\left(\f{\rho}{R}\right)^\gamma+\ld_R\right)+\f{C}{\ln\f\rho
r}+ C\rho^2,\ea \ee

Taking $\rho=R^2$ and $r=R^2\left(\ld_R+R^\gamma\right)^{\f14}$, it
follows from \eqref{c54} that we have
$$\ba{rl}\fint_{B_r \cap\{\psi>0\}}\max\{\ld^2-|\nabla\psi|^2,0\} dX\leq C\left(\ld_R+R^\gamma\right)^{\f12}+\f{C}{\ln{\f1{\ld_R+R^\gamma}}}+CR^4,\ea
$$ for sufficiently small $r>0$, which gives \eqref{c510}.

\end{proof}

With the aid of Lemma \ref{lc6}, we have
\begin{lemma}\label{lc7} For any blow-up limit $\psi_0$ of $\psi$ at $X_0\in\Gamma$, $\psi_0$ is a half plane solution with slope $\ld(X_0)$.
Furthermore, the linear growth \eqref{cc0} holds.
\end{lemma}

\begin{proof} Let $\psi_n(X)=\f{\psi(X_0+\rho_nX)}{\rho_n}$ be a blow-up sequence with respect to $B_{\rho_nR}(X_0)$ for any $R>0$, and $\ld_n(X)=\ld(X_0+\rho_nX)$ with $X\in B_R(0)$, it follows from Lemma \ref{lc5} and
Lemma \ref{lc6} that
$$\sup_{X\in B_R(0)}|\g\psi_n(X)|=\sup_{Y\in B_{\rho_nR}(X_0)}|\g\psi(Y)|\leq \sup_{Y\in
B_{\sqrt{\rho_n}R}(X_0)}\ld(Y)+C\rho_n^{\f\gamma2}\rightarrow\ld(X_0),$$
and
$$\fint_{B_{R}(0)\cap\{\psi_n>0\}} \max\{\ld_n^2-|\nabla\psi_n|^2,0\}
dX=\fint_{B_{\rho_n R}(X_0)\cap\{\psi>0\}}
\max\{\ld^2-|\nabla\psi|^2,0\} dY\rightarrow0,$$ as
$\rho_n\rightarrow 0$. Those imply that for any blow-up limit
$\psi_0$, we have
$$|\nabla\psi_0|=\ld(X_0)\ \ \text{in}\ \ \{\psi_0>0\}.$$ On the other hand, $\psi_0$
is harmonic in $\{\psi_0>0\}$, thus $\nabla\psi_0$ must be invariant in
each connected component of $\{\phi_0>0\}$. In fact, applying the operator
$\p_x$ and $\p_y$ to $\f12|\g\psi_0|^2=\f{\ld^2(X_0)}2$, respectively, we have
$$\left\{\ba{rl}
\p_x\psi_0\p_{xx}\psi_0+\p_y\psi_0\p_{xy}\psi_0=0, \\
\p_x\psi_0\p_{xy}\psi_0+\p_y\psi_0\p_{yy}\psi_0=0. \ea \right.$$ This together with the fact $\Delta\psi_0=0$ gives that
$$\mathbb{T}\mathbf{X}=(0, 0, 0)^{tr},$$ where the vector
$\mathbf{X}=(\p_{xx}\psi_0, \p_{xy}\psi_0,
\p_{yy}\psi_0)^{tr}$ and matrix $\mathbb{T}$ as $$\mathbb{T}=\left(\begin{matrix} 1 &0 &1\\
\p_x\psi_0 &\p_y\psi_0 &0\\
0 &\p_x\psi_0 &\p_y\psi_0
\end{matrix}\right).$$ Furthermore, one has $$\det\mathbb{T}=|\g\psi_0|^2=\ld^2(X_0)>0,$$ which implies $$\mathbf{X}=(\p_{xx}\psi_0,
\p_{xy}\psi_0, \p_{yy}\psi_0)=(0,0,0).$$ Thus, $\psi_0(X)$ is a
linear function, there exist a unit vector $\nu_0$, and two numbers
$\gamma_1\geq0$ and $\gamma_2\geq0$, such that one has
\be\label{cc1}\psi_0(X)=\gamma_1\max\{-X\cdot\nu_0,0\}+\gamma_2\max\{X\cdot\nu_0,0\}.\ee
Since $\psi_0$ is a local minimizer, it follows from Lemma 3.7 in
\cite{AC1} that the set $\{\psi_0=0\}$ has a positive measure, this
gives that $\gamma_1=\ld(X_0)$ and $\gamma_2=0$. In view of
\eqref{cc1}, one has
$$\psi_0(X)=\ld(X_0)\max\{-X\cdot\nu_0,0\},$$ which implies the
linear growth \eqref{cc0}.
\end{proof}

\subsection{Flatness of the free boundary} In this subsection, we will
study the regularity of the free boundary $\O\cap\p\{\psi>0\}$, and obtain some flatness property of the free boundary point.

 First, we introduce the relevant flatness class of the free boundary
 point (See the definition 5.1 in \cite{AC1}).

\begin{definition}\label{def3}(Flat free boundary points) Let $0<\sigma_{+}, \sigma_{-}\leq1$ and
$\delta>0$. We say that $\psi$ is of {\it the flatness class}
$\mathcal{F}(\sigma_{+},\sigma_{-};\delta)$ in $B_\rho=B_\rho(X_0)$
with a unit vector $\nu$ (see Figure \ref{f2}) if

(i) $\psi$ is a minimizer to the variational problem $(P)$.

(ii) $X_0\in\Gamma$ and
$$\psi(X)=0\ \ \text{for}\ \
(X-X_0)\cdot\nu\geq\sigma_{+}\rho,$$ and
$$\psi(X)\geq-\ld(X_0)((X-X_0)\cdot\nu+\sigma_{-}\rho)\ \ \text{for}\ \ (X-X_0)\cdot\nu\leq-\sigma_{-}\rho.$$

(iii) $|\g\psi|\leq\ld(X_0)(1+\delta)$ in $B_\rho$ and $\sup_{X,Y\in
B_\rho}|\ld(X)-\ld(Y)|\leq\ld(X_0)\delta$.
\end{definition}
\begin{figure}[!h]
\includegraphics[width=100mm]{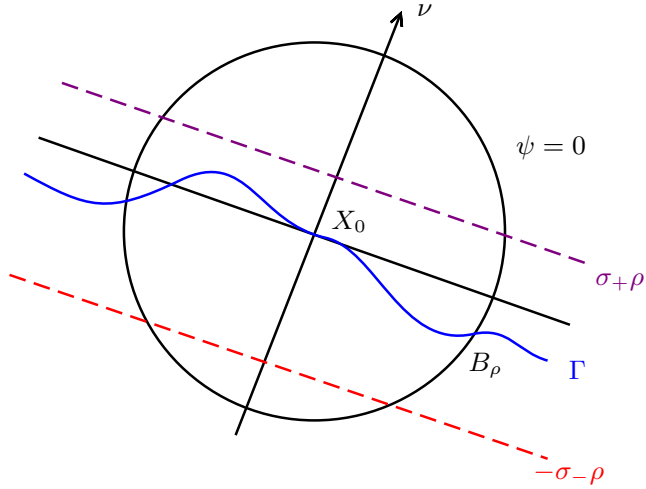}
\caption{Flatness free boundary point}\label{f2}
\end{figure}

We next show that the flatness from above implies the flatness from
below.

\begin{lemma}\label{lc8} There exists a positive constant $C=C(\L,c_0)$, such
that for any small $\sigma>0$, if $\psi\in
\mathcal{F}(\sigma,1;\delta)$ in $B_\rho(X_0)$ and $\delta\leq
c_0\sigma, \rho\leq c_0\sigma$, then $\psi\in
\mathcal{F}(2\sigma,C\sigma;\delta)$ in $B_{\f\rho2}(X_0)$ (see Figure \ref{f3}).
\end{lemma}
\begin{figure}[!h]
\includegraphics[width=110mm]{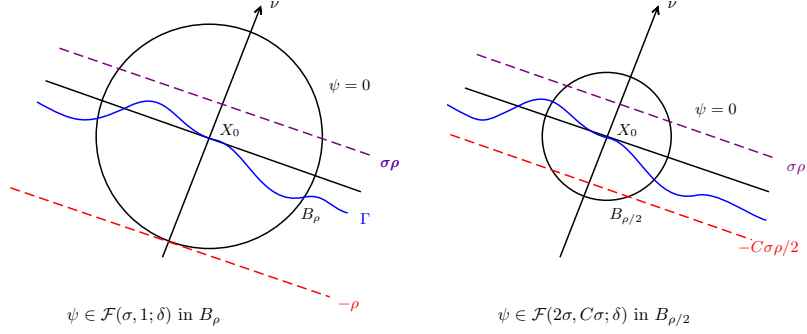}
\caption{Flatness from below}\label{f3}
\end{figure}

\begin{proof}
We divide two steps to show this lemma.

{\bf Step 1.} In this step, we will show that
\be\label{c55}\limsup_{X\rightarrow
X_0}\f{\psi(X)}{\text{dist}(X,B)}=\ld(X_0),\ \text{if $B$ is a disc
in $\{\psi=0\}$ touching $\p\{\psi>0\}$ at $X_0$}.\ee

 Denote $\limsup_{X\rightarrow X_0}\f{\psi(X)}{\text{dist}(X,B)}=\kappa$, it suffices to show that
$\kappa=\ld(X_0)$. Taking a sequence $\{Y_n\}$ with $Y_n\rightarrow
X_0$, such that $\psi(Y_n)>0$, $d_n=\text{dist}(Y_n,B)$ and
$$\f{\psi(Y_n)}{d_n}\rightarrow\kappa.$$ It follows from the non-degeneracy Lemma \ref{lb7} that $\kappa>0$.
Consider a blow-up sequence $\psi_n(X)=\f{\psi(X_n+d_nX)}{d_n}$ with
respect to $B_{d_n}(X_n)$, where $X_n\in\p B$ and $|X_n-Y_n|=d_n$.
Taking a subsequence $\{\psi_n\}$ with a blow-up limit $\psi_0\in
C^{0,1}(\mathbb{R}^2)$, such that
$$\f{X_n-Y_n}{d_n}\rightarrow\nu\ \ \text{and}\ \
\psi_n\rightarrow\psi_0\ \ \text{uniformly in any compact subset of
$\mathbb{R}^2$}.$$ It is easy to check that $\psi_0(-\nu)=\kappa$.
We next claim that
\be\label{c550}\psi_0(X)=\kappa\max\{-X\cdot\nu,0\}.\ee In fact, the
definition of $\psi_0$ gives that
$$\psi_0(X)\leq -\kappa X\cdot\nu\ \ \text{for}\ \ X\cdot\nu\leq 0,\ \ \text{and}\ \  \psi_0(X)=0\ \ \text{for}\ \ X\cdot\nu\geq 0.$$
Since $\psi_0$ is harmonic in $\{\psi_0>0\}$ and
$\psi_0(-\nu)=\kappa$, the strong maximum principle gives the claim
\eqref{c550}.

With the aid of the claim \eqref{c550}, it follows from Lemma \ref{lc7} that $\ld(X_0)=\kappa$.

{\bf Step 2.} Without loss of generality, we assume that $X_0=0$ and
$\nu=e_2=(0,1)$. Set $\psi_\rho(X)=\f{\psi(\rho X)}{\rho}$,
$\ld_\rho(X)=\ld(\rho X)$ and $\ld_0=\ld(0)$, one has
 \be\label{c56}\Delta\psi_\rho+\rho
f(\rho\psi_\rho)=0\ \ \text{in}\ \ \{\psi_\rho>0\},\ \text{and
$\psi_\rho\in F(\sigma,1;\delta)$ in $B_1$.}\ee
Let $$\eta(x)=\left\{\ba{ll}0~~&\text{for}\ \ |x|\geq \f13,\\
e^{-\f{9x^2}{1-9x^2}}~~&\text{for}\ \ |x|<\f13,\ea\right.$$ and
choose $s\geq 0$ be the maximum with the property
$$B_1\cap\{\psi_\rho>0\}\subset D_\sigma=\{X\in B_1\mid
y<\sigma-s\eta(x)\}.$$ Thus there exists a point $Z\in
B_{\f12}\cap\p D_\sigma\cap\p\{\psi_\rho>0\}$. It is easy to check
that $s\leq \sigma$, due to $0\in\Gamma$. Define a function $\phi(x,y)$ solving the problem
$$\left\{\ba{ll}&\Delta\phi+\rho f(0)=0~~~~\text{in}~~D_\sigma,\\
&\phi=0~~~~\text{on}~~\p D_\sigma\cap B_1,\ \ \ \ \
\phi=\ld_0(1+2\sigma)(\sigma-y)~~~~\text{on}~~\p D_\sigma\setminus
B_1.\ea\right.$$¡¡Noticing that $\ld_0(1+2\s)(\sigma-y)+\ld_0(1+\sigma)y\geq
0$ on $\p D_\sigma$, the maximum principle gives that
\be\label{c61}\phi(X)\geq -\ld_0(1+\sigma)y\ \ \text{in}\ \
D_\sigma.\ee

Next, we will show that \be\label{c57}\p_{-\nu}\phi(Z)\leq
\ld_0(1+C\sigma).\ee

To see this, define a function $\phi_1$ as
follows
$$\phi_1(x,y)=\f{\mu_2}{\mu_1}\left(1-e^{-\mu_1g(x,y)}\right) \ \ \text{in}\ \ D_\sigma,$$ where $g(x,y)=-y+\s-s\eta(x)$ and the constants $\mu_1$ and $\mu_2\geq
\ld_0$ depending on $\s$ will be determined later. It is easy to
check that \be\label{c59}1\leq|\g g(x,y)|\leq 1+C\s\ \ \text{and}\ \
|D^2g(x,y)|\leq C\s.\ee Therefore, it
follows from \eqref{c59} and $\rho\leq\s_0\s$ that
$$\Delta\phi_1+\rho f(0)=\mu_2 e^{-\mu_1g}(\Delta g-\mu_1|\g g|^2)+\rho f(0)\leq \mu_2 e^{-\mu_1g}(C\s-\mu_1)+\rho
f(0)<0,$$ provided that
$$\mu_1=C_1\s, \ \ \text{$C_1$ is sufficiently large and $\s$ is small}.$$
On the other hand, take $\mu_2=\ld_0(1+C_2\s)$, $C_2$ is sufficiently
large and $\s$ is small, one has
$$\phi_1(x,y)\geq\mu_2 g(x,y)(1-C\mu_1)\geq \ld_0(1+2\s)g(x,y)\ \ \text{for any $(x,y)\in
\p D_\s$},$$ which together with the maximum principle gives that
$$\phi_1(X)\geq\phi(X)\ \ \text{in $D_\s$}.$$ Recalling that $\phi_1(Z)=\phi(Z)=0$, we have
$$\p_{-\nu}\phi(Z)\leq\p_{-\nu}\phi_1(Z)\leq|\g\phi_1(Z)|=\mu_2|\g
g(Z)|\leq \ld_0(1+C\sigma),\ \ \text{$\s$ is small}.$$

In view of $\psi_\rho\in F(\sigma,1;\delta)$, one has that
$\psi_\rho\leq \phi$ on $\p D_\sigma$. The maximum principle gives
that $\psi_\rho\leq \phi$ in $D_\sigma$.

Consider the points $\xi\in \p B_{\f34}$ with $\xi=(\xi_1,\xi_2)$
and $\xi_2<-\f12$, and define a function $\varphi_\xi$ solving the following problem
$$\left\{\ba{ll}&\Delta\varphi_\xi=0~~~~\text{in}~~D_\sigma\setminus B_{\f1{10}}(\xi),\\
&\varphi_\xi=0~~~~\text{on}~~\p D_\sigma,\ \ \ \ \
\varphi_\xi=-\ld_0y~~~~\text{on}~~\p B_{\f1{10}}(\xi).\ea\right.$$
Hopf's Lemma gives that \be\label{c58}\p_{-\nu}\varphi_\xi(Z)\geq
c\ld_0>0.\ee

Suppose that $$\psi_\rho(X)\leq\phi(X)+\gamma\sigma\ld_0 y\ \
\text{for any $X\in B_{\f1{10}}(\xi)$},$$ for a positive constant
$\gamma>0$. It should be noted that $\Delta(\psi_\rho-
\phi+\gamma\sigma\varphi_\xi)=-\rho f(\rho\psi_\rho)+\rho f(0)\geq 0$
in $D_\sigma\setminus B_{\f1{10}}(\xi)$, and the maximum principle gives
that
$$\psi_\rho\leq \phi-\gamma\sigma\varphi_\xi\ \ \ \ \text{in}\ \ D_\sigma\setminus
B_{\f1{10}}(\xi).$$ Therefore, we have
$$\ld_0=\limsup_{X\rightarrow Z}\f{\psi_\rho(X)}{\text{dist}(X,\p D_\sigma)}\leq\p_{-\nu}\phi(Z)-\gamma\sigma\p_{-\nu}\varphi_\xi(Z)\leq\ld_0(1+C\sigma-c\gamma\sigma),$$
where $\nu$ is the outer normal vector and we have used \eqref{c58}.
This leads a contradiction, provided that $\gamma$ is sufficiently
large.

Hence, there exists some point $X_\xi=(x_\xi,y_\xi)\in
B_{\f1{10}}(\xi)$, such
that\be\label{c60}\psi_\rho(X_\xi)\geq\phi(X_\xi)+C\sigma\ld_0
y_\xi.\ee

For any $X\in B_{\f15}(\xi)$, it follows from \eqref{c61},
\eqref{c60} and (iii) in Definition \ref{def3} that
$$\psi_\rho(X)\geq \psi_\rho(X_\xi)-\f{3\ld_0(1+\delta)}{10}\geq-\ld_0(1+\sigma)y_\xi+C\sigma\ld_0 y_\xi-\f{3\ld_0(1+\delta)}{10}\geq \ld_0\left(\f1{10}-C\sigma\right)>0,
$$ provided that $\sigma$ is small. Therefore, we have
$$\Delta(\phi-\psi_\rho)=-\rho f(0)+\rho f(\rho\psi_\rho)\ \ \text{and}\ \ \phi-\psi_\rho\geq0\ \ \text{in}\ \
B_{\f15}(\xi).$$ Thanks to Harnack's inequality in \cite{T}, one has
$$(\phi-\psi_\rho)(\xi)\leq(\phi-\psi_\rho)(X_\xi)+C\rho\|f(0)-f(\rho\psi_\rho)\|_{L^\infty(B_1)}\leq  C\ld_0\sigma+C\sigma^2\leq C\ld_0\sigma,$$
if $\sigma$ is small, which implies that
\be\label{c62}\psi_\rho(\xi)\geq\phi(\xi)-C\ld_0\s\geq\ld_0(-(1+\sigma)\xi_2-C\sigma\geq
-\xi_2-C\sigma.\ee

It follows from \eqref{c62} and $(iii)$ in Definition \ref{def3}
that
$$\psi_\rho(\xi+te_2)\geq \psi_\rho(\xi)-\ld_0(1+\delta)t\geq \ld_0(-(\xi_2+t)-C\sigma),$$ for $t>0$ and $\xi+te_2\in B_{\f12}$,
which implies that $\psi_\rho$ is of
$\mathcal{F}(2\sigma,C\sigma;\delta)$ in $B_{\f12}$. This completes the proof of this lemma.
\end{proof}

In the following, denote $B_r=B_r(0)$, we will investigate the
blow-up limit. Thanks to Lemma 7.3 and Corollary 7.4 in \cite{AC1},
we have the following non-homogeneous blow-up limit and we omit the
proof here.
\begin{lemma}\label{lc9} (Non-homogenous blow up)Let $\psi_n\in \mathcal{F}(\sigma_n,\sigma_n;\delta_n)$ in $B_{\rho_n}$
with $v=e_2$, $\sigma_n\rightarrow0$,
$\delta_n\sigma^{-2}_n\rightarrow0$ and
$\rho_n\s_n^{-1}\rightarrow0$. For any $x\in(-1,1)$, set
$$b^+_n(x)=\sup\{h\mid
(\rho_nx,\sigma_n\rho_nh)\in\p\{\psi_n>0\}\},$$ and
$$b^-_n(x)=\inf\{h\mid
(\rho_nx,\sigma_n\rho_nh)\in\p\{\psi_n>0\}\},$$ then for a
subsequence,
$$b(x)=\limsup_{z\rightarrow x,n\rightarrow\infty}b_n^+(z)=\liminf_{z\rightarrow x,n\rightarrow\infty}b_n^-(z)\ \ \text{for any $x\in(-1,1)$}.$$
Furthermore, $b_n^+(x)\rightarrow b(x)$ and $b_n^-(x)\rightarrow b(x)$
uniformly in $(-1,1)$, $b(0)=0$ and $b(x)$ is a continuous function.
\end{lemma}

Furthermore, we can show that the limit function $b(x)$ is a convex function.
\begin{lemma}\label{lc10} $b(x)$ is convex with respect to $x$.

\end{lemma}

\begin{proof} Set $\psi_n(X)=\f{\psi(\rho_nX)}{\rho_n}$, $\ld_n(X)=\ld(\rho_nX)$ and $\ld_0=\ld(0)$,
one has
$$\Delta\psi_n+\rho_n f(\rho_n\psi_n)=0\ \ \text{in}\ \
\{\psi_n>0\}\ \ \text{and $\psi_n\in
\mathcal{F}(\sigma_n,\sigma_n;\delta_n)$ in $B_1$}.$$

If the assertion is not true, there exists a $x_0\in(-1,1)$ and a linear function $g(x)$ in
$I_\rho=(-\rho+x_0,\rho+x_0)\subset(-1,1)$, such that
$$g(x_0\pm\rho)>b(x_0\pm\rho) \ \ \text{ and}\ b(x_0)>g(x_0).$$
Denote
$$\mathcal{D}=I_\rho\times\mathbb{R}\ \text{and}\ \mathcal{D}^+(k)=\{(x,y)\in\mathcal{D}\mid
k(x)<y<g_0+2\},$$ and
$$\mathcal{D}^0(k)=\{(x,y)\in\mathcal{D}\mid y=k(x)\}\ \text{and}\ \mathcal{D}^-(k)=\{(x,y)\in\mathcal{D}\mid y<k(x)\},$$ for continuous function $k(x)$, where $g_0=\max\{g(x_0-\rho),g_0(x_0+\rho)\}$.
 Define a
test function
$d_{\e,E}(X)=\min\left\{\f{\text{dist}(X,\mathbb{R}^2\setminus
E)}{\e},1\right\}$ for a set $E$, it is easy to check that
$d_{\e,E}(X)$ converges to the characteristic function $I_{E}(X)$ as
$\e\rightarrow 0$.

It follows from Proposition \ref{lc1} that
\be\label{c63}\ba{rl}&-\int_{\O_n}\nabla\psi_n\cdot\nabla
d_{\e,E}(\mathcal{D}^+(\sigma_ng)) dX
+\int_{\O_n\cap\{\psi_n>0\}}\rho_n f(\rho_n\psi_n)
d_{\e,E}(\mathcal{D}^+(\sigma_ng)) dX\\
=&
 \int_{ \O_n\cap\p_{red}\{\psi_n>0\}} d_{\e,E}(\mathcal{D}^+(\sigma_ng))\ld_n d\mathcal{H}^1,\ea
 \ee where $\O_n=\{X\mid \rho_nX\in\O\}$.

Taking $\e\rightarrow0$ in \eqref{c63}, one has
$$\ba{rl}&-\int_{\p\mathcal{D}^+(\sigma_ng)\cap\{\psi_n>0\}}\nabla\psi_n\cdot\nu d\mathcal{H}^1+\int_{\mathcal{D}^+(\sigma_ng)\cap\{\psi_n>0\}}\rho_n f(\rho_n\psi_n) dX\\
=&
 \int_{\mathcal{D}^+(\sigma_ng)\cap\p_{red}\{\psi_n>0\}} \ld_n d\mathcal{H}^1,\ea
 $$ which gives that
\be\label{c64}\ba{rl}&\mathcal{H}^1(\mathcal{D}^+(\sigma_ng)\cap\p_{red}\{\psi_n>0\})\\
\leq &\f{1}{\ld_0(1-\delta_n)} \int_{\mathcal{D}^+(\sigma_ng)\cap\p_{red}\{\psi_n>0\}} \ld_n d\mathcal{H}^1\\
\leq&\f{1}{\ld_0(1-\delta_n)}
\left(\int_{\p\mathcal{D}^+(\sigma_ng)\cap\{\psi_n>0\}}|\nabla\psi_n|
d\mathcal{H}^1+C\rho_n\int_{\mathcal{D}^+(\sigma_ng)\cap\{\psi_n>0\}}
dX\right)\\
\leq&\f{1+\delta_n}{1-\delta_n}\mathcal{H}^1(\mathcal{D}^0(\sigma_ng)\cap\{\psi_n>0\})+\f{C\rho_n\sigma_n}{\ld_0(1-\delta_n)}.
\ea\ee Denote $$E_n=\{\psi_n>0\}\cup \mathcal{D}^-(\sigma_ng).$$ It
is easy to check that $E_n$ has finite perimeter in $\mathcal{D}$
and
\be\label{c65}\ba{rl}\mathcal{H}^1(\mathcal{D}\cap\p_{red}E_n)\leq
\mathcal{H}^1(\mathcal{D}^+(\sigma_ng)\cap\p_{red}\{\psi_n>0\})+\mathcal{H}^1(\mathcal{D}^0(\sigma_ng)\cap\{\psi_n=0\}).
\ea\ee

By using the similar estimate in Page 136-Page 137 in \cite{AC1},
one has
\be\label{c66}\ba{rl}\mathcal{H}^1(\mathcal{D}\cap\p_{red}E_n)\geq
\mathcal{H}^1(\mathcal{D}^0(\sigma_ng))+c\sigma_n^2. \ea\ee It
follows from \eqref{c64}-\eqref{c66} that
\be\label{c67}\ba{rl}&\mathcal{H}^1(\mathcal{D}^0(\sigma_ng))+c\sigma_n^2\leq\f{1+\delta_n}{1-\delta_n}\mathcal{H}^1(\mathcal{D}^0(\sigma_ng))
+\f{C\rho_n\sigma_n}{\ld_0(1-\delta_n)}, \ea\ee that is
$$0<c\leq \f{2\delta_n\sigma_n^{-2}}{1-\delta_n}\mathcal{H}^1(\mathcal{D}^0(\sigma_ng))+\f{C\rho_n\sigma_n^{-1}}{\ld_0(1-\delta_n)},$$
which contradicts to the facts $\rho_n\sigma_n^{-1}\rightarrow0$ and
$\delta_n\sigma^{-2}_n\rightarrow0$.

\end{proof}







With the aid of Lemma \ref{lc10}, based on Lemma 7.6 in \cite{AC1}
and Lemma 5.6 in \cite{ACF8}, we will show that
\begin{lemma}\label{lc11} There exists a constant $C$, such that for any $x\in\left(-\f 12,\f 12\right)$,
$$\int_{0}^{\f 14}\f{1}{r^2}(\text{Av}_r(b(x))-b(x))dr\leq C,$$
where $\text{Av}_r(b(x))=\f{b(x+r)+b(x-r)}{2}$.

\end{lemma}

\begin{proof} Set $\psi_n(X)=\f{\psi(\rho_nX)}{\rho_n}$, $\ld_n(X)=\ld(\rho_nX)$ and $\ld_0=\ld(0)$,
one has \be\label{c69}\Delta\psi_n+\rho_n f(\rho_n\psi_n)=0\ \
\text{in}\ \ \{\psi_n>0\},\ \ \text{and $\psi_n\in
\mathcal{F}(\sigma_n,\sigma_n;\delta_n)$ in $B_1$}.\ee It is easy to
check that $\psi_n$ be the class of
$\mathcal{F}(8\sigma_n,8\sigma_n;\delta_n)$ in $B_{\f14}(x,\sigma_nb_n^+(x))$,
provided that $n$ is sufficiently large. Therefore, it suffices to
show the lemma for $x=0$, that is \be\label{c68}\int_{0}^{\f
14}\f{1}{r^2}(\text{Av}_r(b(0)))dr\leq C.\ee

Set
$$\phi_n(x,y)=\f{\psi_n(x,y)+\ld_0y}{\sigma_n}.$$
Recalling the flatness condition for $\psi_n$, the free boundary
of $\psi_n$ lies in the strip $|y|\leq c\sigma_n$. Since
$|\g\psi_n|\leq\ld_0(1+\delta_n)$ and $\delta_n\leq \sigma_n$, we
have $$\text{$\phi_n\leq C$ in $B_1^-=B_1\cap\{y<0\}$.}$$ The flatness condition implies that
$\phi_n\geq -C$ in $B_1^-$. Therefore, one has
\be\label{c70}|\phi_n|\leq C\ \ \text{in}\ \ B_1^-.\ee

It is easy to check that
\be\label{c700}\Delta\phi_n=\f{\rho_n}{\sigma_n}f(\rho_n(\sigma_n\phi_n-\ld_0y))\
\ \text{in}\ \ B_1^-\cap\{y<-\sigma_n\}.\ee

Thanks to the flatness conditions on $\psi_n$ and Lemma \ref{lc7},
there exists a subsequence still labeled by $\{\psi_n\}$, such that
$$\psi_n(x,y)\rightarrow-\ld_0y\ \ \text{in $C^2$ in any compact subset of
$B_1^-$}.$$ In view of \eqref{c70}, we can conclude that
\be\label{c71}\phi_n\rightarrow\phi\ \ \text{in $C^2$ in compact
subsets of $B_1^-$}.\ee

Since $\rho_n\sigma_n^{-1}\rightarrow0$, it follows from
\eqref{c700} and \eqref{c71} that \be\label{c72}\Delta\phi=0\ \
\text{in}\ \ B_1^-,\ \ \text{and}\ \ |\phi|\leq C.\ee

By virtue of $\phi_n(0,0)=0$,  we have that for any $y\leq 0$, there
exists a $\xi\in(y,0)$, such that
\be\label{c73}\phi_n(0,y)\leq\f{\p\phi_n(0,\xi)}{\p
y}y=\f1{\sigma_n}\left(\f{\p\psi_n(0,\xi)}{\p
y}+\ld_0\right)y\leq\f{|\g\psi_n|-\ld_0}{\sigma_n}|y|\leq
\f{\ld_0\delta_n}{\sigma_n}|y|\rightarrow0.\ee Then it gives that
\be\label{c74}\phi(0,y)\leq 0\ \ \text{for any $y\leq0$}.\ee

 Next, we will show that
 \be\label{c75}\phi(x,0)=\ld_0b(x)\ \ \text{in the sense that $\lim_{y\uparrow
 0}\phi(x,y)=\ld_0b(x)$}.\ee

 To obtain the fact \eqref{c75}, we first show that for any small $\eta>0$ and any large constant $M$, we have
  \be\label{c760}\phi_n(x,y\sigma_n)\rightarrow  \ld_0b(x)\ \ \text{uniformly for $x\in I_\eta=(-1+\eta,1-\eta)$ and $-M\leq y\leq
  -1$},\ee as $n\rightarrow\infty$.

On another hand, to see this, thanks to the non-homogeneous blow-up Lemma \ref{lc9}, it suffices to show that
 \be\label{c76}\phi_n(x,h\sigma_n)-\ld_0b_n^+(x)\rightarrow 0.\ee
It follows from \eqref{c73} that for any $x\in I_\eta$, one has
\be\label{c77}\ba{rl} \phi_n(x,y\sigma_n)-\ld_0b_n^+(x)=&\phi_n(x,\sigma_nb_n^+(x))-\ld_0b_n^+(x)+\f{\p\phi_n(0,\xi)}{\p y}\sigma_n(y-b_n^+(x))\\
\leq &\delta_n(b_n^+(x)-y) \leq (1+M)\delta_n\rightarrow 0,\ea\ee as $n\rightarrow\infty$,
where we have used the fact
$$\phi_n(x,\sigma_nb_n^+(x))-\ld_0b_n^+(x)=\f{\psi_n(x,\sigma_nb_n^+(x))+\ld_0\sigma_nb_n^+(x)}{\sigma_n}-\ld_0b_n^+(x)=\f{\psi_n(x,\sigma_nb_n^+(x))}{\sigma_n}=0.$$

Taking any sequence $\{x_n\}$ with $x_n\in I_\eta$ and $-M\leq
h_n\leq -1$, and consider $\psi_n$ in $B_{R\sigma_n}(X_n)$, where
$X_n=(x_n,\sigma_nb_n^+(x_n))$ is the free boundary point and $R$ is
an arbitrary large constant. Denote $\t \sigma_n=\f1R\sup_{x\in
I_{R\sigma_n}}(b_n^+(x)-b_n^+(x_n))$ with
$I_{R\sigma_n}=(x_n-R\sigma_n,x_n+R\sigma_n)$. It follows from Lemma
\ref{lc9} that $\t\sigma_n\rightarrow 0$. Next, we claim that
\be\label{c707}\psi_n\in \mathcal{F}(\t \sigma_n,1;C_1\delta_n)\ \
\text{in}\ \ B_{R\sigma_n}(X_n),\ee where $C_1$ is a constant
depending on $\ld_1$ and $\ld_2$, provided that $n$ is sufficiently
large. In fact, the definition of $b^+_n(x)$ gives that$$\text{
$\psi_n(x,y)=0\ $ for $\ y\geq\sigma_nb_n^+(x_n)+
\t\s_n R\s_n$}.$$ 
Since
$|\g\psi_n(X)|\leq\ld_0(1+\delta_n)$ in $B_{1}(0)$ and $\sup_{X,Y\in
B_{1}(0)}|\ld_n(X)-\ld_n(Y)|\leq\ld_0\delta_n,$ we have that
$$|\g\psi_n(X)|\leq(\ld_n(X_n)+\ld_0\delta_n)(1+\delta_n)\leq \ld_n(X_n)\left(1+C_1\delta_n\right)\ \ \text{in
$B_{R\s_n}(X_n)$,}$$ and
$$\sup_{X,Y\in
B_{R\s_n}(X_n)}|\ld_n(X)-\ld_n(Y)|\leq\ld_0\delta_n\leq(\ld_n(X_n)+\ld_0\delta_n)\delta_n\leq\ld_n(X_n)
C_1\delta_n,$$ where $C_1$ is a constant depending on $\ld_1$ and
$\ld_2$. Therefore, we obtain the claim \eqref{c707}.

In view of the claim \eqref{c707}, by virtue of the flatness Lemma
\ref{lc8}, we have \be\label{c708}\psi_n\in \mathcal{F}(2\bar
\sigma_n,C\bar\sigma_n;C_1\delta_n)\ \ \text{in}\ \
B_{\f{R\sigma_n}2}(X_n),\ee provided that $\bar
\sigma_n=\max\{\t\sigma_n,\delta_n\}$.

Hence, for any $h$ with $|h|<\f R2$, it follows from \eqref{c708}
that
$$\psi_n(X_n+h\sigma_ne_2)\geq -\ld_n(X_n)\left(h\sigma_n+\f{C\bar\sigma_nR\sigma_n}{2}\right),$$
that is
$$\ba{rl}\phi_n(X_n+h\sigma_ne_2)-\ld_0b_n^+(x_n)=&\f{\psi_n(X_n+h\sigma_ne_2)+\ld_0h\sigma_n}{\sigma_n}\\
\geq&(\ld(0)-\ld(\rho_nX_n))h-\f{C\ld(\rho_nX_n)\bar\sigma_nR}{2}\rightarrow
0.\ea$$ This together with \eqref{c77} gives \eqref{c76}.

For any $\e>0$, 
let $\varphi_\e$ be the solution of the Dirichlet problem
$$\left\{\ba{ll}&\Delta\varphi_\e=1~~~~\text{in}~~B_{1-\eta}^-,\\
&\varphi_\e=\ld_0b(x)-\e~~~~\text{on}~~\p B_{1-\eta}^-\cap\{y=0\},\ \ \ \ \
\varphi_\e=\inf_{B_1^-}\phi~~~~\text{on}~~\p
B_{1-\eta}^-\cap\{y<0\}.\ea\right.$$ It follows from \eqref{c760} that \be\label{c780}\phi_n>\varphi_\e\ \ \text{on}\
\ \p( B_{1-\eta}^-\cap\{y<-M\sigma_n\}),\ee for any large constant
$M$ (independent of $\eta$ and $\e$), provided that $n$ is
sufficiently large (depending on $\e$ and $M$). Moreover,
$$\Delta(\varphi_\e-\phi_n)=1+\f{\rho_n}{\sigma_n}f(\rho_n(\sigma_n\phi_n-\ld_0y))>0\
\text{ in $B_{1-\eta}^-\cap\{y<-M\sigma_n\}$},$$ provided that $n$
is sufficiently large. In view of \eqref{c780}, the maximum principe
gives that \be\label{c790}\varphi_\e\leq\phi_n\ \ \text{ in
$B_{1-\eta}^-\cap\{y<-M\sigma_n\}$, if $n$ is large enough}.\ee
Taking $n\rightarrow\infty$ in \eqref{c790}, we have
$$\varphi_\e(x,y)\leq\phi(x,y)\ \
\text{ in $B_{1-\eta}^-$}.$$ Consequently,
\be\label{c79}\ld_0b(x)-2\e\leq\lim_{y\uparrow0}\varphi_\e(x,y)\leq\liminf_{y\uparrow
0}\phi(x,y)\ \ \text{ for $x\in(-1+2\eta,1-2\eta)$ and $\eta>0$}.\ee

Similarly, by working with a solution $\t\varphi_\e(x,y)$ to the Dirichlet problem
$$\left\{\ba{ll}&\Delta\t\varphi_\e=-1~~~~\text{in}~~B_{1-\eta}^-,\\
&\t\varphi_\e=\ld_0b(x)+\e~~~~\text{on}~~\p
B_{1-\eta}^-\cap\{y=0\},\ \ \ \ \
\t\varphi_\e=\sup_{B_1^-}\phi~~~~\text{on}~~\p
B_{1-\eta}^-\cap\{y<0\},\ea\right.$$ for any $\e>0$, we can obtain
that \be\label{c80}\ld_0b(x)+2\e\geq\limsup_{y\uparrow 0}\phi(x,y)\
\ \text{ for $x\in(-1+2\eta,1-2\eta)$}.\ee
 By virtue of the arbitrariness of $\e$, it follows from \eqref{c79} and \eqref{c80} that we obtain the claim
 \eqref{c75}.

In view of Lemma \ref{lc10}, \eqref{c70}, \eqref{c72},\eqref{c74}
and \eqref{c75}, we can apply Lemma 5.5 in \cite{ACF8} to obtain
\eqref{c68}.
\end{proof}

With the aid of Lemma \ref{lc11}, it follows from Lemma 7.7 - Lemma
7.9 in \cite{AC1} that the flatness from below implies the flatness
from above (see Figure \ref{f4}).

\begin{figure}[!h]
\includegraphics[width=130mm]{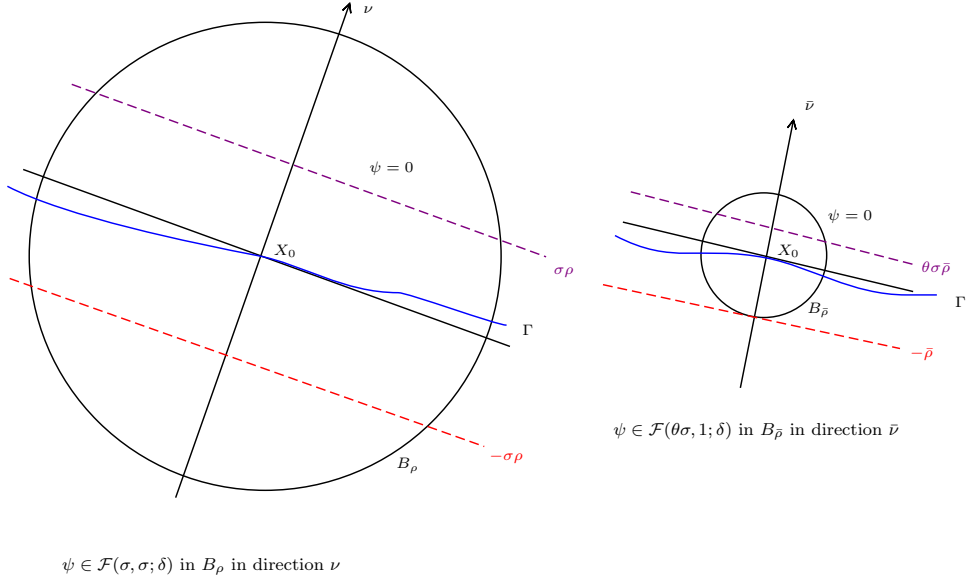}
\caption{Flatness from above}\label{f4}
\end{figure}

\begin{lemma}\label{lc12} For any $\th>0$, there exist a large constant $C$, $c_\th=c(\th)$ and a small $\sigma_\th=\sigma(\th)$, such that if
$$\psi\in \mathcal{F}(\sigma,\sigma;\delta)\ \ \text{in $B_\rho$ in direction $\nu$ with $\sigma\leq\sigma_\th, \delta\leq\sigma_\th\sigma^2, \rho\leq\sigma_\th\sigma$},$$
then
$$\psi\in \mathcal{F}(\th\sigma,1;\delta)\ \ \text{in $B_{\bar\rho}$ in direction $\bar\nu$ for some $\bar\rho$ and $\bar\nu$
 with $c_\th\rho\leq\bar\rho\leq\th\rho$ and $|\nu-\bar\nu|\leq C\sigma$}.$$

\end{lemma}
By virtue of Lemma \ref{lc8} and Lemma \ref{lc12}, we will show
that
\begin{lemma}\label{lc13} Let $\o(s)$ be a monotone increasing and continuous function with $\o(0)=0$. For any $\th\in
(0,1)$, there exists a small $\sigma_0>0$, such that if
$$\psi\in \mathcal{F}(\sigma,1;\delta)\ \ \text{in $B_\rho$ in direction $\nu$ with
$\sigma\leq\sigma_0,\ \delta\leq\sigma_0\sigma^2,\
\rho\leq\sigma_0\sigma$},$$ and $\sup_{X,Y\in
B_{s\rho}}|\ld(X)-\ld(Y)|\leq \o(s)\ld(0)\delta$ for $s\in[0,1]$,
then there exist a large constant $C$ and $c_\th=c(\th,\o)$, such
that
$$\psi\in \mathcal{F}(\th\sigma,\th\sigma;\th^2\delta)\ \ \text{in $B_{\bar\rho}$ in direction $\bar\nu$}$$ for some $\bar\rho$ and $\bar\nu$,
where $c_\th\rho\leq\bar\rho\leq\f14\rho$ and $|\nu-\bar\nu|\leq C\sigma$.

\end{lemma}
\begin{proof} {\bf Step 1}. By virtue of Lemma \ref{lc8}, one has
$$\psi\in \mathcal{F}(C_0\sigma,C_0\sigma;\delta)\ \ \text{in $B_{\f\rho2}$ in direction $\nu$, where $C_0>2$.}$$

Then for some small $\th_1\in\left(0,1/2\right]$ to be determined
later, it follows from Lemma \ref{lc12} that
$$\psi\in \mathcal{F}(C_0\th_1\sigma,1;\delta)\ \ \text{in $B_{\rho_1}$ in direction $\nu_1$ for some $\rho_1$ and $\nu_1$},$$ where $$\text{$c_{\th_1}\f\rho2\leq\rho_1\leq\th_1\f\rho2$ and $|\nu-\nu_1|\leq C\sigma$}.$$
To improve the value of $\delta$, we define
$$U(X)=\max\left\{|\g\psi(X)|-\sup_{X\in B_{3\rho_1}}\ld(X),0\right\}\ \ \text{in}\ \
B_{2\rho_1}\cap\{\psi>0\}.$$ It is easy to check that
$$\Delta U=\f{|D^2\psi||\g\psi|^2-\p_{k}\psi\p_{j}\psi\p_{ij}\psi\p_{ik}\psi}{|\g\psi|^3}-f'(\psi)|\g\psi|\geq 0\ \ \text{in}\ \
B_{2\rho_1}\cap\{\psi>0\}.$$ For subharmonic function $U$, we can use the similar arguments
in P141 in \cite{AC1} and obtain that there exists a $c\in(0,1)$, such that
$$U(X)=\max\left\{|\g\psi(X)|-\sup_{X\in B_{3\rho_1}}\ld(X),0\right\}\leq(1-c)\ld(0)\delta\ \ \text{in}\ \
B_{\rho_1},$$ which implies that
$$|\g\psi(X)|\leq\sup_{X\in B_{3\rho_1}}\ld(X)+\ld(0)(1-c)\delta\leq\ld(0)+\ld(0)(1-c+w(2\th_1))\delta\ \ \text{in}\ \ B_{\rho_1}.$$
Denote $$\th_0=\sqrt{1-\f{c}2}.$$ Taking $\th_1$ be sufficiently small, such that $\f{\th_1}{2\th_0}<1$ and $C_0\th_1<\th_0$, then the continuity of $\o(s)$ gives that
$$\o(2\th_1)\leq \f{c}2\ \ \text{and}\ \ \o\left(\f{\th_1}2\right)<
\th_0^2.$$ Furthermore,
$$|\g\psi(X)|\leq\ld(0)+\ld(0)\left(1-\f{c}2\right)\delta\leq \ld(0)(1+\th_0^2\delta)\ \ \text{in}\ \
B_{\rho_1},$$ and
$$\sup_{X,Y\in
B_{\rho_1}}|\ld(X)-\ld(Y)|\leq\sup_{X,Y\in
B_{\f{\th_1}{2}\rho}}|\ld(X)-\ld(Y)|\leq
\o\left(\f{\th_1}2\right)\ld(0)\delta\leq \th_0^2\ld(0)\delta.
$$  Thus we have
$$\psi\in \mathcal{F}(\th_0\sigma,1;\th_0^2\delta)\ \ \text{in $B_{\rho_1}$ in direction
$\nu_1$}.$$

On the other hand, we have
$$\sup_{X,Y\in
B_{s\rho_1}}|\ld(X)-\ld(Y)|\leq \o_1(s)\th_0^2\ld(0)\delta \ \ \text{and}\ \
\o_1(s)=\th_0^{-2}\o\left(\f{\th_1}2 s\right)\leq\th_0^{-2}\o\left(\f{\th_1}2\right)<1. $$  Notice that
$$\rho_1\leq\th_1\f\rho2\leq \left(\f{\th_1}{2\th_0}\sigma_0\right)\th_0\sigma<\sigma_0(\th_0\sigma)\ \ \text{and}\ \ \th^2_0\delta\leq \sigma_0(\th_0\sigma)^2,$$
due to $\f{\th_1}{2\th_0}<1$ and $\rho\leq\s_0\s$.

{\bf Step 2.}  Repeating the arguments in Step 1, and choosing
sufficiently small $\sigma_0>0$, we obtain
$$\psi\in \mathcal{F}(\th^2_0\sigma,1;\th_0^4\delta)\ \ \text{in $B_{\rho_2}$ in direction
$\nu_2$},$$ for some $\rho_2$ and $\nu_2$, where
$$\text{$c_{\th_1}\f{\rho_1}2\leq\rho_2\leq\th_1\f{\rho_1}2$ and
$|\nu_2-\nu_1|\leq C\th_0\sigma$}.$$

Similarly, we can repeat the arguments in Step 1 for a finite number
$n$, and we choose the constant $\sigma_0$ being small enough in the
statement for each step, thus \be\label{c81}\psi\in
\mathcal{F}(\th^n_0\sigma,1;\th_0^{2n}\delta)\ \ \text{in
$B_{\rho_n}$ in direction $\nu_n$},\ee for some $\rho_n$ and
$\nu_n$, where
\be\label{c82}\text{$c_{\th_1}\f{\rho_{n-1}}2\leq\rho_n\leq\th_1\f{\rho_{n-1}}2$
and $|\nu_n-\nu_{n-1}|\leq C\th^{n-1}_0\sigma$}.\ee Applying Lemma
\ref{lc8} again, we have
$$\psi\in \mathcal{F}(2\th_0^n\sigma,C\th_0^n\sigma;\th_0^{2n}\delta)\ \ \text{in $B_{\f{\rho_n}2}$ in direction
$\nu_n$.}$$

Choosing $N=N(\th)$ as the smallest positive integer such that
$$\max\left\{2\th_0^N, C\th_0^N\right\}\leq \th,$$ where the constant $C$ is
obtained in Lemma \ref{lc8}. It follows from \eqref{c82} that
\be\label{c83}\text{$\f{c^N_{\th_1}}{2^N}\rho\leq\rho_N\leq\f{\th_1^N}{2^N}\rho$
and $|\nu_N-\nu|\leq \f{C\sigma}{1-\th_0}$}.\ee Then we have
$$\psi\in \mathcal{F}(\th\sigma,\th\sigma;\th^2\delta)\ \ \text{in $B_{\bar\rho}$ in direction
$\bar\nu$,}$$ where $\bar\rho=\f{\rho_N}{2}$ and $\bar\nu=\nu_N$.
Denote $c_\th=\f{c^N_{\th_1}}{2^{N(\th)+1}}$, it follows from
\eqref{c83} that
$$\text{$c_\th\rho\leq\bar\rho\leq\f{\th_1^N}{2^{N+1}}\rho\leq\f\rho4$ and $|\nu-\bar\nu|\leq \f{C\sigma}{1-\th_0}\leq C\sigma$}.$$
\end{proof}

\subsection{The regularity of the free boundary}
In Lemma \ref{lc13}, we show that the flatness in $B_\rho$ is
improved in a smaller disc $B_{\bar\rho}$. Based on the results in
the previous subsection, we will investigate the regularity of the
free boundary in this subsection.

By virtue of Theorem 8.1 in \cite{AC1}, the flatness of free boundary implies immediately the $C^{1,\alpha}$-smoothness of the free boundary, which is the main result in this paper. We omit the proof and present the result as follows.

\begin{theorem}\label{lc14} For any compact subset $D$ of $\O$, there exists $\beta>0,\ \sigma_0>0,\ \gamma_0>0$ and $C>0$, such that
if
\be\label{b830}\psi\in \mathcal{F}(\sigma,1;\infty)\ \ \text{in $B_\rho(X_0)$ in direction $\nu$}\ee
where $X_0\in\p\{\psi>0\}$, $B_\rho(X_0)\subset D$ and $$\text{
$\sigma\leq\sigma_0\ \ \text{and}\ \
\rho\leq\gamma_0\sigma^{\f2\beta}$},$$ then
$$B_{\f\rho4}(X_0)\cap\p\{\psi>0\}\ \ \text{is a $C^{1,\alpha}$
curve}\ \ \text{for some}\ \alpha>0.$$ Namely, a graph in direction
$\nu$ of a $C^{1,\alpha}$ function, for any $X_1$ and $X_2$ on this
curve,
$$|\nu(X_1)-\nu(X_2)|\leq C\sigma\left|\f{X_1-X_2}{\rho}\right|^\alpha.$$
In particular, if $\nu=e_2$, then $B_{\f\rho4}(X_0)\cap\p\{\psi>0\}$
can be written in the form $y=g(x)\in C^{1,\alpha}$; Furthermore,
$$|g'(x_1)-g'(x_2)|\leq C\sigma\left|\f{X_1-X_2}{\rho}\right|^\alpha.$$
\end{theorem}

Moreover, the free boundary possesses higher regularity, provided
that the functions $\ld(X)$ and the vorticity strength function
$f(t)$ possess higher regularity. With the aid of Lemma \ref{lc7}
and Theorem \ref{lc14}, by using the similar arguments in Theorem
3.12 and Corollary 3.13 in \cite{FA2}, we have

\begin{theorem}\label{lc15} (1) If $\ld(X)$ is $C^{k,\beta}$ and $f(t)$ is $C^{1,\beta}$, for $k=0,1,2$,
then the free boundary $\O\cap\p\{\psi>0\}$ is $C^{k+1,\alpha}$ locally in
$\O$ with $\alpha\in(0,\beta)$. \\
(2) If $\ld(X)$ is $C^{k,\beta}$ and $f(t)$ is $C^{k-1,\beta}$, for
$k=3,4, \mathellipsis$, then the free boundary $\O\cap\p\{\psi>0\}$ is
$C^{k+1,\alpha}$ locally in $\O$ with $\alpha\in(0,\beta)$.\\
(3) If $\ld(X)$ and $f(t)$ are analytic, then the free boundary
$\O\cap\p\{\psi>0\}$ is locally analytic in $\O$.\end{theorem}

\begin{proof}
For any $X_0\in\O\cap\p\{\psi>0\}$, let
$\psi_n(X)=\f{\psi(X_0+\rho_nX)}{\rho_n}$ be a blow up sequence with
respect to the discs $B_{\rho_n}(X_0)$, it follows from Lemma
\ref{lc7} that there exists a unit vector $\nu_0$, such that
$$\psi_n(X)\rightarrow\psi_0(X)=\ld(X_0)\max\{-X\cdot\nu_0,0\}\ \
\text{in any compact subset of $\mathbb{R}^2$}.$$

With the aid of Lemma \ref{lc2}, we have
$$\psi(X_0+\rho_nX)=\ld(X_0)\rho_n\max\{-X\cdot\nu_0,0\}+o(\rho_n)\ \ \text{in}\ \
B_1(0),$$ which implies that there exists a sequence $\{\s_n\}$ with
$\s_n\rightarrow0$, such that
$$\psi\in F(\s_n,1,\infty) \ \ \text{in $B_{\rho_n}(X_0)$ in direction $\nu_0$, $\s_n\leq\s_0$ and
$\rho_n\leq\gamma_0\s_n^{\f{2}{\beta}}$},$$ provided that $n$ is
sufficiently large. Applying Theorem \ref{lc14}, we have that
$B_{\f{\rho_n}4}(X_0)\cap\p\{\psi>0\}$ is $C^{1,\alpha}$.

Next, we can take a $C^{1,\alpha}$ transformation to flatten the free
boundary. Then reflect $\psi$ to the full neighborhood of the free
boundary, applying the Schauder estimates for the elliptic equations in
divergence form in Section 9 in \cite{A1}, we can obtain the
$C^{1,\alpha}$ regularity of $\psi$ up to the free boundary.

It follows from Proposition \ref{lb2} that
$$\lim_{\e\rightarrow 0^+}\int_{\O\cap\p\{\psi>\e\}}(|\g\psi|^2-\ld^2(X)-F(\psi))\eta\cdot\nu_\e dS=0.$$
 Since $\psi$ is $C^{1,\alpha}$ up to the free boundary and $F(\psi)=0$ on the free boundary, we
have
$$|\g\psi(X)|=\ld(X)\ \ \text{on the free boundary}.$$

Without loss of generality, we assume that the outward normal direction to
$\p\{\psi>0\}$ is in the direction of the positive $y$-axis. Extend
$\psi$ and $f$ as $C^{1,\alpha}$ functions into a full neighborhood
of $0\in\p\{\psi>0\}$. In view of $|\g\psi(0)|=\ld(0)\geq\ld_1> 0$, we have
that \be\label{c103}\psi_y(0)<0.\ee Define a mapping as follows,
$$S=TX=(s,t)\triangleq (x,\psi(x,y)), \ \ X=(x,y).$$ In view of \eqref{c103}, it is easy to
check that
$$\text{det}\left(\f{\p S}{\p X}\right)=\psi_y(x,y)<0\ \ \text{in a neighborhood of $0$}.$$
And thus the mapping $T$ is a local diffeomorphism near $0$.

Construct a function $\phi(S)$ as follows,
$$\phi(S)=y.$$

Therefore, the free boundary $\Gamma$ is transformed into $t=0$, and
we have
$$\left(\f{\p X}{\p S}\right)=\left(\f{\p S}{\p X}\right)^{-1}=\left(\begin{matrix} 1 &0 \\
-\f{\psi_x}{\psi_y} &\f{1}{\psi_y}
\end{matrix}\right).$$
Consequently, one has \be\label{c104}\phi_s=\f{\p y}{\p
s}=-\f{\psi_x}{\psi_y},\ \ \phi_t=\f{\p y}{\p t}=\f{1}{\psi_y},\ \
\f{\p t}{\p x}=\psi_x=-\f{\phi_s}{\phi_t}\ \ \text{and} \ \ \f{\p
t}{\p y}=\psi_y=\f{1}{\phi_t}.\ee It follows from \eqref{c104} that
$$\mathcal{Q}\phi=\p_s\left(-\f{\phi_s}{\phi_t}\right)+\p_t\left(\f{1+\phi^2_s}{2\phi^2_t}\right)+f(t)=0\ \ \text{in the neighborhood of $0$}.$$

It is easy to check that $\mathcal{Q}\phi=0$ is a quasilinear elliptic
equation in a neighborhood $E$ of $0$. Furthermore, $\phi$ satisfies
the Neumann type boundary condition as follows,
\be\label{b05}\left\{\ba{ll}\mathcal{Q}\phi=0~~~~&\text{in}\ \ \ D,\\
\f{\phi_t}{\sqrt{1+\phi_s^2}}=\f{1}{\ld(s)}\ \ &\text{on}\ \ \bar
D\cap\{t=0\},\ea\right.\ee where $D=E\cap\{t>0\}$.

Noting that $\psi$ is in $C^{1,\alpha}$ near $0$, which implies that the coefficients of
 the operator $\mathcal{Q}$ are in $C^{\alpha}$. By using the elliptic regularity in Section 9 in \cite{A1}, we obtain
that
 $\phi(S)$ is in $C^{2,\alpha}$ near $0$. Furthermore, the free boundary $\Gamma$ can be denoted by
 $y=\phi(s,0)=\phi(x,0)$, and thus the free boundary $\Gamma$ is $C^{2,\alpha}$ near $0$.

Applying the Schauder estimates for elliptic equations in \cite{A1},
we can obtain the $C^{2,\alpha}$ regularity of $\psi$ up to the free
boundary $\Gamma$. Using the above arguments again, we can conclude
that the free boundary $\Gamma$ is $C^{3,\alpha}$ near $0$, provided
that $\ld(X)\in C^{1,\beta}$ and $f\in C^{1,\beta}$.

By bootstrap argument, we can obtain the higher regularity of the free
boundary $\Gamma$.

Finally, if $f(s)$ is analytic and $\ld(X)$ is analytic, by virtue
of the results of Section 6.7 in \cite{MB}, we can conclude that
$\phi$ is analytic. Hence, we obtain the analyticity of the free
boundary $\Gamma$.

\end{proof}

\begin{remark} The results in this paper can be extended for the $n$-dimensional case ($n\geq3$), and Theorem
\ref{lc14} replaced by Corollary 3.11 in \cite{FA2}.
\end{remark}

\section{Applications: steady ideal jet flow and cavitational flow with general vorticity}

As an important application of the mathematical theory of the free boundary problem \eqref{aa1}, we will investigate the well-posedness of the steady incompressible inviscid fluid with free streamline. There are at least two classical hydrodynamical problems of two-dimensional steady flows with non-trivial vorticity can be described mathematically as a free boundary problem \eqref{aa1} for a semilinear elliptic equation, the one is the two-dimensional incompressible inviscid jet flow problem and another one is the two-dimensional incompressible inviscid cavitational flow problem. For a classical example, we will investigate briefly  the existence and uniqueness of the two-dimensional incompressible inviscid jet flow issuing from a given semi-infinitely long nozzle in this section, and state the similar results on the incompressible cavitational flow problem.

In the previous sections, the technique of variational method has provided some solutions to the free boundary problem \eqref{aa1}. The regularity of the free boundary and the regularity of the derivative to the solution $\psi$ up to it, follows from Theorem \ref{lc15}. However, the common topological properties of the free boundary between $\O\cap\{\psi>0\}$ and $\O\cap\{\psi=0\}$ are still unknown. Of course, it is an important and natural question to ask if the free boundary is smooth enough to provide a classical solution of the steady incompressible jet flow problem under consideration. The main aim of this section is to establish the existence and uniqueness of the steady incompressible jet flow issuing from symmetric semi-infinitely long nozzle via the mathematical theory on the free boundary problem established before.

\subsection{Statement of the physical problem and main results}

The problem we address is that the flow of an incompressible,
inviscid fluid issues from a given symmetric semi-infinitely long
nozzle and emerges as a jet with two symmetric free boundaries (see
Figure \ref{f5}). The flow is assumed to be both steady and
irrotational. The free boundary $\Gamma$ initiates smoothly at the
end point of the nozzle wall and extends to infinity in downstream,
where the jet flow tends to some uniform flow.

\begin{figure}[!h]
\includegraphics[width=100mm]{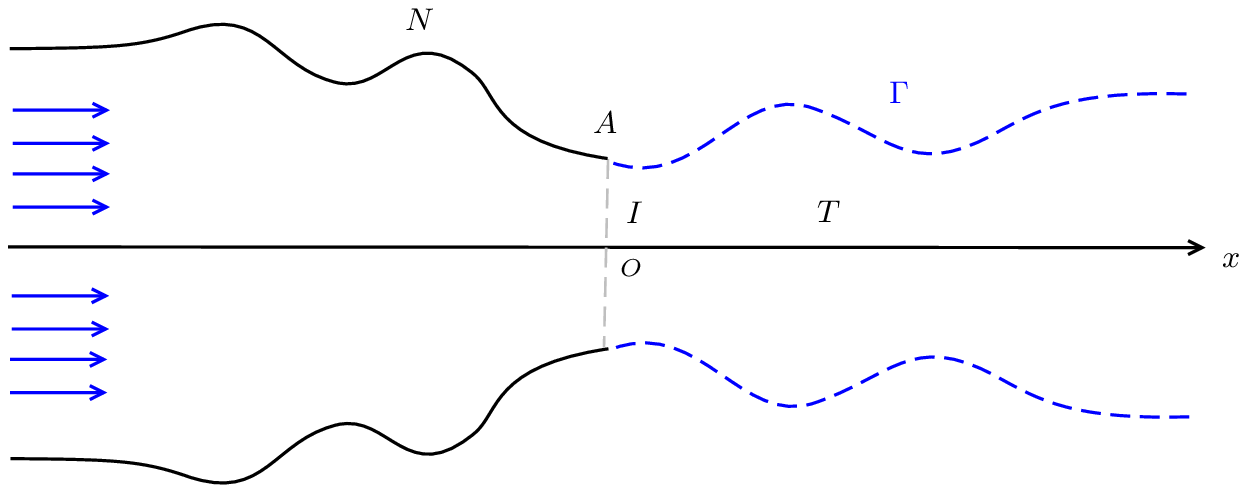}
\caption{Symmetric jet flow problem}\label{f5}
\end{figure}

Denote $N=\{y=g(x),x\in(-\infty,0]\}$ the upper nozzle wall of the semi-infinitely long plane symmetric nozzle, which satisfies that \be\label{g0}\text{ $g(x)\in C^{2,\alpha}((-\infty,0])$, $\lim_{x\rightarrow-\infty}g(x)=H$ and $g(0)=a=\min_{x\leq0}g(x)$.}\ee  Let $T=\{y=0\}$ be the symmetric axis, and $A=(0,a)$ be the end point of the nozzle wall.

The problem of incompressible jet flow with general vorticity can be formulated as the free boundary problem of finding a domain $\O_0$ in $(x,y)$-plane, whose boundary consists of the upper nozzle wall $N$, the symmetric axis $T$, and a priori unknown curve $\Gamma$ expressed by $y=k(x)$
for $x\geq 0$ with \be\label{g1}k(0)=g(0)=a\ \ \ \text{and}\ \ \ k'(0)=g'(0).\ee
The condition \eqref{g1} is the so-called continuous fit condition and smooth fit condition, respectively.

A vector $(u,v,p)$ representing the horizontal velocity, the vertical velocity and the pressure of the flow in $\O_0$, which belongs to $\left(C^{1,\alpha}(\O_0)\cap C^0(\bar\O_0)\right)^3$ and satisfies the steady incompressible Euler equations
\be\label{g2}\left\{\ba{ll} &\p_x u+\p_y v=0,\\
 &u\p_xu+v\p_yu+\p_x p=0,\\
 &u\p_xv+v\p_yv+\p_y p=0,\ea\right.\ee and the slip boundary condition
 \be\label{g3}(u,v)\cdot\vec{n}=0\ \ \ \text{on}\ \ \ N\cup  T,\ee where $\vec{n}$ is the normal direction of the boundary.

 The free boundary $\Gamma$ is assumed to be a material surface of the incompressible fluid, and then the velocity still satisfies the slip boundary condition \eqref{g3} on $\Gamma$. Moreover, the classical assumption (neglecting the effects of surface tension) is constant pressure condition, namely,
 $$p=p_{atm}\ \ \ \text{on}\ \ \ \Gamma.$$ Here, we denote $p_{atm}$ the constant atmosphere pressure.

In the upstream, we assume that
   \be\label{g4}u(x,y)\rightarrow u_0(y)\ \ \text{and} \ \ v(x,y)\rightarrow 0  \ \ \text{as}\ \ x\rightarrow-\infty,\ee and denote $p_{in}$ the constant champer pressure in the inlet of the nozzle. Here, the variation and amplitude of the horizontal velocity $u_0(y)$ are arbitrary, and it implies that the vorticity of the jet flow is arbitrary in the upstream.

    \begin{remark} Once we impose the horizontal velocity $u_0(y)$ in the inlet of the nozzle, the mass flux of incoming flow is determined by
    $$Q=\int_{0}^{H}u_0(y)dy.$$
    \end{remark}

   Before we state the well-posedness results on the incompressible jet problem, it should be noted that there are two invariant quantities for the steady incompressible inviscid fluid, i.e.,
   \be\label{a04}(u,v)\cdot\nabla
\o=0,\ee and \be\label{a05}
(u,v)\cdot\nabla\left(\f{1}{2}(u^2+v^2)+p\right)=0,\ee where $\o=v_x-u_y$ denotes the vorticity of the fluid in two dimensions. In particular, the relation \eqref{a05} gives that quantity $\f{1}{2}(u^2+v^2)+p$ called Bernoulli's function remains invariant along the each streamline, which tells us two facts,\\
(1) the speed remains a constant denoted as $\ld$ along the free boundary $\Gamma$.\\
(2) along the upper nozzle wall and the free boundary $\Gamma$,
$$\ld=\sqrt{u^2_0(H)+2(p_{in}-p_{atm})}$$ as long as the continuous fit condition of the free boundary holds, where $p_{in}$ denotes the uniform pressure in the upstream. Here, the quantity $p_{diff}=p_{in}-p_{atm}$ is nothing but the pressure difference between the upstream and the downstream. It should be noted that the quantity $p_{diff}$ is an undermined parameter here, and we will show that the appropriate choice $p_{diff}$ is guaranteed by the continuous fit condition.

The incompressible jet problem is stated as follows.

{\bf The incompressible jet flow problem.} Given a symmetric semi-infinitely long nozzle $N$, a horizontal velocity $u_0(y)$ in the inlet and a constant atmosphere pressure $p_{atm}$ on the free surface, does there exist a unique symmetric incompressible inviscid jet flow issuing from the nozzle $N$, and the free boundary $\Gamma$ initiates smoothly at the endpoint $A$.

Moreover, we introduce the definition of a solution to the incompressible jet flow problem in the following.

\begin{definition}\label{def2}
{\bf(A solution to the incompressible jet flow problem).}\\ Given an atmosphere pressure $p_{atm}$ and an incoming velocity
$u_0(y)$,  a vector $(u,v,p,\Gamma)$ is called a solution to the incompressible jet
flow problem, provided that \\
(1) The free boundary $\Gamma$ can be expressed by a $C^1$-smooth
 function $y=k(x)>0$ for any $x\in[0,+\infty)$, and there exists an appropriate pressure difference $p_{diff}$, such that $\Gamma$ satisfies the continuous and smooth fit condition \eqref{g1}.\\
(2)~$(u,v,p)\in\left(C^{1,\alpha}(\O_0)\cap C(\bar{\O}_0)\right)^3$
solves the Euler system \eqref{g2}, and satisfies the boundary
condition \eqref{g3}.\\
(3)~There exists a unique positive constant $h\in (0,a]$ such that
$$k(x)\rightarrow h\ \ \text{and}\ \ k'(x)\rightarrow 0\ \ \text{as $x\rightarrow+\infty$}.$$
where $h$ is the asymptotic height of the free boundary in downstream.\\
 (4) $p=p_{atm}$ on $\Gamma$.\\
 (5) $(u,v)\rightarrow(u_0(y),0)$ uniformly for $y\in(0,H)$, as $x\rightarrow-\infty$.
\end{definition}

The main results on the existence and uniqueness of the incompressible jet flow problem read as follows.

\begin{theorem}\label{th1}
Suppose that the nozzle wall $N$ satisfies the assumption condition \eqref{g0}. Assume that the horizontal velocity in upstream $u_0(y)\in
C^{2,\beta}([0,H])$ satisfies that
 \be\label{a08}
u_0(y)>0, \ u_0'(0)=0,\ \ \text{and}\ \ u_0''(y)\geq
0\ \ \text{for any $y\in[0,H]$}.\ee Then, there exist a unique difference pressure $p_{diff}\geq 0$ and a unique
solution
$(u,v,p,\Gamma)$ to the incompressible jet flow problem, such that \\
(1) The jet flow satisfies the following asymptotic behavior in the
far fields, $$(u,v,p)\rightarrow(u_0(y),0,p_{in}),\ \nabla
u\rightarrow(0,u_0'(y)),~~\nabla v\rightarrow 0,~~\nabla
p\rightarrow0,$$ uniformly in any compact subset of $(0,H)$, as
$x\rightarrow-\infty$, and
$$(u,v,p)\rightarrow(u_1(y),0,p_{atm}),\ \nabla u\rightarrow(0,u_1'(y)),~~\nabla v\rightarrow 0,~~\nabla
p\rightarrow0,$$ uniformly in any compact subset of $(0,h)$, as
$x\rightarrow+\infty,$ where $p_{in}=p_{diff}+p_{atm}$, and $u_1(y)$ and $h$ are uniquely
determined by $u_0(y)$, $p_{in}$ and $p_{atm}$.\\
 (2) $u>0$ in $\bar{\O}_0$.\\
 (3) $h\leq k(x)\leq \bar H=\max_{x\leq 0}g(x)$ for any $x\in[0,\infty)$.
\end{theorem}

\begin{remark} The assumption $u'_0(0)=0$ in \eqref{a08} follows from the symmetry of the incompressible jet flow. However, if we consider an asymmetric incompressible jet flow, the condition $u'_0(0)=0$ can be instead of $u'_0(H)\geq 0$.
\end{remark}
\begin{remark}\label{re2} To obtain the continuous fit condition, we choose the difference pressure $p_{diff}$ as a parameter, and then show that there exists a unique $p_{diff}$, such that the free boundary $\Gamma$ satisfies the continuous fit condition in \eqref{g1}.\end{remark}

Similarly, we can obtain the well-posedness results on the incompressible cavitational flow problem. We will give the statement of the physical problem and the well-posedness results as follows.

 Given a two-dimensional obstacle as
 \be\label{f01}N: y=g(x)\in C^{2,\alpha}((-b,0]), \ g(-b)=0\ \text{and $g(0)=a=\max_{-b\leq x\leq 0}g(x)$}.\ee
 and denote
 $T_0=\{(x,0)\mid-\infty<x<+\infty\}$ the asymmetric axis and $T_1=\{(x,H)\mid-\infty<x<+\infty\}$ (see Figure \ref{f6}).
 \begin{figure}[!h]
\includegraphics[width=100mm]{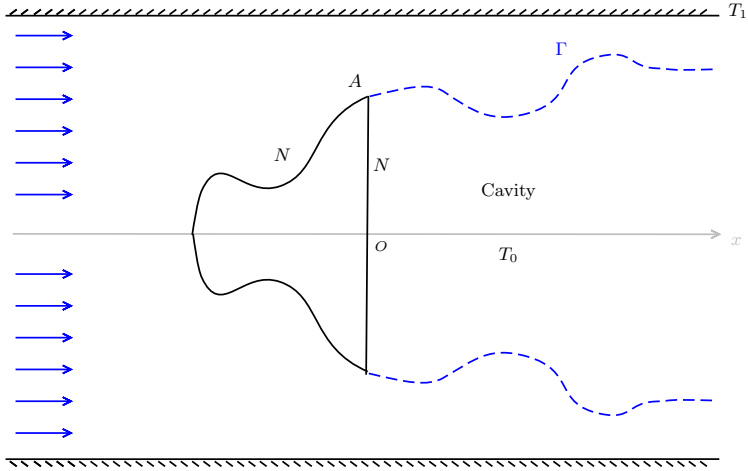}
\caption{Symmetric cavitational flow}\label{f6}
\end{figure}

{\bf The incompressible cavitational flow problem.} Given a
two-dimensional symmetric obstacle $N$, atmosphere pressure
$p_{atm}$ and the horizontal velocity $u_0(y)$ in the upstream, does
there exist a unique incompressible symmetric inviscid cavitational
flow past the given obstacle, and the free boundary $\Gamma$
initiates smoothly at the corner $A$ of the obstacle?

\begin{definition}\label{def5}
{\bf(A solution to the incompressible cavitational flow problem).}\\
A vector $(u,v,p,\Gamma)$ is called a solution to the incompressible cavitational flow
problem, provided that \\
(1) The free boundary $\Gamma$ can be expressed by a $C^1$-smooth
 function $y=k(x)>0$ for any $x\in[0,+\infty)$, such that $N\cup
 \Gamma$ is $C^1$, namely,
\be\label{f07}g(0)=k(0)=a~~\text{and}~~ g'(0)=k'(0).\ee
(2)~$(u,v,p)\in\left(C^{1,\alpha}(\O_0)\cap C(\bar{\O}_0)\right)^3$
solves the Euler system \eqref{g2}, and satisfies the boundary
condition \eqref{a04}, where
$\O_0$ is the flow field bounded by $N,T_0,T_1$ and $\Gamma$.\\
(3)~There exists a positive constant $h\in (a,H)$ such that
$$k(x)\rightarrow h\ \ \text{and}\ \ k'(x)\rightarrow 0\ \ \text{as}\ \ x\rightarrow+\infty.$$
where $h$ is the asymptotic height of the free boundary in downstream.\\
 (4) $p=p_{atm}$ on $\Gamma$.\\
 (5) $(u,v)\rightarrow(u_0(y),0)$ uniformly for $y\in(0,H)$, as $x\rightarrow-\infty$.
\end{definition}

We give the results as follows.

\begin{theorem}\label{th5}
Suppose that the solid wall $N$ satisfies the assumption
\eqref{f01}. Given an atmosphere pressure $p_{atm}$ and $u_0(y)\in
C^{2,\beta}([0,H])$, which satisfies that
 \be\label{f08}
u_0(y)>0,\ \ u_0'(0)=0\ \ \text{and}\ \ u_0''(y)\geq 0\ \ \text{for
any $y\in[0,H]$}.\ee Then, there exist a unique difference pressure
$p_{diff}=p_{up}-p_{atm}$ and a unique solution
$(u,v,p,\Gamma)$ to the cavitational flow problem, where $p_{in}$ denotes the pressure in the inlet of the channel. Furthermore, \\
(1) The cavitational flow satisfies the following asymptotic
behavior in the far fields, $$(u,v,p)\rightarrow(u_0(y),0,p_{in}),\
\nabla u\rightarrow(0,u_0'(y)),~~\nabla v\rightarrow 0,~~\nabla
p\rightarrow0,$$ uniformly in any compact subset of $(0,H)$, as
$x\rightarrow-\infty$, and
$$(u,v,p)\rightarrow(u_1(y),0,p_{atm}),\ \nabla u\rightarrow(0,u_1'(y)),~~\nabla v\rightarrow 0,~~\nabla
p\rightarrow0,$$ uniformly in any compact subset of $(h,H)$, as
$x\rightarrow+\infty,$ where $u_1(y)$ and $h$ are uniquely
determined by $u_0(y)$, $p_{in}$ and $p_{atm}$.\\
 (2) $u>0$ in $\bar{\O}_0$.\\
 (3) $k(x)>0$ for any $x>0$.
\end{theorem}

\subsection{Reformulation of the free boundary problem}

In the following, we will reformulate the original physical problem
into a free boundary problem of a semilinear elliptic equation. The
similar idea has been adapted in the compressible subsonic flows
with general vorticity in an infinitely long nozzle in
\cite{DX,DXX,XX3}.

First, it follows from the continuity equation in the incompressible
Euler system \eqref{g2} that there exists a stream function $\psi$,
such that \be\label{b01}u=\psi_y\ \ ~~\text{and}~~\ \ v=-\psi_x.\ee

Second, suppose that the streamlines are well-defined in the whole
flow fluid, and thus the incompressible jet flow problem can be
solved along the each streamline as follows. On another hand, the
positivity of the horizontal velocity of the jet will be verified
later, which gives the well-definedness of the streamlines in the
whole flow fluid.

\begin{figure}[!h]
\includegraphics[width=130mm]{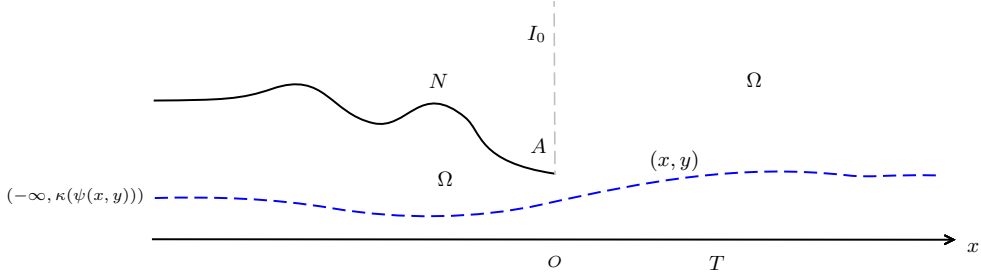}
\caption{The streamline in the fluid field}\label{f7}
\end{figure}

Denote $\O=\{-\infty<x\leq 0,0<y<g(x)\}\cup\{x>0,y>0\}$ the possible flow field (see Figure \ref{f7}) and $\O_0=\O\cap\{0<\psi<Q\}$ the flow field.
 For any point $(x,y)\in\O_0$, it can be pulled back along one
streamline to the point $(-\infty,\kappa(\psi(x,y)))$ in the inlet. Thus
$$\psi(x,y)=\int_0^{\kappa(\psi(x,y))}u_0(s)ds,$$ which implies that
\be\label{b02}\left\{\ba{ll}&1=u_0(\kappa(\psi))\kappa'(\psi),\\
&\kappa(0)=0.\ea\right.\ee It's easy to see that the function
$\kappa(t)$ can be solved uniquely by \eqref{b02} provided that
$u_0(y)$ is $C^{2,\beta}$-smooth. Meanwhile, it's clear that
$\kappa(Q)=H$. Due to the fact that the vorticity $\o$ is invariable
on each streamline, we obtain the governing equation to the stream
function,
$$\ba{rl}-\Delta\psi=\o=-u'_0(\kappa(\psi))\triangleq f_0(\psi)\ \ \ \ \text{in the fluid field $\O_0$}.\ea$$
Furthermore, it is not difficult to check the following facts \be\label{b05}\ba{rl}&\text{(1) $f_0(t)$ is $C^{1,\beta}$,}\\
&\text{(2) $f_0(0)=0,
f_0(Q)=-u_0'(H)\leq 0$},\\
&\text{(3) $f_0'(t)=-\f{u_0''(\kappa(t))}{u_0(\kappa(t))}\leq 0$}, \ea\ee  provided that $u_0(y)$ satisfies the assumption \eqref{a08}.

On another hand, we can impose the
Dirichlet boundary conditions,
$$\psi=Q~~\text{on}~~N\cup I_0,~~\text{and} ~\psi=0 ~\text{on}~
T,$$ where $I_0=\{(0,y)\mid y\geq a\}$. Moreover, the free boundary
can be defined as
\be\label{b04}\Gamma=\O\cap\{x>0\}\cap\p{\{\psi<Q\}}.\ee  The
constant pressure condition on the free boundary together with the
Bernoulli's law gives that the speed remains a constant on $\Gamma$,
denote $\ld$ the constant speed, i.e.,
$$|\nabla\psi|=\ld=\sqrt{u^2_0(H)+2p_{diff}}\ \ \text{on
$\Gamma$}.$$ Clearly, the constant $\ld$ is determined uniquely by
$p_{diff}\geq0$.

Therefore, we formulate the following free boundary problem
of the stream function that
\be\label{b06}\left\{\ba{ll}&-\Delta\psi=f_0(\psi)~~~~\ \ \ \text{in}~~\O\cap\{\psi<Q\},\\
&\psi=Q~~~~\ \ \ \text{on}~~N\cup \Gamma,\ \ \ \ \ \psi=0~~~~\text{on}~~T,\\
&|\g\psi|=\lambda~~~~\ \ \ \text{on}~~\Gamma.\ea\right.\ee

\subsection{Uniqueness of asymptotic behavior in downstream}

Next, we will show that the asymptotic behavior of the jet flow in downstream can be determined by the state $(u_0(y),0)$ of the incoming flow and the difference pressure $p_{diff}$.

For any initiate point $(-\infty,s)$ in the inlet, the
well-definedness of the streamlines implies that there exists a
unique point denoted  $(0,\chi(s;p_{diff}))$ in the downstream for
any $s\in[0,H]$ (see Figure \ref{f8}), which can be pulled back to
the imposed initiate point $(-\infty,s)$ along a streamline.

\begin{figure}[!h]
\includegraphics[width=130mm]{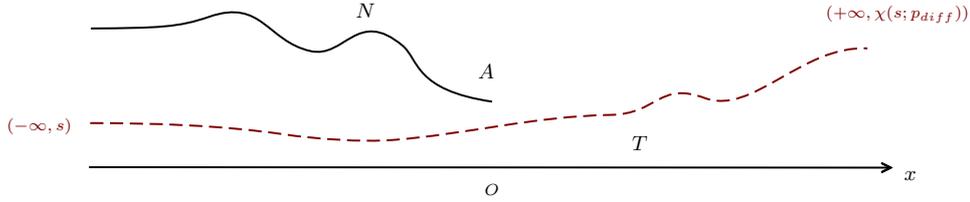}
\caption{The streamline in $\O_0$}\label{f8}
\end{figure}

Denote $u_1(y)$ the horizontal velocity in the downstream, the
conservation of mass and Bernoulli's law gives that
$$\int_{0}^{s}u_0(t)dt=\int_0^{\chi(s;p_{diff})}u_1(t)dt,$$ and
$$\f{u^2_0(s)}{2}+p_{diff}=\f{u^2_1(\chi(s;p_{diff}))}{2}\ \ \text{in the fluid field}.$$ Here, the difference pressure $p_{diff}\geq0$ between the inlet and the downstream is regarded as a parameter. These also give the initial value problem to the function $\chi(s;p_{diff})$ as follows,
\be\label{a09}\left\{\begin{array}{ll}&\f{d\chi(s;p_{diff})}{ds}=\f{u_0(s)}{\sqrt{u_0^2(s)+2p_{diff}}}>0,\\
&\chi(0;p_{diff})=0,\end{array}\right.\ee for any $s\in[0,H]$ and $p_{diff}\geq0$. Thus,
\be\label{a010}\chi(s;p_{diff})=\int_0^s\f{u_0(t)}{\sqrt{u_0^2(t)+2p_{diff}}}dt\leq
s.\ee  It is easy to check that $\chi(s;p_{diff})$ is strictly decreasing with respect to the parameter $p_{diff}$. In particular, the asymptotic height of the jet in downstream is in fact $h=\chi(H;p_{diff})$.

Noting that $\f{d\chi(s;p_{diff})}{d
s}>0$ for any $s\in[0,H]$ and $p_{diff}\geq0$, thus there exists an inverse function $\chi^{-1}(t;p_{diff})$, such that $s=\chi^{-1}(t;p_{diff})$ for any $t\in[0,h]$. Then the horizontal velocity $u_1(y)$ in downstream can be solved uniquely by
\be\label{a012}u_1(t)=\sqrt{u_0^2(\chi^{-1}(t;p_{diff}))+2p_{diff}}\
\ \text{for any $t\in[0,h]$},\ee once $\chi(s;p_{diff})$ is solved by the initiate value problem \eqref{a09} for any $p_{diff}\geq0$.

\begin{remark}\label{re4} For any $p_{diff}\geq 0$, it is easy to check that
$$\text{$\f{d\chi(H;p_{diff})}{d
p_{diff}}<0$ and}\ \ \ld=\sqrt{u^2_0(H)+2p_{diff}}\geq
u_0(H)\triangleq \ld_0\geq 0.$$ Moreover, the asymptotic height
$h=\chi(H;p_{diff})=\chi\left(H;\f{\ld^2-u^2_0(H)}{2}\right)=
h(\ld)\triangleq h_{\ld}$ is strictly increasing with respect to
$\ld$, and $h_{\ld_0}=H$.
\end{remark}

\subsection{The variational approach}

To solve the free boundary problem \eqref{b06}, we introduce a
variational problem with a parameter $\ld\geq\ld_0$. For any $L>\bar
H=\max_{x\leq0}g(x)$, denote $\O_L$ the truncated domain (see Figure
\ref{f10}),
$$\O_{L}=\O\cap\{(x,y)\mid-L<x<L,y<L\},\ \ \ \ D_L=\O_L\cap\{x>0\},\ \ \ N_L=N\cap\{x>-L\},$$ and
 $$\sigma_{-L}=\{(-L,y)|~0\leq y\leq
g(-L)\}\ \ \ \text{and}\ \ \ \ \sigma_L=\{(L,y)|~0\leq y\leq L\},$$
and
$$T_L=T\cap\{-L\leq x\leq L\},\ I_{0,L}=\{(0,y)\mid a\leq y\leq L\} \ \ \text{and} \ \ l_{L}=\p B_{L/2}\left(\left(L/2,L\right)\right)\cap\{y\geq L\}.$$

\begin{figure}[!h]
\includegraphics[width=120mm]{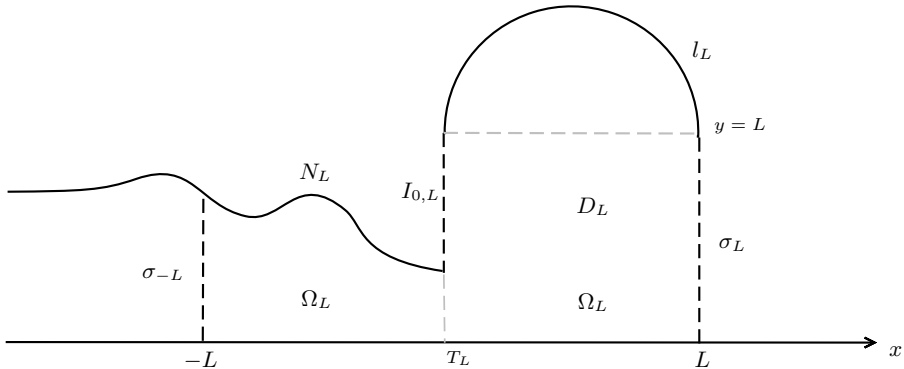}
\caption{The truncated domain $\O_L$}\label{f10}
\end{figure}

Since, a priorily, one does not know whether the stream function
$\psi$ satisfies that $0\leq\psi\leq Q$. Hence, we need extend the
function $f_0(t)$ as follows
\be\label{bb06}\t f_0(t)=\left\{\ba{ll}f_0(Q)+\f{f_0'(Q)}2~~~~&\text{if}~~ t\geq Q+1,\\
f_0(Q)+f_0'(Q)\left(t-Q-\f{(t-Q)^2}2\right)~~~~&\text{if}~~Q\leq t\leq Q+1,\\
f_0(t)~~~~&\text{if}~~0\leq t\leq Q,\\
f_0'(0)\left(t+\f{t^2}2\right)~~~~&\text{if}~~-1\leq t\leq 0,\\
-\f{f_0'(0)}2~~~~&\text{if}~~t\leq -1.\ea\right.\ee

Set \be\label{b07}\text{$F_0(t)=2\int_t^{Q}\t f_0(t)ds$\ \ \ and \ \
\ $f(t)=-\t f_0(Q-t)$}.\ee

For $y\in[0,g(-L)]$, define a function $\Psi_{-L}(y)$, which
satisfies that
$$\Psi_{-L}''(y)=-\t f_0(\Psi_{-L}(y)), \ \Psi_{-L}(0)=0,\ \Psi_{-L}(g(-L))=Q.$$ By virtue of \eqref{b05}, it is easy to check that
$$0<\Psi_{-L}(y)<Q\ \ \text{and}\ \ \Psi'_{-L}(y)\geq 0\ \ \text{for any $y\in(0,g(-L))$.}$$
 Firstly, we give the following functional that
$$J_{\ld,L}(\psi)=\int_{\O_L}\left(|\nabla\psi|^2+F_0(\psi)+\ld^2I_{\{\psi<Q\}\cap D_L}\right)dX,$$
where the admissible set is defined as follows
$$\ba{rl}K_{\ld,L}=\{\psi\in H^1_{loc}(\mathbb{R}^2)\mid &\psi=0~~\text{lies below}~T,
\psi=Q~~\text{lies above}~N\cup I_{0,L}\cup l_{L},\\
&\psi=\Psi_{-L}~~~\text{on}~~\sigma_{-L},~~~\psi=\min\{\Psi_\ld,Q\}~~~\text{on}~~\sigma_L\},\ea$$
and \be\label{b08}
\Psi_\ld(y)=\int_0^yu_1(s)ds.\ee

{\bf Truncated variational problem $(P_{\ld,L})$:} For any $L>\bar H$ and
$\ld\geq\ld_0$, find a $\psi_{\ld,L}\in K_{\ld,L}$ such that
$$J_{\lambda,L}(\psi_{\ld,L})=\min_{\psi\in K_{\ld,L}} J_{\lambda,L}(\psi).$$


\begin{remark} Set $\phi=Q-\psi$, by virtue of \eqref{b05} and
\eqref{b07}, the variational problem $(P_{\ld,L})$ can be
transformed into the variational problem in Section 2, and $f(t)$
satisfies the condition \eqref{c0}. Thus, we can use the results in
Section 2 and Section 3 in the following.
\end{remark}

\begin{proposition}\label{ld0} For any $\ld\geq\ld_0$ and $L>\bar H$, there exists a minimizer $\psi_{\ld,L}$ to the variational problem $(P_{\ld,L})$ with  $0\leq\psi_{\ld,L}(x,y)\leq Q$ in $\O_L$.
\end{proposition}

\begin{proof} The existence of the minimizer to the variational problem $P_{\ld,L}$ follows from Lemma \ref{lb3}, and denote $\psi_{\ld,L}$ the minimizer to the variational problem $(P_{\ld,L})$.

Denote $\psi=\psi_{\ld,L}$ and $\psi^\e=\psi-\e\min\{\psi,0\}$ for any $\e\in(0, 1)$. Thus, $\psi^\e\in K_{\ld,L}$ and $\psi^\e\geq\psi$. Moreover,
$\psi<Q$ if and only if $\psi^\e<Q$ in $\O_L$, and we have
$$\ba{rl}0\leq&J_{\ld,L}(\psi^\e)-J_{\ld,L}(\psi)\\
 \leq&\int_{\O_L\cap\{\psi<0\}}((1-\e)^2-1)|\nabla\psi|^2-\e
 F'_0((1-\e)\psi)\psi
dX\\
\leq&0, \ea$$ due to $F'(t)\leq 0$ for $t\leq0$, which implies that
$$\text{$\psi(x,y)\geq 0$  in $\O_L$}.$$
Similarly, we can show that
$$\text{$\psi(x,y)\leq Q$  in $\O_L$}.$$

\end{proof}

Since $0\leq\psi_{\ld,L}\leq Q$ in $\O_L$, we can remove the truncation of the function $\t f_0(t)$ in \eqref{bb06}. Denote the free boundary of $\psi_{\ld,L}$ as
$$\Gamma_{\ld,L}=D_L\cap\p\{\psi_{\ld,L}<Q\}.$$ With the aid of Lemma \ref{lb3} and Theorem \ref{lc15}, we have
\begin{proposition}\label{ld1} The minimizer $\psi_{\ld,L}$ satisfies that \\
(1) $\psi_{\ld,L}\in
C^{0,1}(\O_L)$.\\
(2) $\Delta\psi_{\ld,L}+f_0(\psi_{\ld,L})=0$ in
$\O_L\cap\{\psi_{\ld,L}<Q\}$ and $\Delta\psi_{\ld,L}+
f_0(\psi_{\ld,L})\leq0$ in $\O_L$ in the weak sense. Furthermore, $\psi_{\ld,L}>0$ in $\O_L$ and
$\psi_{\ld,L}\in C^{2,\alpha}(G)$ for any compact subset $G$ of
$\O_L\cap\{\psi_{\ld,L}<Q\}$,
$\alpha\in(0,1)$.\\
(3) The free boundary of the minimizer $\psi_{\ld,L}$ is $C^{3,\alpha}$.\\
(4) $|\g\psi_{\ld,L}|=\ld$ on the free boundary $\Gamma_{\ld,L}$.
Moreover, $|\g\psi_{\ld,L}|\geq\ld$ on the segment $l\in
I_{0,L}\cap\p\{\psi_{\ld,L}<Q\}$ and  $|\g\psi_{\ld,L}|\leq\ld$ on
the segment $l\in \{x=0,0<y<a\}\cap\p\{\psi_{\ld,L}<Q\}$.
\end{proposition}

The Statement (1) in Proposition \ref{ld1} gives the Lipschitz continuity of the minimizer in the interior of $\O_L$. Next, we will show that the minimizer $\psi_{\ld,L}(x,y)$
is also Lipschitz continuous near the boundary of $\O_L$.

\begin{lemma}\label{ld2} $\psi_{\ld,L}(x,y)$ is Lipschitz continuous
in every compact subsect of $\bar{\O}_L$ that does not contain the point $A$.
\end{lemma}
\begin{proof}

We first consider the Lipschitz continuity of $\psi_{\ld,L}$ near
the boundary $\p D_L$. Denote $\phi(X)=Q-\psi_{\ld,L}(X)$ and
$d(X)=\text{dist}(X,\Gamma_{\ld,L})$ and $d_1(X)=\text{dist}(X,\p
D_L)$ for any $X\in D_L$. We consider the following two cases.

{\bf Case 1.} $d(X)\leq d_1(X)$, we can obtain the Lipschitz
continuity of $\phi$ at $X$ by using the similar arguments in the proof
Theorem \ref{lb5}.

{\bf Case 2.} $d(X)> d_1(X)$. Without loss of generality, we assume
that $d_1(X)=|X-X_0|$ for $X_0=(0,y)\in I_{0,L}$ and $X=(x,y)\in
D_L$. Set $s_0=\min\left\{\f12,y-a\right\}$ and
$B_{s_0}=B_{s_0}(X_0)$. Let $\varphi$ be a solution to the following
boundary problem,
$$\left\{\begin{array}{ll}&\Delta\varphi+f(\varphi)=0\ \ \text{in}\ B_{s_0}\cap\{x>0\},\\
&\varphi=0\ \text{on}\ B_{s_0}\cap\{x=0\},\ \ \varphi=Q\ \text{on}\
\p B_{s_0}\cap\{x>0\}.\end{array}\right.$$ Since $\phi\leq
\varphi$ on $\p(B_{s_0}\cap\{x>0\})$, the maximum principle gives
that
$$\phi\leq\varphi\ \ \text{in $B_{s_0}\cap\{x>0\}$}.$$

Denote $\t\varphi(\t X)=\varphi(X_0+s_0\t X)$ for $\t X=(\t x,\t
y)$, and $\Delta\t\varphi+s_0^2f(\t\varphi)=0$ in $B_1(0)\cap\{\t
x>0\}$. By using the elliptic estimates in \cite{GT}, we have
$$\t\varphi(\t X)\leq C\t x\ \ \text{in $B_{\f12}(0)\cap\{\t
x>0\}$},$$ which implies that
\be\label{b09}\phi(X)\leq\varphi(X)\leq C\f{x}{s_0}\ \ \text{in
$B_{\f{s_0}2}(X_0)\cap\{ x>0\}$}.\ee

If $r=x<\f{s_0}{4}$, we have
$$\phi>0\ \ \text{and}\ \ \Delta\phi+f(\phi)=0\ \ \text{ in $B_r(X)\subset
B_{\f{s_0}2}\cap\{x>0\}$}.$$ Denote $\t\phi(\t X)=\f{\phi(X+r\t
X)}{r}$ for $\t X=(\t x,\t y)$, it follows from \eqref{b09} that
$$0<\t\phi<\f{C}{s_0}\ \ \text{and}\ \ \Delta\t\phi+rf(r\t\phi)=0\ \ \text{ in
$B_1(0)$}.$$ By using the elliptic estimates in \cite{GT}, we have
$$|\g\phi(X)|=|\g\t\phi(0)|\leq C(\|\t\phi\|_{L^\infty(B_1(0))}+r\|f(r\t\phi)\|_{L^{\infty}(B_1(0))})\leq\f{C}{s_0}+Cs_0.$$

If $r=x\geq\f{s_0}{4}$, the elliptic estimate in \cite{GT} gives
that $|\g\phi(X)|\leq C$.

Since $\Delta\psi_{\ld,L}+f_0(\psi_{\ld,L})=0$ in $\O_L\setminus D_L$,
the Lipschitz continuity of $\psi_{\ld,L}$ near $\p(\O_L\setminus
D_L)$ can be obtained by using elliptic regularity.

\end{proof}

\subsection{The uniqueness, monotonicity and free boundary of the minimizer}
We first give the uniqueness of the minimizer $\psi_{\ld,L}(x,y)$
for any given $\ld\geq\ld_0$ and $L>\bar H$, and show that
$\psi_{\ld,L}(x,y)$ is monotone increasing with respect to $y$.

Once we have the regularity of the free boundary $\Gamma_{\ld,L}$ at hand, we will establish some topological properties of the free boundary, provided that some special geometric conditions on the solid boundaries are assumed. For example, it's desired that the free boundary is $x$-graph, as long as we assume that the nozzle wall $N$ is a $x$-graph and the horizontal velocity of the incoming flow is positive.

To obtain this topological property of the free boundary, we will
establish the monotonicity of the minimizer $\psi_{\ld,L}(x,y)$ with
respect to $y$ first.
\begin{lemma}\label{ld3} The minimizer of the truncated variational problem
$(P_{\ld,L})$ is unique and $\psi_{\ld,L}(x,y)$ is monotone
increasing with respect to $y$.
\end{lemma}
\begin{proof} Assume that $\psi_{\ld,L}$ and $\t\psi_{\ld,L}$ are two minimizers to the truncated variational problem
$(P_{\ld,L})$. Set $\psi_{\ld,L}^{\e}(x,y)=\psi_{\ld,L}(x,y-\e)$ in
$\O_L^\e=\{(x,y)\mid(x,y-\e)\in\O_L\}$ for
any $\e>0$.

Noticing that $\psi_{\ld,L}^{\e}$ is a minimizer to the functional
$$J_{\ld,L}^\e(\psi)=\int_{\O^\e_L}\left(|\nabla\psi|^2+F_0(\psi)+\ld^2I_{\{\psi<Q\}\cap D^\e_L}\right)dX$$ and the admissible set $K^{\e}_{\ld,L}=\{\psi\mid
\psi(x,y-\e)\in K_{\ld,L}\}$, where $D_L^\e=\O_L^\e\cap\{x>0\}$.
Extend $\psi_{\ld,L}^\e=0$ in $\O_L\setminus\O^\e_L$ and $\t\psi_{\ld,L}=Q$ in
$\O^\e_L\setminus\O_L$.

Denote $\t\psi=\t\psi_{\ld,L}$ and $\psi^\e=\psi_{\ld,L}^\e$ for
simplicity, it follows from Proposition \ref{ld0} that
$$\phi_1=\min\{\psi^{\e},\t\psi\}\in K^\e_{\ld,L}\ \
\text{and}\ \ \phi_2=\max\{\psi^{\e},\t\psi\}\in K_{\ld,L},$$ and
\be\label{bbb}
\O_L\cap\{\psi^{\e}>\t\psi\}=\O_L^\e\cap\{\psi^{\e}>\t\psi\}.\ee

Next, we will show that
\be\label{b010}J_{\ld,L}^\e(\phi_1)=J_{\ld,L}^\e(\psi^\e)\ \
\text{and}\ \ J_{\ld,L}(\phi_2)=J_{\ld,L}(\t\psi).\ee
Since $\t\psi$ and $\psi^\e$ are minimizers, it suffices to show
that
\be\label{bb0}J_{\ld,L}^\e(\phi_1)+J_{\ld,L}(\phi_2)=J_{\ld,L}(\t\psi)+J_{\ld,L}^\e(\psi^\e).\ee

It follows from \eqref{bbb} that
\be\label{bb1}\ba{rl}&\int_{\Omega_{L}^{\e}}|\nabla
\phi_1|^2dX+\int_{\Omega_{L}}|\nabla
\phi_2|^2dX\\
=&\int_{\Omega_{L}^{\e}\cap\{\psi^\e\leq\t\psi\}}|\nabla
\psi^\e|^2dX+\int_{\Omega_{L}^\e\cap\{\psi^\e>\t\psi\}}|\nabla
\t\psi|^2dX\\
&+\int_{\Omega_{L}\cap\{\psi^\e\leq\t\psi\}}|\nabla
\t\psi|^2dX+\int_{\Omega_{L}\cap\{\psi^\e>\t\psi\}}|\nabla
\psi^\e|^2dX\\
=&\int_{\Omega_{L}^{\e}\cap\{\psi^\e\leq\t\psi\}}|\nabla
\psi^\e|^2dX+\int_{\Omega_{L}\cap\{\psi^\e>\t\psi\}}|\nabla
\t\psi|^2dX\\
&+\int_{\Omega_{L}\cap\{\psi^\e\leq\t\psi\}}|\nabla
\t\psi|^2dX+\int_{\Omega_{L}^\e\cap\{\psi^\e>\t\psi\}}|\nabla
\psi^\e|^2dX\\
=&\int_{\Omega_{L}}|\nabla \t\psi|^2dX+\int_{\Omega_{L}^\e}|\nabla
\psi^\e|^2dX. \ea\ee Similarly, we can verify that
$$\int_{\Omega_{L}^{\e}}F_0(
\phi_1)dX+\int_{\Omega_{L}}F_0(
\phi_2)dX=\int_{\Omega_{L}^{\e}}F_0(\psi^\e)dX+\int_{\Omega_{L}}F_0(\t\psi)dX,$$
and
$$\ba{rl}\int_{D_{L}^\e}I_{\{\phi_1<Q\}}dX+\int_{D_{L}}I_{\{\phi_2<Q\}}dX=\int_{D_{L}^\e}I_{\{\psi^\e<Q\}}dX+\int_{D_{L}}I_{\{\t\psi<Q\}}dX,\ea$$
which together with \eqref{bb1} yield \eqref{bb0}.

 Next, we claim that
\be\label{b011}\text{if $\psi^\e(X_0)=\t\psi(X_0)<Q$ for
$X_0\in\O_L$, then either $\psi^\e\geq\t\psi$ or $\psi^\e\leq\t\psi$
in $B_r(X_0)$},\ee for small $r>0$. Suppose that the claim \eqref{b011} is not
true, the continuity of $\t\psi$ and $\psi^\e$ give that
$0<\t\psi<Q$ and $0<\psi^\e<Q$ in $B_r(X_0)$ for small $r>0$, then
we have that $\phi_2$ is not a solution of
$\Delta\phi_2+f_0(\phi_2)=0$ in $B_r(X_0)$. In fact, if not, the
maximum principle gives that $\max\{\t\psi,\psi^\e\}=\phi_2=\t\psi$
in $B_r(X_0)$, due to $\phi_2(X_0)=\t\psi(X_0)$. This contradicts to
our assumptions.

Let $\phi$ be the solution of the following boundary value problem
$$\left\{\begin{array}{ll}\Delta\phi+f_0(\phi)=0\ \ &\text{in}\ B_r(X_0),\\
\phi=\phi_2\ \ &\text{on}\ \p B_r(X_0).\end{array}\right.$$ Thus
$\phi\neq\phi_2$ in $B_r(X_0)$, it is easy to check that
\be\label{b012}\ba{rl}&\int_{B_r(X_0)}|\g\phi|^2+F_0(\phi)dX-\int_{B_r(X_0)}|\g\phi_2|^2+F_0(\phi_2)dX\\
\leq&-\int_{B_r(X_0)}|\g(\phi-\phi_2)|^2dX <0.\ea\ee Extend
$\phi=\phi_2$ in $\O_L\setminus B_r(X_0)$, it follows from
\eqref{b010} and \eqref{b012} that
$$J_{\ld,L}(\phi)<J_{\ld,L}(\phi_2)=J_{\ld,L}(\t\psi),$$ which
contradicts to the minimality of $\t\psi$.

Since $\t\psi>\psi^{\e}$ near $T_L$, it follows from \eqref{b011} that
$\t\psi\geq\psi^\e$ in the connected component $\O_0$ of $\{\t\psi<Q\}$
which contains an $\O_L$-neighborhood of $T_L$. In view of the
boundary value of $\t\psi$, we have that $\{\t\psi<Q\}\cap\p\O_L$ is
a connected arc. The maximum principle gives that any component of
$\{\t\psi<Q\}$ must touch the boundary $\{\t\psi<Q\}\cap\p\O_L$.
Therefore, the domain $\O_L\cap\{\t\psi<Q\}$ is connected and
\be\label{b013}\psi(x,y-\e)=\psi^{\e}(x,y)\leq\t\psi(x,y)\
\text{in}~~\O_L.\ee

Similarly, we can show that
\be\label{b014}\psi(x,y+\e)\geq\t\psi(x,y)~~\text{in}~~\O_L.\ee
Taking $\e\rightarrow0$ in \eqref{b013} and \eqref{b014}, which
yields that
$$\psi(x,y)=\t\psi(x,y)~~\text{in}~~\O_L.$$

In particular, taking $\psi=\t\psi$ in \eqref{b013}, we have that
the minimizer $\psi=\psi_{\ld,L}(x,y)$ is monotone increasing with
respect to $y$.

\end{proof}

The monotonicity of $\psi_{\ld,L}(x,y)$ with respect to $y$ implies that the free boundary $\Gamma_\ld$ is a $x$-graph, and then there exists a
function $y=k_{\ld,L}(x)$ for $x\in(0,L)$, such that
 $$\{\psi_{\ld,L}<Q\}\cap
D_L=\{(x,y)\in D_L\mid 0<y<k_{\ld,L}(x),0<x<L\}.$$

To obtain the continuity of $k_{\ld,L}(x)$, we first give the
following non-oscillation lemma.

\begin{lemma}\label{ld4}(Non-oscillation Lemma) Suppose that there exist some constants $\alpha_1,\alpha_2$ with $\alpha_1<\alpha_2$ and a domain
 $D\subset D_L\cap\{\psi_{\ld,L}<Q\}$ with dist$(A, D)>c_0$ for some $c_0>0$, which is bounded by two disjointed arcs $\gamma_1, \gamma_2$
($\gamma_1,\gamma_2\subset\Gamma_{\ld,L})$, the lines
$\{y=\alpha_1\}$ and $\{y=\alpha_2\}$. Denote $(\beta_i,\alpha_1)$
and $(\eta_i,\alpha_2)$ the endpoints of $\gamma_i$  for $i=1,2$.
Then there exists a constant $C>0$, such that
$$\alpha_2-\alpha_1\leq C\max\{|\beta_1-\beta_2|,|\eta_1-\eta_2|\}.$$

\end{lemma}

\begin{proof}Denote $h=\max\left\{|\beta_1-\beta_2|,|\eta_1-\eta_2|\right\}$ and $\psi=\psi_{\ld,L}$. Since $\Delta\psi=-f_0(\psi)$ in $D$, we have
$$\int_{\p D}\f{\p\psi}{\p\nu}dS=-\int_{D}f_0(\psi)dX\leq\int_{D}-f_0(Q)+CQdX\leq C h(\alpha_2-\alpha_1)\leq Ch,$$
which implies that \be\label{b015}\int_{\gamma_1\cup\gamma_2}\ld
dS=\int_{\gamma_1\cup\gamma_2}\f{\p\psi}{\p\nu}dS\leq\int_{\p
D\cap(\{y=\alpha_1\}\cup\{y=\alpha_2\})}\left|\f{\p \psi} {\p
x}\right|dy+Ch\leq Ch,\ee  where we have used the Lipschitz
continuity of $\psi$ due to Lemma \ref{ld2}.

 On the other hand, one gets
$$\int_{\gamma_1\cup\gamma_2}\ld dS\geq 2\ld(\alpha_2-\alpha_1),$$
which together with \eqref{b015} gives that
$$\alpha_2-\alpha_1\leq Ch.$$

\end{proof}

With the aid of Lemma \ref{ld4}, we have

\begin{lemma}\label{ld5}
$k_{\ld,L}(x)$ is a continuous function in $[0,L]$.
\end{lemma}
\begin{proof} By using the non-oscillation Lemma \ref{ld4}, it follows from the similar arguments in Lemma
5.4 in \cite{ACF3} that $k_{\ld,L}(x)$ has at most one limit point
as $x\downarrow 0$ and $x\uparrow L$. Moreover, $k_{\ld,L}(x)$ has most one limit
point as $x\uparrow x_0$ or as $x\downarrow x_0$ for any
$x_0\in(0,L)$. It suffices to show that
$$k_{\ld,L}(x+0)=k_{\ld,L}(x-0)=k_{\ld,L}(x)\ \ \text{for any $x\in(0,L)$}.$$
If not, without loss of generality, we assume that there exists a
point $x_0\in(0,L)$, such that
$k_{\ld,L}(x_0)<k_{\ld,L}(x_0+0)$. Denote
$I_\delta=\{(x_0,y)\mid
k_{\ld,L}(x_0)+\delta<y<k_{\ld,L}(x_0)+3\delta\}$ with
$\delta=\f{k_{\ld,L}(x_0+0)-k_{\ld,L}(x_0)}{4}$. The
monotonicity of $\psi_{\ld,L}(x,y)$ with respect to $y$ and the
Lipschitz continuity of $\psi_{\ld,L}(x,y)$ give that
$$\psi_{\ld,L}=Q\ \ \text{on}\ I_\delta,\ \ \text{$\Delta\psi_{\ld,L}+f_0(\psi_{\ld,L})=0$ and $u=\f{\p\psi_{\ld,L}}{\p y}\geq 0$ in
$E_{\delta,\e}$,}$$ where $E_{\delta,\e}=\{(x,y)\mid x_0<x<x_0+\e,
k_{\ld,L}(x_0)+\delta<y<k_{\ld,L}(x_0)+3\delta\}$
for small $\e>0$. Thus, $I_\delta$ is a part of the free boundary
$\Gamma_{\ld,L}$, it follows from (4) in Proposition \ref{ld1} that
\be\label{b150}\f{\p\psi_{\ld,L}(x_0+0,y)}{\p
x}=|\g\psi_{\ld,L}|=\ld\ \ \text{on}\ \ I_\delta.\ee It is easy to
check that $$\Delta u+b(X)u=0\ \ \text{in}\ \ E_{\delta,\e} \ \
\text{and}\ \ b(X)=f'_0(\psi_{\ld,L}(X))\leq 0.$$ The strong maximum
principle gives that $u>0$ in $E_{\delta,\e}$. In view of $u=0$ on
$I_\delta$, thanks to Hopf's lemma, we have
\be\label{b151}\f{\p^2\psi_{\ld,L}}{\p x\p y}=\f{\p u}{\p x}>0\ \
\text{on}\ \ I_\delta.\ee On the other hand, it follows from
\eqref{b150} that
$$\f{\p^2\psi_{\ld,L}}{\p y\p x}=0\ \ \text{on}\ \
I_\delta,$$ which contradicts to \eqref{b151}.

Hence, we obtain the continuity of the function $k_{\ld,L}(x)$.


\end{proof}

Next, we will introduce the bounded gradient lemma for
$\psi_{\ld,L}$ near the boundary of $D_{L}$.

\begin{lemma}\label{ld6} For any $X_0\in\p D_L$, if $B_R(X_0)\cap\Gamma_{\ld,L}\neq\varnothing$ for any $R>0$, then for small $r>0$, there exists a
constant $C$ independent of $Q$ and $r$, such that
\be\label{b152}|\g\psi_{\ld,L}(X)|\leq C\ \ \text{in
$B_r(X_0)\cap\O_L$}.\ee
\end{lemma}
\begin{proof}
We first consider the case $X_0=A$. It suffices to show
that
$$|\g\psi_{\ld,L}(X)|\leq C\ \ \text{in
$\O_L\cap\{\rho<|X-A|<2\rho\}$},$$ for any small $\rho>0$.

Denote $\psi_\rho(\t X)=\f1\rho \psi_{\ld,L}(A+\rho \t X)$,
$D_1=\left\{\t X\mid \f12<|\t X|<\f52, A+\rho\t X\in\O_L\right\}$
and $D_2=\left\{\t X\mid 1<|\t X|<2, A+\rho\t X\in\O_L\right\}$.
It follows from Proposition \ref{ld1} that
$$0\leq\psi_\rho\leq\f{Q}{\rho}\ \ \text{in $D_1$, $\psi_\rho=\f{Q}{\rho}$ and $\f{\p\psi_\rho}{\p\nu}=\ld$ on the free boundary}.$$
Denote $\phi_\rho=\f{Q}{\rho}-\psi_\rho$, it is easy to check
that
$$0\leq\phi_\rho\leq\f{Q}{\rho}\ \ \text{and}\ \ \Delta\phi_\rho+\rho f(\rho\phi_\rho)=0\ \ \text{ in
$D_1\cap\{\phi_\rho>0\}$}.$$ Obviously, $D_2\subset D_1$. Since
the boundary $\p D_2\cap\{\t X\mid A+\rho\t X\in\p\O_L\}$ is
$C^{2,\alpha}$, and $\phi_\rho=0$ on $\p D_2\cap\{\t X\mid
A+\rho\t X\in\p\O_L\}$, then the Harnack's inequality is still
valid up to the boundary $\p D_2\cap\{\t X\mid A+\rho\t
X\in\p\O_L\}$. By using the similar arguments in Theorem \ref{lb5}
and Lemma \ref{ld2}, we can obtain that
$$|\g\phi_\rho(\t X)|\leq C\ \ \text{in $D_2$},$$ where $C$ is a constant independent
of $\f{Q}{\rho}$. This implies  the estimate \eqref{b152}.

For the other case $X_0\in\p D_L$ with
$B_R(X_0)\cap\Gamma_{\ld,L}\neq\varnothing$ for any $R>0$, we can
show that the estimate \eqref{b152} is still valid by using the
above arguments.

\end{proof}

For the point $X_0=(x_0,y_0)\in(\p D_L\cap\bar\Gamma_{\ld,L})$,
there is a possible case that
$B_r(X_0)\cap\Gamma_{\ld,L}\cap\{x<x_0\}\neq\varnothing$ and
$B_r(X_0)\cap\Gamma_{\ld,L}\cap\{x>x_0\}\neq\varnothing$ for any
$r>0$, and thus for the $C^1$-regularity of $\psi_{\ld,L}$ and
$\Gamma_{\ld,L}$ at $X_0$, here we can not use the smooth fit
condition in Theorem 6.1 and Lemma 6.4 in \cite{ACF9} directly. (see
Figure \ref{f11})

Therefore, we will estimate the $C^1$-regularity of $\psi_{\ld,L}$
and $\Gamma_{\ld,L}$ near the free boundary $\p
D_L\cap\bar\Gamma_{\ld,L}$ as follows.

\begin{figure}[!h]
\includegraphics[width=120mm]{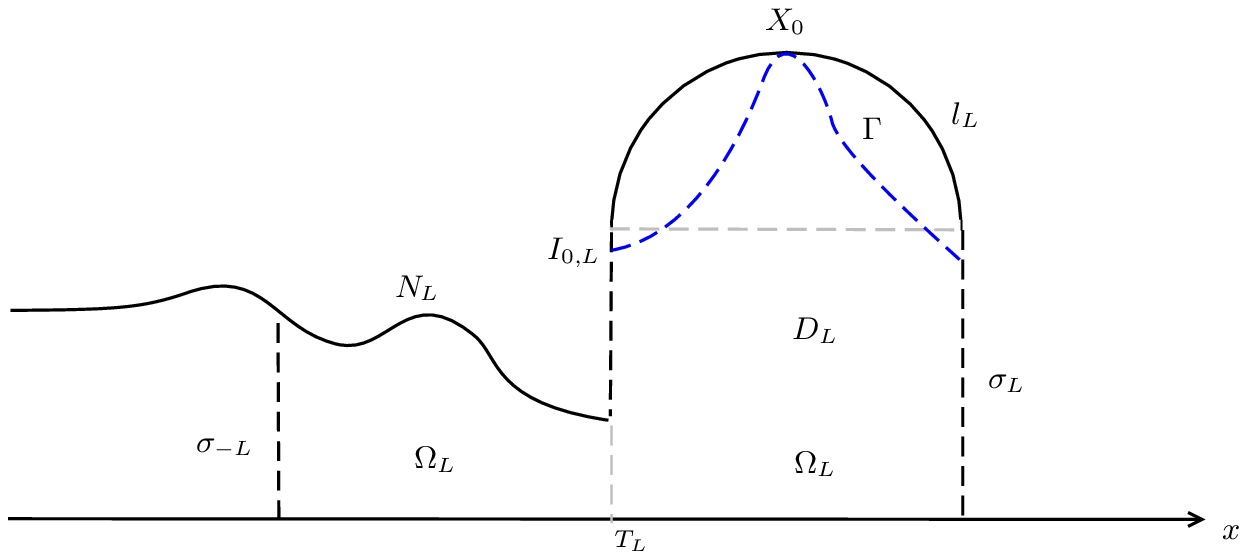}
\caption{The regularity near the free boundary}\label{f11}
\end{figure}

\begin{proposition}\label{lbb3}For any $X_0=(x_0,y_0)\in \p D_L$, if $B_r(X_0)\cap\Gamma_{\ld,L}\neq\varnothing$ for any $r>0$, then we have
$$\g\psi_{\ld,L}(X)\rightarrow\ld\nu_0\ \  \text{as}\ \ \ X\rightarrow X_0, \ \ X\in\O_L\cap\{\psi_{\ld,L}<Q\},$$ where $\nu_0$ is the outer normal vector to $\p D_L\cap\p\{\psi_{\ld,L}<Q\}$ at $X_0$. Moreover, $k_{\ld,L}(x)$ is $C^1$-smooth at $X=X_0$.
\end{proposition}

\begin{proof}

 Without loss of generality, we assume $X_0=\left(\f{L}2,\f{3L}2\right)$ and $\nu_0=(0,1)$.

Suppose $X_n\rightarrow X_0\in \p D_L$ with $X_n\in \p D_L$, and $\phi_0$ is a blow-up limit of $\phi_n(X)=\f{Q-\psi_{\ld,L}(X_n+r_nX)}{r_n}$, where $r_n=|X_n-X_0|$ and $X\in B_R(0)$ for any $R>0$.  By virtue of the results of Subsection 3.2, there exists a blow-up limit $\phi_0$, such that
\be\label{h1}\phi_0(x,y)=0\ \ \text{for any $y\geq0$, and }\ \Delta\phi_0=0\ \ \text{in $\{\phi_0>0\}$}.\ee It is easy to check that $B_{2}(0)\cap\p\{\phi_n>0\}\neq\varnothing$, we next claim that
\be\label{h2}\phi_0(X)=-\ld\min\{y-y_1,0\}\ \ \text{in $\mathbb{R}^2$ for some $y_1\leq 0$}.\ee

Consider the complex $z$-plane with $z=x+iy$. Let $l$ be a straight line in $z$-plane with the direction $(0,1)$, and passing through the origin. Set
$$\mathcal{D}_\ld=\{z\mid|z|\leq\ld\}\ \ \text{and}\ \ \mathcal{C}_\ld=\p\mathcal{D}_\ld.$$  Since the blow-up limit $\phi_0$ is still a harmonic function, we can use the similar arguments in the proof of Lemma 6.2 in \cite{ACF9} to show that
\be\label{h3}\limsup_{\phi(X)>0, X\rightarrow X_0}\text{dist}(\g\psi_{\ld,L}(X),\mathcal{D}_\ld\cup l)=0.\ee

Next, we will show that
\be\label{h4}|\g\phi_0(X)|\leq \ld.\ee
Suppose not, it follows from \eqref{h3} that
 \be\label{h5}\text{$\g\phi_0(X_0)\in l$ and $|\g\phi_0(X_0)|>\ld$ for some $X_0\in\{\phi_0>0\}$.}\ee
For the harmonic function $\phi_0$, one has
$$\text{if $
\g\phi_0\neq$const in $E$, then $X\rightarrow\g\phi_0(X)$ is an open mapping in $E$}$$ for any compact subsect $E$ of $\{\phi_0>0\}$. This together with \eqref{h5} implies that
\be\label{h6}\text{$\g\phi_0=$const in $\{\phi_0>0\}$}.\ee

Since $\phi_0(X)=0$ for any $y\geq0$, it follows from \eqref{h6}
that $\phi_0$ is linear in $E$ respect to $y$, and $\p E$ is a
straight line of the form $\{y=y_1\}$ with $y_1\leq0$. Thus, there
exists a constant $\gamma>\ld$, such that
$$\phi_0(X)=-\gamma(y-y_1)\ \ \ \text{in}\ \ \{y<y_1\},\ \ \ y_1\leq 0.$$ We next consider the following two cases for $y_1$.

{\bf Case 1.} $y_1<0$. In view of Lemma \ref{lc2}, $\phi_0$ is a local minimizer in $\{y<0\}$ and $\p\{\phi_0>0\}\subset\{y<0\}$, and Theorem 2.5 in \cite{AC1} gives that $|\g\phi_0|=\ld=\gamma$, which contradicts to $\gamma>\ld$.

{\bf Case 2.} $y_1=0$. Since $\phi_n\rightarrow\phi_0$ uniformly in
any compact subset of $\mathbb{R}^2$, for any $R>0$ and $\e>0$,
there exists a $N=N(\e,R)$, such that
$$\phi_n>0\ \ \text{in}\ \ B_R(0)\cap\{y<-\e\},\ \ \text{for}\ \ n>N.$$

Noticing that there are free boundary points $Z_n=(s_n,t_n)$ of $\phi_n$ in $B_{2}(0)$, the assumption $y_1=0$ implies that $t_n\rightarrow0$. Without loss of generality, we assume that $s_n\rightarrow 0$. Therefore,
$$|s_n|<\f{\e}4\ \ \ \text{and}\ \ \ |t_n|<\f\e4\ \ \ \text{for sufficiently large $n$}.$$
Introduce a function $h_{\delta,n}(x)$ as follows
$$h_{\delta,n}(x)=-2\e+\delta\eta_n(x),\ \ \ \text{for some}\ \ \delta>0,$$
where $$\eta_n(x)=\left\{\begin{array}{ll}
e^{-\f{9|x-\tau_n|^2}{1-9|x-\tau_n|^2}}&\text{for}\
\ |x-\tau_n|<\f13,\\
0 &\text{for}\ \ |x-\tau_n|\geq\f13,
\end{array}\right.$$ for a sequence $\tau_n\rightarrow 0$. Denote the domain
$E_{\delta,n}=B_1(0)\cap\{y<h_{\delta,n}(x)\}$, and choose a largest $\delta=\delta_n$, such that
$\phi_n>0$ in $E_{\delta_n,n}$ and $B_1(0)\cap\p E_{\delta_n,n}$
contains a free boundary point $\t Z_n=(\t s_n,\t t_n)$ of $\phi_n$. The $\tau_n\in\left(-\f\e n,\f\e n\right)$ can be chosen such that $\delta_n<2\e$ and $\t y_n=h_{\delta_n,n}(\t x_n)$.

Let $\o_n$ be the solution of the following Dirichlet problem
$$\left\{\begin{array}{ll}&\Delta\o_n+r_nf(r_n\o_n)=0\ \ \ \ \text{in}\ E_{\delta_n,n},\\
&\o_n=0\ \ \ \text{on}\ \p E_{\delta_n,n}\cap B_{\f 12}(0),\
\o_n=\zeta\phi_n\ \text{on}\ \p E_{\delta_n,n}\cap(
B_1(0)\setminus B_{\f 12}(0)),\\
&\o_n=\phi_n\ \ \ \text{on}\ \p E_{\delta_n,n}\cap \p B_1(0),
\end{array}\right.$$ where $\zeta(X)=\min\left\{\max\left\{2|X|-1,0\right\},1\right\}$.
Obviously, $\phi_n\geq\o_n$ on $\p E_{\delta_n,n}$ and $\Delta\phi_n+r_nf(r_n\phi_n)=0$ in $E_{\delta_n,n}$, the strong maximum principle gives that $\o_n<\phi_n$
in $E_{\delta_n,n}$. Thanks to Hopf's Lemma, one has
 \be\label{h7}\left|\f{\p\phi_n}{\p\nu_n}\right|>
\left|\f{\p\o_n}{\p\nu_n}\right|\ \ \ \text{at}\ \ \t Z_n,\ee where $\nu_n$ is the inner normal
vector to $E_{\delta_n,n}$ at $\t Z_n$.

The definition of $h_{\delta_n,n}$ gives that
$\|h_{\delta_n,n}\|_{C^{2,\alpha}}\leq C\e $, where the constant $C$
is independent of $\e$ and $n$. Applying the regularity theory for
the semilinear elliptic equation yields that
$$|\g\o_n-\g\phi_0|\leq C_1\e+C_2r_n \ \ \ \text{at}\ \ \ \t Z_n,$$
where the constants $C_1,C_2$ are independent of $\e$ and $n$, which together with \eqref{h7} implies that
$$\ld=|\phi_n(\t Z_n)|>|\g\o_n(\t Z_n)|\geq |\g\phi_0(\t Z_n)|-C_1\e-C_2r_n=\gamma-C_1\e-C_2r_n.$$ This leads a contradiction to the assumption $\gamma>\ld$, provided that $\e$ is small enough and $n$ is sufficiently large.

Hence, we complete the proof of \eqref{h4}.

Next, we will show the claim \eqref{h2}. Since $\phi_0$ is harmonic in $\{\phi_0>0\}$, it follows from Lemma 7.2 in \cite{ACF9} that
\be\label{h9}\f{\p q}{\p \nu}+\kappa q=0\ \ \text{on the free boundary of $\phi_0$},\ee where $q=|\g\phi_0|$, $\nu$ is the outer normal vector and the curvature $\kappa>0$ if the streamline is concave to the
fluid. Due to $|\g\phi_0|\leq\ld$ and $|\g\phi_0|=\ld$ on the free boundary $\p\{\phi_0>0\}$,  it follows from \eqref{h9} that $\p\{\phi_0>0\}$ is convex to the domain $\{\phi_0>0\}$, and the set $\{\phi_0=0\}$ is convex.

If $\phi_0>0$ in $\{y<0\}$. Denote $\o=\p_x\phi_0$, it is easy to see that $\o$ is harmonic in $\{\phi_0>0\}$. The fact $\phi_0=0$ on $\{y=0\}$ implies that $\o=0$. In view of $|\g\phi_0|\leq \ld$, one has
$$\lim_{r\rightarrow+\infty}\f{m(r)}{r}=0\ \  \ \text{where $r=|X|$ and $m(r)=\max_{|X|=r}|\o(X)|$}. $$ Therefore, thanks to Phragm\'en-Lindel\"of
Theorem (see also Theorem 1.1 in \cite{G1}), we have
$$\o=\p_x\phi_0=0\ \ \text{in}\ \ \{y<0\},$$ which implies that $\phi_0(X)$ is a function of $y$ and $$\phi_0(y)=-\gamma y\ \ \ \text{in}\ \ \{y<0\}.$$ It follows from Theorem 2.5 in \cite{AC1} that $\gamma=\ld$, and thus the claim \eqref{h2} holds.

If there exists a free boundary point $\t X=(\t x,\t y)$ of $\phi_0$ with $\t y<0$, denote $\Gamma_0$ the maximal free boundary arc containing $\t X$. We next consider the following two cases.

{\bf Case 1.} The free boundary $\Gamma_0$ is a straight line $l$. The fact $\phi_0=0$ in $\{y\geq0\}$ implies that the straight line $l$ must be parallel to the $x$-axis. Moreover, $\phi_0>0$ below $l$, due to that $\phi_0(X)$ is monotone decreasing with respect to $y$. Thus one has
$$l:\{y=y_1\}\ \ \ \text{and}\ \ \ \phi_0(y)=-\ld\min\{y-y_1,0\}.$$

{\bf Case 2.} The free boundary $\Gamma_0$ is not a straight line.
We next consider the following two subcases (see Figure \ref{f9}).

\begin{figure}[!h]
\includegraphics[width=110mm]{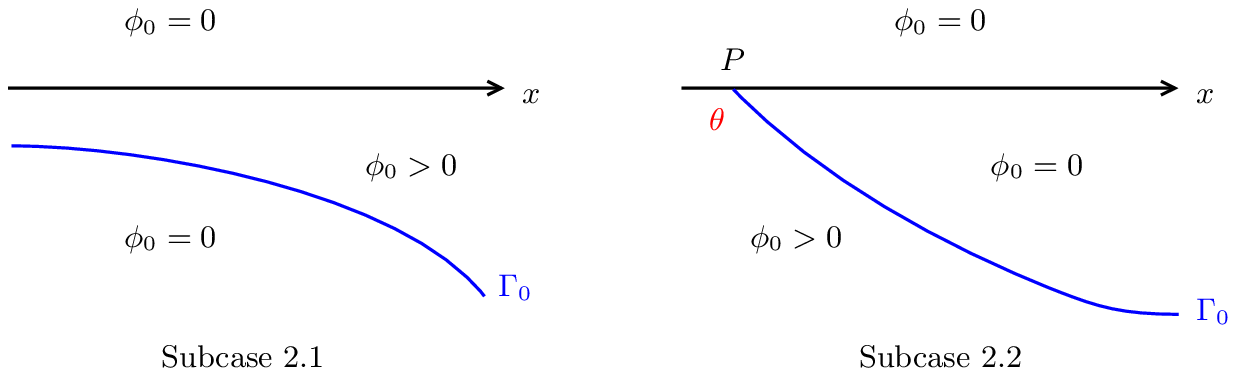}
\caption{Case 2}\label{f9}
\end{figure}

{\bf Subcase 2.1.} $\phi_0>0$ above $\Gamma_0$. Since $\phi_0(X)$ is monotone decreasing with respect to $y$, we can rule out this subcase.

 {\bf Subcase 2.2.} $\phi_0>0$ below $\Gamma_0$. The convexity of the free boundary $\Gamma_0$ implies that $\Gamma_0$ must intersect the $x$-axis, denote $P$ the intersection point. Moreover, the convexity of $\Gamma_0$ and monotonicity of $\phi_0$ with respect to $y$ imply that the angle $\th$ between $\Gamma$ and $x$-axis at point $P$ lies in $\left[\f\pi2,\pi\right)$. Applying the similar arguments in the proof of (2.30) in \cite{DX}, one has
\be\label{h10}\phi_0(X)\leq C|X-P|^{\f{2\pi}{\th+\pi+\e}}\ \  \ \text{near the intersection point $P$},\ee for small $\e\in\left(0,\f{\pi-\th}{4}\right)$. Denote $r=|X-P|$, using the non-degeneracy Lemma 3.4 and Remark 3.5 in \cite{ACF1}, we have
\be\label{h11}\f{1}{r}\fint_{\p B_r(P)}\phi_0dS\geq c_*\ld.\ee
 On the other hand, it follows from \eqref{h10} and \eqref{h11} that
  $$c_*\ld\leq\f{1}{r}\fint_{\p B_r(P)}\phi_0dS\leq C\f{r^{\f{2\pi}{\th+\pi+\e}}}{r}=C r^{\f{\pi-\th-\e}{\pi+\th+\e}},$$ which leads a contradiction, provided that $r$ is small enough.

  Hence, the proof of the claim \eqref{h2} is done.

With the aid of the claim \eqref{h2}, we will show that
\be\label{h12}\g\psi_{\ld,L}(X)\rightarrow\ld(0,1)\ \  \text{as}\ \ \ X\rightarrow X_0, \ \ X\in\O_L\cap\{\psi_{\ld,L}<Q\}.\ee
For any sequence of points $X_n=(x_n,y_n)\in\O_L\cap\{\psi_{\ld,L}<Q\}$ with $X_n\rightarrow X_0$, let $\t X_n=(\t x_n,\t y_n)\in \p D_L$ be the nearest point to $X_n$. Denote
$$r_n=|X_n-\t X_n|\ \ \ \text{and}\ \ \ d_n=\text{dist}(X_n,\Gamma_{\ld,L}).$$ Next, we consider the following two cases.

{\bf Case 1.} There exists a constant $C_0>1$, such that $r_n<C_0d_n$. Define a blow-up sequence $$\phi_n(X)=\f{Q-\psi_{\ld,L}(\t X_n+C_0d_nX)}{C_0d_n}\ \ \ \text{and}\ \ \ Z_n=\f{X_n-\t X_n}{C_0d_n}.$$ It is easy to check that $B_{2}(0)$ contains some free boundary points of $\phi_n$ and $B_{2c_0}(Z_n)\cap\p\{\phi_n>0\}=\varnothing$, which does not intersect the free boundary of $\phi_n$, where $c_0=\f{1}{4C_0}$. Moreover,
\be\label{h13}\Delta\phi_n+C_0d_nf(C_0x_n\phi_n)=0 \ \text{in $B_{2c_0}(Z_n)\cap\{\phi_n>0\}$}.\ee There exist a blow-up limit $\phi_0$ and a point $Z_0$ with $|Z_0|\leq 1$, such that
$$\phi_n\rightarrow\phi_0\ \ \text{in any compact subsets of $\mathbb{R}^2$ and $Z_n\rightarrow Z_0$.}$$
Due to the fact that $B_{2c_0}(Z_0)$ does not intersect the free boundary of $\phi_0$ and $\p D_L$ is $C^2$-smooth near $X_0$, one has
$$\text{$\phi_n\in C^{1,\alpha}$ in $B_{c_0}(Z_n)\cap\{\phi_n>0\}$ uniformly with respect to $n$}.$$ Then we have
\be\label{b22}\text{$-\g\psi_{\ld,L}(X_n)=\g\phi_n(Z_n)\rightarrow\g\phi_0(Z_0)$, as $n\rightarrow\infty$},\ee
which together with \eqref{h2} gives that \eqref{h12} holds.

{\bf Case 2.} $\f{r_n}{d_n}\rightarrow+\infty$. Define a blow-up
sequence
$$\phi_n(X)=\f{Q-\psi_{\ld,L}(\t X_n+r_nX)}{r_n}\ \ \ \text{and}\ \ \ Y_n=\f{X_n-\t X_n}{r_n}.$$ Obviously, $B_{2}(0)$ contains the free boundary points of $\phi_n$ for sufficiently large $n$. Then there exist a blow-up limit $\phi_0$ and a point $Y_0$ with $|Y_0|=1$, such that
$$\phi_n\rightarrow\phi_0=-\ld\min\{y-y_1,0\}\ \ \text{in any compact subsets of $\mathbb{R}^2$ and $Y_n\rightarrow Y_0$,}$$ where $y_1\leq 0$. Since the free boundary $\p\{\phi_n>0\}\rightarrow\p\{\phi_0>0\}$ locally in the Hausdorff distance (see Lemma \ref{lc2}) and $\p\{\phi_0>0\}$ is a straight line, we can conclude that the free boundary of $\phi_n$ satisfies the flatness condition at $(0,y_1)$ in the direction $(0,1)$, for sufficiently large $n$. By virtue of Theorem \ref{lc15}, one has the free boundary of $\phi_n$ is $C^{1,\alpha}$ $(0<\alpha<1)$ uniformly in $n$. The elliptic regularity gives that $\phi_n$ is uniformly $C^{1,\alpha}$ in $B_R(Y_n)\cap\{\phi_n>0\}$ for some $R>0$. Thus we have
$$\text{$-\g\psi_{\ld,L}(X_n)=\g\phi_n(Y_n)\rightarrow\g\phi_0(Y_0)=-\ld(0,1)$, as $n\rightarrow\infty$}.$$

Since the blow-up limit $\phi_0$ is a harmonic function, recalling the fact \eqref{h12}, along the similar arguments on the smooth fit condition (see Section 11 in Chapter 3 in \cite{FA1}), we can show  that
$$k'_{\ld,L}\left(\f{L}{2}\right)=0.$$
Hence, we complete the proof of the proposition.
\end{proof}





\subsection{The continuous fit condition for free boundary }
In this subsection, we will check the continuous fit condition of
the free boundary $\Gamma_{\ld,L}$ at $A$, namely, for any $L>\bar H=\max_{x\leq0}g(x)$,
there exists a $\ld_L\geq\ld_0$, such that
\be\label{b019}k_{\ld_L,L}(0)=g(0)=a.\ee

To obtain the continuous fit condition \eqref{b019}, we first
establish the continuous dependence of $\psi_{\ld,L}$ and
$k_{\ld,L}(0)$ with respect to the parameter $\ld$.
\begin{lemma}\label{ld8} If $\ld_n\rightarrow \ld$ with $\ld_n\geq\ld_0$, then
$$\psi_{\lambda_n, L}\rightarrow\psi_{\lambda, L}\ \ \text{weakly~~ in}~~H^1(\O_L)~~\text{and
uniformly in}~\O_L,$$ and
\be\label{b020}k_{\lambda_n, L}(0)\rightarrow k_{\lambda, L}(0)\ \ \text{for}~~ k_{\ld,L}(0)<L.\ee
\end{lemma}

\begin{proof}

It
follows from the similar arguments in the proof of Lemma \ref{lc2},
 there exists a subsequence $\{\psi_{\ld_n,L}\}$
and a  $\psi_0\in C^{0,1}(\O_L)$, such that
$$\psi_{\ld_n,L}(X)\rightarrow\psi_0(X)\ \text{uniformly in $\O_L$}.$$
Furthermore, $\psi_0$ is a minimizer to the variational problem
$(P_{\ld,L})$. The uniqueness of the minimizer to the variational
problem $(P_{\ld,L})$ gives that $\psi_0=\psi_{\ld,L}$.


Suppose that the assertion \eqref{b020} is not true, then there exists a sequence $\{k_{\ld_n,L}(0)\}$, such
that $$k_{\ld_n,L}(0)\rightarrow k_{\ld,L}(0)+\delta\ \
\text{as $n\rightarrow+\infty$},\ \ \delta\neq 0.$$ We consider the
following three cases and derive a contradiction.

{\bf Case 1.} $\delta<0$. The monotonicity of $\psi_{\ld,L}(x,y)$
with respect to $y$ gives that $k_{\ld,L}(0)+\delta\geq a$.

Next, we claim that \be\label{b021}\f{\p\psi_{\ld,L}(0+0,y)}{\p x}=-\ld\ \ \
\text{on}\ \ I_\e,\ee where $I_\e=\{(0,y)\mid
k_{\ld,L}(0)-4\e<y<k_{\ld,L}(0)-2\e\}$ with
$\e=-\f{\delta}6$.

It follows from the statement (4) in Proposition \ref{ld1} that
\be\label{b022}-\f{\p\psi_{\ld,L}(0+0,y)}{\p x}=|\g\psi_{\ld,L}|\geq\ld \ \ \text{on}\ \ I_\e.\ee

To obtain the claim \eqref{b021}, it suffices to show that
 \be\label{b023}\f{\p\psi_{\ld,L}(0+0,y)}{\p x}\geq-\ld\ \ \
\text{on}\ \ I_\e.\ee

Denote $\phi_n=Q-\psi_{\ld_n,L}$, $\phi=Q-\psi_{\ld,L}$ and $$U_{\e,\tau}=\{(x,y)\mid -\tau<x<\tau,
k_{\ld,L}(0)-5\e<y<k_{\ld,L}(0)-\e\}$$ for small $\tau>0$, Proposition \ref{ld1}
gives that the free boundary $U_{\e,\tau}\cap\p\{\phi_n>0\}$ is $C^{3,\alpha}$, then we have
$$\text{$U_\e\cap\p\{\phi_n>0\}\rightarrow U_{\e,\tau}\cap\p\{\phi>0\}$  in
$C^{1,\alpha}$ for some $\alpha\in(0,1)$.}$$

For any fixed $X_1=(0,y_1)\in U_{\e,\tau}\cap\p\{\phi>0\}$, then there exists
a sequence $X_n\in U_\e$ with $\phi_n(X_n)=0$, such that
$X_n\rightarrow X_1$ as $n\rightarrow+\infty$. In fact, suppose that
there exists a small $r>0$, such that $\phi_n>0$ and
$\Delta\phi_n+f(\phi_n)=0$ in $B_r(X_1)$, which gives that
$\Delta\phi+f(\phi)=0$ in $B_r(X_1)$ and $\phi(X_1)=0$. The maximum
principle implies that $\phi\equiv 0$ in $B_r(X_1)$, which
contradicts to $X_1\in U_{\e,\tau}\cap\p\{\phi>0\}$.

Take a small $r>0$ with $E_n\subset B_r(X_1)\subset U_{\e,\tau}$, such
that \be\label{b027}E_n=B_r(X_1)\cap\{x>\e_n\},\ \phi_n>0\
\text{in}\ E_n\ \text{and}\ X_n\notin E_n,\ee where $\e_n\downarrow
0$ as $n\rightarrow+\infty$. Define a function $h_{s,n}(y)$ as
follows
$$h_{s,n}(y)=\e_n-s\eta\left(\f{2(y-y_1)}{r}\right),\ \ s>0,$$
where \be\label{b024}\eta(y)=\left\{\begin{array}{ll}
e^{-\f{y^2}{1-y^2}}&\text{for}\
\ |y|<1\\
0 &\text{for}\ |y|\geq1.
\end{array}\right.\ee Denote the domain
$E_{s,n}=B_r(X_1)\cap\{x>h_{s,n}(y)\}$. It is easy to check that
$E_{0,n}=E_n$. Let $s=s_n\leq\e_n$ to be the largest one, such that
$\phi_n>0$ in $E_{s_n,n}$, and $\t X_n=(\t x_n,\t
y_n)\in(B_r(X_1)\cap\p E_{s_n,n})\cap\Gamma_{\ld_n,L}$. Furthermore,
$$\t x_n=h_{s_n,n}(\t y_n)\ \ \text{and} \ \ s_n\rightarrow0 \ \ \text{as}\ \ n\rightarrow+\infty.$$

Let $\varphi_n$ be the solution of the following Dirichlet problem
$$\left\{\begin{array}{ll}&\Delta\varphi_n+f(\varphi_n)=0\ \ \text{in}\ E_{s_n,n},\\
&\varphi_n=0\ \text{on}\ \p E_{s_n,n}\cap B_{\f r2}(X_1),\
\varphi_n=\zeta\phi_n\ \text{on}\ \p E_{s_n,n}\cap(
B_r(X_1)\setminus B_{\f r2}(X_1)),\\
&\varphi_n=\phi_n\ \text{on}\ \p E_{s_n,n}\cap \p B_r(X_1),
\end{array}\right.$$ where $\zeta(X)=\min\left\{\max\left\{\f{2|X-X_1|-r}{r},0\right\},1\right\}$.
It is clear that $\varphi_n\leq\phi_n$ on $\p E_{s_n,n}$, which
gives that \be\label{b028}\ld_n=\f{\p\phi_n(\t X_n)}{\p\nu_n}\geq
\f{\p\varphi_n(\t X_n)}{\p\nu_n},\ee where $\nu_n$ is the inner
normal vector to $\p E_{s_n,n}$ at $\t X_n$.

Taking $\e_n\rightarrow0$ implies that
$$h_{s_n,n}(y)\rightarrow 0\ \ \text{in}\ \
C^{1,\beta}\ \ \text{as}\ \ n\rightarrow+\infty.$$ Thanks to the
standard estimates of the solutions of the semilinear elliptic equation, we conclude that $\varphi_n$ in $E_{s_n,n}\cap B_{\f
r2}(X_1)$ converges to $\phi$ in $\{\phi>0\}\cap B_{\f r2}(X_1)$ in
$C^{1,\beta}$-sense. Denote $\t X_n\rightarrow \t
X\in\p\{\phi>0\}\cap B_{\f r2}(X_1).$ It follows from \eqref{b028}
that \be\label{b029}\f{\p\varphi_n(\t
X_n)}{\p\nu}\rightarrow\f{\p\phi(\t X)}{\p\nu} \ \ \text{and}\ \
\f{\p\phi(\t X)}{\p\nu}\leq \ld.\ee  Taking $r\rightarrow 0$ yields that $\t X\rightarrow X_1$ and $$\ld\geq\f{\p\phi(\t
X)}{\p\nu}\rightarrow\f{\p\phi( X_1)}{\p\nu}=-\f{\p\psi_{\ld,L}(X_1)}{\p
x},\ \ X_1\in U_{\e,\tau}\cap\p\{\phi>0\},$$ which together with \eqref{b023} gives that the claim \eqref{b021} holds.

With the aid of the fact \eqref{b021}, we can
obtain a contradiction by using the similar arguments in the proof
of Lemma \ref{ld5}.

{\bf Case 2.} $\delta>0$ and $k_{\ld,L}(0)<a$. Similar to the
proof of the claim \eqref{b021} in Case 1, one has
 \be\label{bb021}\f{\p\psi_{\ld,L}(0-0,y)}{\p x}=\ld\ \ \
\text{on}\ \ I_\e,\ee where $I_\e=\{(0,y)\mid
k_{\ld,L}(0)+\e<y<k_{\ld,L}(0)+2\e\}$ with
$\e=\f{\min\{\delta,a-k_{\ld,L}(0)\}}3$, and we can obtain a
contradiction.

{\bf Case 3.} $\delta>0$ and $k_{\ld,L}(0)\geq a$.  Since
$\psi_{\ld,L}=Q$ on $I_{0,L}$, it follows from the statement (4) in Proposition \ref{ld1}
that
$$\f{\p\psi_{\ld_n,L}(0+0,y)}{\p x}\geq\ld_n\ \ \text{on}\ \
I_\delta,$$ for sufficiently large $n$, where
$I_\delta=\left\{(0,y)\mid
k_{\ld,L}(0)+\f{\delta}4<y<k_{\ld,L}(0)+\f{3\delta}4\right\}$.

Let $E_n$ be a domain bounded by $x=0$, $y=k_{\ld_n,L}(x)$,
$y=k_{\ld,L}(0)+\f{\delta}4$ and
$y=k_{\ld,L}(0)+\f{3\delta}4$. Moreover,
$$x_n=\max\left\{x\mid
k_{\ld_n,L}(x)=k_{\ld,L}(0)+\f{\delta}4\right\}\rightarrow
0\ \ \text{as}\ \ n\rightarrow+\infty.$$ Thanks to the
non-oscillation Lemma \ref{ld4} for $\psi_{\ld_n,L}$ in $E_n$, there
exists a constant $C$ independent of $n$, such that
$$0<\f\delta2\leq Cx_n,$$
which gives a contradiction for sufficiently large $n$.

\end{proof}

Next, we will study the relation between the initial point $(0,k_{\ld,L}(0))$ and $\ld$ for
any $L>\bar H$.

\begin{lemma}\label{ld9}For any $L>\bar H$, the initial point $(0,k_{\ld,L}(0))$ satisfies that \\
(1) there exists a constant $C_0>0$ (independent of $L$), such that $k_{\ld,L}(0)<a$
for any $\ld>C_0$.\\
(2) there exists a $c_0\geq\ld_0$ (independent of $L$), such
that $k_{\ld,L}(0)\geq a$ for any $\ld\leq c_0$.

\end{lemma}
\begin{proof} (1) Suppose that there exist a free boundary point and a sufficiently large
$\ld$, such that $k_{\ld,L}(0)\geq a$. Then there exist a
$X_0\in D_L$ and a disc $B_{r_0}(X_0)$ with $r_0>0$ (independent of $L$), such
that $\Gamma_{\ld,L}\cap B_{\f {r_0}2}(X_0)\neq\varnothing$. It
follows from the non-degeneracy Lemma \ref{lb7} and Remark \ref{re1}
that
$$\f{Q}{r_0}\geq\f{1}{r_0}\fint_{\p{B_{r_0}(X_0)}}Q-\psi_{\ld,L}dS\geq
c_{\f12}^*\ld,$$ which leads to a contradiction for sufficiently large $\ld$.

Similarly, we can obtain a contradiction if there is no free
boundary point, taking $X_0=(0,y_0)$ with $y_0\in(a,L)$ and using
the non-degeneracy Lemma \ref{lb6} for the boundary point (see
Remark \ref{re1}).

(2) Suppose not, then for any $c_0\geq\ld_0$, there exists a $\t\ld\in[\ld_0,c_0]$, such that $k_{\t\ld,L}(0)<a$. In view of Remark \ref{re4}, we have
\be\label{h15}\text{the asymptotic height $\t h=h_{\t\ld}$ in downstream lies in $[a,H]$, }\ee provided that $c_0-\ld_0$ is small.
Let $\t\Psi_\e(y)=\Psi_{\t\ld}(y+\e)$ for $\e\geq 0$, choosing $\e_0\geq0$ to be the smallest one, such that
$$\psi_{\t\ld,L}(X)\leq\t\Psi_{\e_0}(y)\ \text{in $\O_L\cap\{\t\Psi_{\e_0}<Q\}$ and $\psi_{\t\ld,L}(X_0)=\t\Psi_{\e_0}(y_0)$}$$  for $X_0=(x_0,y_0)\in\overline{\O_L\cap\{\t\Psi_{\e_0}<Q\}}$. It follows from \eqref{h15} that $\e_0>0$. We first show that
\be\label{b017}X_0\notin\O_L\cap\{\t\Psi_{\e_0}<Q\}.\ee Suppose not,
there exists a point $X_0\in\O_L\cap\{\t\Psi_{\e_0}<Q\}$, such that
$0<\t\Psi_{\e_0}(X_0)=\psi_{\t\ld,L}(X_0)<Q$. Then there exists a small $r>0$
such that
$$0<\psi_{\t\ld,L}<Q\ \text{and}\ \ 0<\t\Psi_{\e_0}<Q\ \text{in}\
B_r(X_0).$$ On the other hand,
$$\Delta(\psi_{\t\ld,L}-\t\Psi_{\e_0})+f_0(\psi_{\t\ld,L})-f_0(\t\Psi_{\e_0})=0\ \ \ \text{in
$B_r(X_0)$}.$$ The
strong maximum principle gives that
$\psi_{\t\ld,L}(X_0)\equiv\t\Psi_{\e_0}(X_0)$ in $B_r(X_0)$, due to $\psi_{\t\ld,L}(X_0)=\t\Psi_{\e_0}(X_0)$.
Applying the strong maximum principle again, then
$\psi_{\t\ld,L}(X_0)\equiv\t\Psi_{\e_0}(X_0)$ in $\O_L$, which leads a contradiction.

The boundary value of $\psi_{\t\ld,L}$ gives that $X_0\notin\p\O_L$. Next, we claim that \be\label{h16}X_0\neq(0,k_{\t\ld,L}(0)).\ee In fact, it follows from Proposition \ref{lbb3} that the free boundary $\Gamma_{\t\ld,L}$ is $C^1$ at $X_0$ and its tangent is in the direction of the positive $y$-axis, the definition of $\e_0$ implies the claim \eqref{h16}.

 Thus, let $X_0$ be a free boundary point of $\psi_{\t\ld,L}$. Since the free boundary
$\Gamma_{\t\ld,L}$ is $C^{3,\alpha}$ at $X_0$, it follows from Hopf's
lemma that
$$\t\ld=\f{\p\t\Psi_{\e_0}}{\p\nu}<\f{\p\psi_{\t\ld,L}}{\p\nu}=\t\ld\ \ \text{at}\ \
X_0,$$ where $\nu=(0,1)$ is the outer normal vector of
$\Gamma_{\t\ld,L}$ at $X_0$, which leads a contradiction.

\end{proof}

With the aid of Lemma \ref{ld9}, we can define $$\Sigma_L=\{\ld\geq\ld_0\mid
k_{\ld,L}(0)<a\}.$$
and the set $\Sigma_L$ is non-empty for sufficiently large $\ld$.

Moreover, set \be\label{b030}\ld_L=\inf_{\ld\in\Sigma_L}\ld.\ee

Furthermore, it follows from Lemma \ref{ld9} that
\be\label{b031}\ld_0\leq\ld_L\leq C_0<\infty,\ee where $C_0$ is a constant
independent of $L$.

 By virtue of Lemma \ref{ld8} and Lemma
\ref{ld9}, we have
\begin{proposition}\label{ld10} For any $L>\bar H=\max_{x\leq0}g(x)$, there exists a $\ld_L\geq\ld_0$ as in \eqref{b030},
 such that the free boundary $\Gamma_{\ld_L,L}$ satisfies the continuous fit condition and the smooth fit condition
 \be\label{h20}k_{\ld_L,L}(0)=g(0)=a\ \ \ \text{and}\ \ \ k'_{\ld_L,L}(0)=g'(0).\ee
 Moreover,
 \be\label{h21}k_{\ld_L,L}(x)\leq \bar H \ \ \ \text{for any $x\in[0,L]$}, \ee and  \be\label{h19}\psi_{\ld_L,L}(x,y)\leq \min\{\Psi_{\ld_L}(y),Q\} \ \ \text{in $\O_L$}. \ee
\end{proposition}

\begin{proof} {\bf Step 1.} The definition of $\ld_L$ and the continuous dependence of $k_{\ld,L}(0)$ with respect to $\ld$ imply that
 $$k_{\ld_L,L}(0)=g(0)=a,$$ which together with Proposition \ref{lbb3} gives that $k'_{\ld_L,L}(0)=g'(0)$. Thus, the continuous fit condition and the smooth fit condition \eqref{h20} hold.

{\bf Step 2.} Suppose that the assertion \eqref{h21} is not true, then there exists an $\t x\in(0,L]$, such that
\be\label{h22}k_{\ld_L,L}(\t x)>\bar H=\max_{x\leq0}g(x).\ee On the other hand, the fact $\ld_L\geq\ld_0$ implies that the asymptotic height $h_{\ld_L}\leq H\leq \bar H$. Thus we have
\be\label{h23}k_{\ld_L,L}(\t x)>\bar H\geq H\geq h_{\ld_L}.\ee

For $\mu<\ld_0$ and $\ld_0-\mu$ is small, define $$\o_\mu(y)=\int_0^yu_{1,
\mu}(t)dt,\ \ u_{1,\mu}=\sqrt{u_0^2(\chi^{-1}(t;p_{diff}))+2p_{diff}} \ \ \text{with}\ \ p_{diff}=\f{\mu^2-\ld^2_0}2,$$ where $\chi^{-1}(t;p_{diff})$ is defined as in Subsection 4.3.
Denote $\o_\mu^\e(y)=\min\{\o_{\mu}(y-\e),Q\}$ for any $\e\geq0$. Let
$\e_0\geq 0$ be the smallest one, such that
$$\o_\mu^{\e_0}(y)\leq\psi_{\ld_L,L}(X)\ \ \text{in $\O_L$, and $\o_\mu^{\e_0}(X_0)=\psi_{\ld_L,L}(X_0)$}$$ for some $X_0=(x_0,y_0)\in\overline{\O_L\cap\{\psi_{\ld_L,L}<Q\}}$.
It follows from \eqref{h23} that $\e_0>0$. The strong maximum principle implies that $X_0\notin\O_L\cap\{\psi_{\ld_L,L}<Q\}$, and thus $\psi_{\ld_L,L}(X_0)=Q$. Similar to the proof of Lemma \ref{ld9}, we can show that $x_0\neq L$. We next consider the following three cases.

{\bf Case 1.} $X_0\in\Gamma_{\ld_L,L}$. Since the free boundary
$\Gamma_{\ld_L,L}$ is $C^{3,\alpha}$ at $X_0$, the Hopf's
lemma gives that
$$\mu=\f{\p\o_\mu^{\e_0}}{\p\nu}>\f{\p\psi_{\ld_L,L}}{\p\nu}=\ld_L\ \ \text{at}\ \
X_0,$$ where $\nu=(0,1)$ is the outer normal vector of
$\Gamma_{\ld_L,L}$ at $X_0$, which contradicts to the assumption $\mu<\ld_0\leq\ld_L$.

{\bf Case 2.} $X_0\in\p D_L$ and $\Gamma_{\ld_L,L}\cap B_r(X_0)=\varnothing$ for small $r>0$. In view of that the boundary
$\p D_L$ is $C^{2}$ at $X_0$, it follows from Hopf's
lemma that
\be\label{h25}\ld_L\geq\ld_0>\mu=\f{\p\o_\mu^{\e_0}}{\p\nu}>\f{\p\psi_{\ld_L,L}}{\p\nu}\ \ \text{at}\ \
X_0.\ee On the other hand, the statement (4) in Proposition \ref{ld1} gives that
$$\f{\p\psi_{\ld_L,L}}{\p\nu}\geq \ld_L\ \ \text{at}\ \
X_0, $$ which contradicts to \eqref{h25}.

{\bf Case 3.} $X_0\in\p D_L$ and $\Gamma_{\ld_L,L}\cap B_r(X_0)\neq\varnothing$ for any $r>0$. In view of Proposition \ref{lbb3} and $\o_\mu^{\e_0}(y)\leq\psi_{\ld_L,L}(X)$ in $\O_L$, one has
$$\ld_0>\mu=\f{\p\o_\mu^{\e_0}}{\p\nu}\geq\f{\p\psi_{\ld_L,L}}{\p\nu}=\ld_0\ \ \text{at}\ \
X_0,$$  which leads a contradiction.

{\bf Step 3.} In this step, we will show the assertion \eqref{h19}. Denote $\o_\e(y)=\min\{\Psi_{\ld_L}(y+\e),Q\}$ for any $\e\geq0$, and choosing the smallest $\e_0\geq0$, such that
$$\o_{\e_0}(y)\geq\psi_{\ld_L,L}(X)\ \ \text{in $\O_L\cap\{\psi_{\ld_L,L}<Q\}$, and $\o_{\e_0}(y_0)=\psi_{\ld_L,L}(X_0)$}$$ for some $X_0=(x_0,y_0)\in\overline{\O_L\cap\{\o_{\e_0}<Q\}}$. It suffices to show that
$$\e_0=0.$$ If not, suppose that $\e_0>0$. By virtue of the strong maximum principle and the boundary value of $\psi_{\ld_L,L}$, we can choose $X_0$ to be the free boundary point of $\psi_{\ld_L,L}$. Thus, we can derive a contradiction by using Hopf's lemma.

\end{proof}








\subsection{The existence of the incompressible jet flow}
In previous subsections, we show that there exist a $\ld_L$ and a
minimizer $\psi_{\ld_L,L}$ for any $L>\bar H$, such that the free
boundary $\Gamma_{\ld_L,L}$ satisfies the continuous fit condition
\eqref{h20}. Taking $L\rightarrow+\infty$, we will obtain the
existence of the solution to the jet flow problem in this
subsection.

First, we consider the following variational problem.

{\bf{The variational problem $(P_\ld)$:}}
 For any  $L_0>\bar H=\max_{x\leq 0}g(x)$, find a $\psi_\ld\in K$ such that
$$J_{\ld,L_0}(\psi_\ld)\leq J_{\ld,L_0}(\psi),$$
for any $\psi\in K$ with $\psi=\psi_\ld$ on $\p \O_{L_0}$, the
admissible set
$$K=\{\psi\in H^1_{loc}(\mathbb{R}^2)\mid  \psi=0~~\text{lies below}~T,
\psi=Q~~\text{lies above}~N\}.$$

By using the similar arguments in Lemma \ref{ld8},  taking a
sequence $\{L_n\}$ with $L_n\rightarrow\infty$, such that
$$
\ld_{L_n}\rightarrow\ld~~\text{and}~~\psi_{\lambda_{L_n},L_n}\rightarrow
\psi_\ld\ \ \text{weakly~~ in}~~H_{loc}^1(\mathbb{R}^2) ~\text{and
uniformly in any compact subset of $\mathbb{R}^2$},$$ as
$n\rightarrow\infty$, and $\psi_\ld$ is a minimizer to the
variational problem $(P_\ld)$. It follows from \eqref{b031} that
$$\ld_0\leq \ld<+\infty.$$
Lemma \ref{ld3} gives that $\psi_\ld(x,y)$ is monotone increasing
with respect to $y$, and the free boundary $\Gamma_\ld$ of
$\psi_\ld$ is $x$-graph. Moreover, $\Gamma_\ld$ can be described by
a continuous function $y=k_{\ld}(x)$ for $x\in(0,+\infty)$, and
$$\Gamma_\ld=\O\cap\p\{\psi_\ld<Q\}=\{(x,y)\in \O\mid x>0,
y=k_\ld(x)\}.$$

By virtue of Proposition \ref{ld10}, one
has
\be\label{h31}k_\ld(0)=g(0)=a,\ \ k'_\ld(0)=g'(0)\ \ \text{and $k_{\ld}(x)\leq \bar H$ for any $x\in(0,\infty)$},\ee and
\be\label{h32}\psi_{\ld}(x,y)\leq \min\{\Psi_{\ld}(y),Q\} \ \ \text{in $\O$}. \ee

 Theorem \ref{lc15} and Proposition \ref{lb2} give that the free
boundary $\Gamma_\ld$ is $C^{3,\alpha}$ and
$|\g\psi_\ld|=\f{\p\psi_\ld}{\p\nu}=\ld$ on $\Gamma_\ld$, where
$\nu$ is the outer normal vector.

Next, we will obtain the positivity of horizontal velocity.

\begin{lemma}\label{ld11}
$u=\f{\p\psi_\ld}{\p y}>0$ in $\overline{\O\cap\{\psi_\ld<Q\}}$.
\end{lemma}
\begin{proof}

Denote
$u=\f{\p\psi_\ld}{\p y}$, one has
$$\Delta u+b(X) u=0\  \ \text{in}~\O\cap\{\psi_\ld<Q\},$$ where $b(X)=f_0'(\psi(X))\leq0$. The monotonicity of $\psi_\ld(x,y)$ with respect to $y$ implies that $u\geq 0$ in $\O\cap\{\psi_\ld<Q\}$, the strong maximum principle gives
that
$$\text{$u>0$ in $\O\cap\{\psi_\ld<Q\}$}.$$

Since $\psi_\ld<Q$ in $\O\cap\{x<0\}$ and $\psi_\ld=Q$ on $N$, it
follows from $f'_0(Q)\leq 0$ and Hopf's lemma that
$$u=\f{1}{\sqrt{1+(g'(x))^2}}\f{\p\psi_{\ld}}{\p \nu}>0\ \ \text{on $N\setminus A$}.$$
Similarly, we have
$$u>0\ \ \text{on $T$.}$$

Next, we will show that $u>0$ on $\Gamma_\ld$. Suppose that there
exists $x_0>0$, such that $u=\p_y\psi_\ld=0$ on $(x_0,k_\ld(x_0))$.
Since the free boundary $\Gamma_\ld$ is $C^{3,\alpha}$-smooth at
$(x_0,k_\ld(x_0))$, the normal outer vector is $(1,0)$ at
$(x_0,k_\ld(x_0))$ and $|\p_x\psi_\ld|=\ld$ on $\Gamma_\ld$. Thanks to Hopf's lemma,
we have \be\label{b320}\left|u_x\right|=\left|\f{\p u}{\p
\nu}\right|> 0\ \ \text{at}\ \ (x_0,k_\ld(x_0)).\ee

Since $u^2+v^2=\ld^2$ on $\Gamma_\ld$, we have
 $$\ba{rl}0=\f{\p(u^2+v^2)}{\p s}
 =2uu_y+2vv_y=2vv_y=-2\ld u_x\ \ \text{at $(x_0,k_\ld(x_0))$,} \ea $$ which contradicts
to \eqref{b320}.

By virtue of the boundedness of $g'(0)$, we have $u>0$ at
$A$.

The positivity of horizontal velocity $u$ implies that the function $y=k_{\ld}(x)$ is $C^1$ for any $x\geq 0$.
\end{proof}

Next, we will obtain the asymptotic height of the free boundary $\Gamma_\ld$.

\begin{lemma}\label{ld13} $\lim_{x\rightarrow\infty}k_\ld(x)=h_\ld$, where $h_\ld$ is the asymptotic height of the free boundary $\Gamma_\ld$ and $h_\ld\leq a$.
\end{lemma}
\begin{proof} Suppose that the limit $\lim_{x\rightarrow\infty}k_\ld(x)$ does not exist. It follows from the non-degeneracy \ref{lb7} and \eqref{h31} that $$0<c_0\leq k_\ld(x)\leq \bar H\ \ \ \text{for any $x\in[0,\infty)$},$$ which implies that there exist two sequence $\{x_n\}$ and $\{\t x_n\}$, such that
\be\label{h40}\lim_{x_n\rightarrow\infty}k_\ld(x_n)=\limsup_{x\rightarrow\infty}k_\ld(x)=h_{\ld,1}\ \ \text{and}\ \ \lim_{\t x_n\rightarrow\infty}k_\ld(\t x_n)=\liminf_{x\rightarrow\infty}k_\ld(x)=h_{\ld,2} \ee with $h_{\ld,1}>h_{\ld,2}$.

By virtue of Remark \ref{re4}, there exist $\ld_1,\ld_2$ with $\ld_0\leq\ld_1<\ld_2<\infty$, such that
\be\label{h41}0< c_0\leq h_{\ld,2}=h_{\ld_2}<h_{\ld_1}=h_{\ld,1}\leq \bar H.\ee

Set $\psi_n(t,s)=\psi_\ld(x_n+t,s)$ and the free boundary $\Gamma_n:s=k_{\ld}(x_n+t)$.
Define a curve
$$\gamma_\delta: s=\bar H-\delta\eta\left(\f{t}{\sqrt{x_n}}\right)$$ where the function $\eta(x)$ is defined as in \eqref{b024}. Let $\delta_n$ be the largest one, such that the curve $\gamma_{\delta_n}$ touches the free boundary $\Gamma_n$ at some points $(t_n,s_n)$, and thus
\be\label{h42}\text{$s_n=k_\ld(x_n+t_n)$ and $\delta_n\leq \bar H-k_{\ld}(x_n)$.}\ee
Let $\o_n$ be the solution to the Dirichlet problem
\be\label{h43}\left\{\ba{ll} &\Delta\o+f_0(\o)=0  \ \text{in}\ \ \ E_n,\\
&\o=Q\ \text{on}\ \gamma_{\delta_n},\ \ \o=\psi_n\ \ \text{on}\ \ \p E_n\setminus\gamma_{\delta_n}, \ea\right.\ee  where domain $E_n$ is bounded by $t=-\f{x_n}2$, $t=\f{x_n}2$, $s=0$ and $\gamma_{\delta_n}$.

Since $\Delta\psi_n+f_0(\psi_n)\leq0$ in $E_n$ and $\psi_n(t_n,s_n)=\o_n(t_n,s_n)=Q$, the maximum principle implies that $\psi_n\geq\o_n$ in $E_n$ and \be\label{h44}\ld=\f{\p\psi_n}{\p\nu_n}\leq\f{\p\o_n}{\p\nu_n}\ \ \text{at}\ \ (t_n,s_n),\ee where $\nu_n$ is the outer normal vector of $\Gamma_n$.

By virtue of \eqref{h42}, we can choose a subsequence $\{x_n\}$, such that
\be\label{h45}\f{t_n}{\sqrt{x_n}}\rightarrow t_0\in[-1,1]\ \ \text{and}\ \ \bar H-\delta_n\eta\left(\f{t}{\sqrt{x_n}}\right)\rightarrow\bar h_{\ld_1}.\ee

After a translation in the $t$ direction, such that the free boundary point $(t_n,s_n)$ lies on the $s$-axis, then we have
\be\label{h46}E_n\rightarrow E_0=\{(t,s)\mid -\infty<t<\infty, 0<s<h_{\ld_1}\}\ \ \text{and} \ \ (t_n,s_n)\rightarrow (0,h_{\ld_1}),\ee and $\o_n\rightarrow\o_0$ uniformly in any compact subset of $E_0$. Moreover, $\o_0$ satisfies that $0\leq\o_0(t,s)\leq Q$ and
\be\label{h47}\left\{\ba{ll} &\Delta\o_0+f_0(\o_0)=0  \ \text{in}\ \ \ E_0,\\
&\o_0(t,0)=0,\ \ \o_0(t,h_{\ld_1})=Q. \ea\right.\ee By using the similar arguments in \cite{XX3}, we can show that the Dirichlet problem \eqref{h47} has a unique solution
$$\o_0(s)=\int_0^{s}u_1(x)dx,$$ where $u_{1}(t)=\sqrt{u_0^2(\chi^{-1}(t;p_{diff}))+2p_{diff}}$ with $p_{diff}=\f{\ld_1^2-\ld^2_0}2$, and $\chi^{-1}(t;p_{diff})$ is defined as in Subsection 4.3.

It follows from \eqref{h44} and \eqref{h46} that  \be\label{h48}\ld\leq\lim_{n\rightarrow\infty}\f{\p\o_n(t_n,s_n)}{\p\nu_n}=\f{\o_0(0,h_{\ld_1})}{\p s}=\ld_1.\ee

Next, we will show that
\be\label{h49}\ld\geq\ld_2.\ee.

Define $\t\psi_n(t,s)=\psi_\ld(\t x_n+t,s)$, $\t\Gamma_n: s=k_\ld(\t x_n+t)$ as the free boundary of $\t\psi_n$ and a curve
$$\t\gamma_\delta:s=\f{h_2}{2}+\delta\eta\left(\f{t}{\sqrt{\t x_n}}\right).$$ Let $\delta_n$ be the largest one, such that
the curve $\t\gamma_{\delta_n}$ touches the free boundary $\t\Gamma_n$ at a point $(t_n,s_n)$, and thus
$$\text{$s_n=k_\ld(\t x_n+t_n)$ and $\delta_n\leq k_{\ld}(\t x_n)-\f{h_{\ld,2}}2$.}$$ Denote a domain $\t E_n$, which is bounded by $t=-\f{\t x_n}2$, $t=\f{\t x_n}2$, $s=0$ and $\t\gamma_{\delta_n}$.
Let $\t\o_n$ be the solutions to the Dirichlet problem \eqref{h43}, replacing $E_n$ and $\gamma_{\delta_n}$ by $\t E_n$ and $\t\gamma_{\delta_n}$.
 Obviously, $\Delta\psi_n+f_0(\psi_n)=0$ in $\t E_n$, it follows from the maximum principle that $\t\o_n\geq\t\psi_n$ in $\t E_n$ and \be\label{h50}\ld=\f{\p\t\psi_n}{\p\nu_n}\geq\f{\p\t\o_n}{\p\nu_n}\ \ \text{at}\ \ (t_n,s_n),\ee where $\nu_n$ is the outer normal vector of $\Gamma_n$.
By using the previous arguments, we can show that there exists a subsequence $\{\t x_n\}$, such that
$$\t E_n\rightarrow \t E_0=\{(t,s)\mid -\infty<t<\infty, 0<s<h_{\ld_2}\}\ \ \text{and} \ \ (t_n,s_n)\rightarrow (0,h_{\ld_2}),$$ and $\t\o_n\rightarrow\t\o_0$ uniformly in any compact subset of $E_0$ and
$$\lim_{n\rightarrow\infty}\f{\p\t\o_n(t_n,s_n)}{\p\nu_n}=\f{\t\o_0(0,h_{\ld_2})}{\p s}=\ld_2,$$ which together with \eqref{h50} give that \eqref{h49} holds.

Hence, it follows from \eqref{h48} and \eqref{h49} that
$$\ld_2\leq\ld\leq\ld_1,$$ which contradicts to the fact $\ld_1<\ld_2$.

Thus, there exists an asymptotic height $\t h$ of the free boundary $\Gamma_\ld$, such that
$$k_{\ld}(x)\rightarrow \t h.$$ Next, we claim that
$$\text{$\t h=h_\ld$}.$$ If not, without loss of generality, we assume that $\t h> h_\ld$, and there exists a $\t\ld<\ld$, such that $\t h=h_{\t\ld}$. Similar to the proof of \eqref{h48} and \eqref{h49}, we can show that $$\ld\leq \t\ld,$$ which contradicts to the assumption $\t\ld<\ld$.

Finally, it follows from \eqref{h32} that
$$h_\ld\leq a.$$

\end{proof}

Finally, the asymptotic behavior of the jet flow will be obtained in
the following.

\begin{proposition}\label{ld12}
 The rotational jet flow satisfies the following asymptotic behavior in the
far fields,
 $$(u,v,p)\rightarrow(u_0(y),0,p_{in}),\ \nabla
u\rightarrow(0,u_0'(y)),~~\nabla v\rightarrow 0,~~\nabla
p\rightarrow0,$$ uniformly in any compact subset of $(0,H)$, as
$x\rightarrow-\infty$, and
$$(u,v,p)\rightarrow(u_1(y),0,p_{atm}),\ \nabla u\rightarrow(0,u_1'(y)),~~\nabla v\rightarrow 0,~~\nabla
p\rightarrow0,$$ uniformly in any compact subset of $(0,h_\ld)$, as
$x\rightarrow+\infty,$ where $u_1(y)$ and $h_\ld$ are uniquely
determined by $u_0(y)$, $p_{in}$ and $p_{atm}$ as in Remark \ref{re4}.
\end{proposition}

\begin{proof}
For any sequence $\{\psi_{\ld}(x-n,y)\}$, by using the uniform elliptic estimate, there exists a subsequence
$\{\psi_{\ld}(x-n,y)\}$, such that
$$\psi_{\ld}(x-n,y)\rightarrow\psi_0(x,y)\ \ \text{in $C^{2,\alpha}(S)$},$$ where $S$ is any compact subsect of $(0,H)$, $0\leq\psi_0\leq Q$ and
\be\label{b032}\left\{\ba{ll} &\Delta\psi_0+f_0(\psi_0)=0  \ \text{in}~~\{-\infty<x<+\infty\}\times\{0<y<H\},\\
&\psi_0(x,H)=Q\ \text{and}\ \psi_0(x,0)=0\ \ \text{for
$-\infty<x<+\infty$}. \ea\right. \ee Then boundary value problem
\eqref{b032} possesses a unique solution
$$\psi_0(y)=\int_0^{y}u_0(s)ds,$$ which implies that
$$(u,v)(x,y)\rightarrow \left(u_0(y),0\right),\ \nabla u(x,y)\rightarrow\left(0,u_0'(y))\right)\ \text{and}\
\nabla v(x,y)\rightarrow(0,0)$$ uniformly in any compact subset of
$(0,H)$, as $x\rightarrow-\infty$.  The Bernoulli's law gives that
$$p(x,y)\rightarrow p_{in}=\f{\ld^2}{2}-\f{u_0^2(H)}2+p_{atm}\ \text{and}\
\nabla p(x,y)\rightarrow 0,$$ uniformly in any compact subset of
$(0,H)$, as $x\rightarrow-\infty$.

For any sequence $\{\psi_{n}\}$ with $\psi_n(x,y)=\psi_{\ld}(x+n,y)$ and $R>0$, by virtue of Lemma \ref{ld13}, one has
\be\label{bbb1}\lim_{n\rightarrow \infty}k_\ld(x+n)=h_\ld.\ee For any large $R$ and small $\e>0$, it follows from \eqref{bbb1} that there exists a $N=N(R,\e)$, such that the free boundary of $\psi_n(X)$ lies in a $\e$-neighborhood of the line $\{y=h_\ld\}$ in $B_R(0)$, and thus the free boundary of $\psi_n$ satisfies the flatness condition in $B_R(0)$, provided that $n$ is sufficiently large. It follows from Theorem \ref{lc15} that the free boundary of $\psi_n$ convergence to the line $\{y=h_\ld\}$ in $C^1$ norm, namely, $$k'_\ld(x+n)\rightarrow 0\ \ \text{as}\
n\rightarrow\infty,\
 \ \text{for any $|x|<R$}.$$ Applying Theorem \ref{lc15} again, one has
$$\left|k^{(j)}_\ld(x)\right|\leq C\ \ \text{for sufficiently large $x>0$},\ \ j=2,3.$$

By using the uniform elliptic estimate, there exists a subsequence
$\{\psi_n\}$, such that
$$\psi_{n}(x,y)\rightarrow\psi_1(x,y),$$ and
\be\label{b033}\left\{\ba{ll} &\Delta\psi_1+f_0(\psi_1)=0  \ \text{in}~~ E,\\
&\psi_1(x,h_\ld)=Q\ \ \text{and}\ \ \psi_1(x,0)=0\ \
\text{for $-\infty<x<+\infty$},  \ea\right. \ee  where $E=\{-\infty<x<+\infty\}\times\{0<y<h_\ld\}$.

Recalling Remark \ref{re4},
\be\label{b034}
\psi_1(y)=\int_0^yu_1(y)ds\ \ \ \text{in}\ \ E,\ee solves uniquely Dirichlet problem \eqref{b033} in $E$, where
$u_1(t)=\sqrt{u_0^2(\chi^{-1}(y;p_{diff}))+2p_{diff}}$ with $p_{diff}=\f{\ld^2-u_0^2(H)}2$.

In view of \eqref{b034}, we conclude that
$$(u,v,p)(x,y)\rightarrow \left(u_1(y),0,p_{atm}\right),\ \nabla
u(x,y)\rightarrow\left(0,u_1'(y)\right),\ \nabla
v(x,y)\rightarrow(0,0)$$ and $\nabla p(x,y)\rightarrow(0,0)$
uniformly in any compact subset of $(0,h_\ld)$, as
$x\rightarrow+\infty$.

\end{proof}







\subsection{The uniqueness of the incompressible jet flow}
In this subsection, we will obtain the uniqueness of $\ld$ and the
solution to the jet flow problem.

\begin{proposition}\label{le1}$\ld$ and the solution $(u,v,p,\Gamma)$ to the jet flow problem are unique.

\end{proposition}
\begin{proof}
Suppose that $(\psi_{\ld_1},\ld_1,\Gamma_{\ld_1})$ and
$(\t\psi_{\ld_2},\ld_2,\t\Gamma_{\ld_2})$ are two solutions to the
jet flow problem, which satisfy the conditions in Definition
\ref{def2}.

{\bf Step 1.} In this step, we will show that $\ld_1=\ld_2$. Suppose
not, without loss of generality, we assume that $\ld_1<\ld_2$. In
view of the asymptotic behaviors of $\psi_{\ld_1}$ and
$\t\psi_{\ld_2}$ in Proposition \ref{ld12}, one has
\be\label{e1}k_{\ld_1}(x)>\t k_{\ld_2}(x)\ \ \text{for sufficiently
large $x>0$.}\ee

Consider a function $\psi_{\ld_1}^{\e}(x,y)=\psi_{\ld_1}(x,y-\e)$
for $\e\geq0$ and $\Gamma_{\ld_1}^{\e}$ is the free boundary of
$\psi_{\ld_1}^{\e}$, choosing the smallest $\e_0\geq0$ such that
\be\label{e2}\psi_{\ld_1}^{\e_0}(X)\leq\t\psi_{\ld_2}(X)\ \text{in
$\O$, and}\ \psi_{\ld_1}^{\e_0}(X_0)=\t\psi_{\ld_2}(X_0)\ \text{for
some $X_0\in\overline{\O\cap\{\t\psi_{\ld_2}<Q\}}$}.\ee
 Similar to the proof of \eqref{b017}, the strong maximum principle implies that
$X_0\notin \O\cap\{\t\psi_{\ld_2}<Q\}$. We next consider the
following two cases.

{\bf Case 1.} $\e_0=0$, we can choose $X_0=A$. Then one has
$$\f{\p\psi_{\ld_1}}{\p\nu}=|\g\psi_{\ld_1}|=\ld_1\ \ \text{and}\ \ \f{\p\t\psi_{\ld_2}}{\p\nu}=|\g\t\psi_{\ld_2}|=\ld_2\ \ \
\text{at $A$,}$$where $\nu$ is the outer normal vector. Since
$\psi_{\ld_1}\leq\t\psi_{\ld_2}$ in $\O$, which yields that
$$\ld_1=\f{\p\psi_{\ld_1}}{\p\nu}\geq\f{\p\t\psi_{\ld_2}}{\p\nu}=\ld_2\ \ \
\text{at $A$,}$$ which leads a contradiction to the assumption
$\ld_1<\ld_2$.

{\bf Case 2.} $\e_0>0$, it follows from \eqref{e1} that
$|X_0|<+\infty$. Then we can take $X_0$ be the free boundary point of
$\psi^{\e_0}_{\ld_1}$ and $\t\psi_{\ld_2}$, and the strong maximum
principle gives that
$$\psi_{\ld_1}^{\e_0}(X)<\t\psi_{\ld_2}(X)\ \ \text{in $\O\cap\{\t\psi_{\t\ld_2}<Q\}$}.$$
Since the free boundaries $\Gamma^{\e_0}_{\ld_1}$ and
$\t\Gamma_{\ld_2}$ are $C^{3,\alpha}$ at $X_0$, it follows from
Hopf's lemma that
$$\ld_1=\f{\p\psi_{\ld_1}^{\e_0}}{\p\nu}>\f{\p\t\psi_{\ld_2}}{\p\nu}=\ld_2\
\ \ \text{at $X_0$,}$$ where $\nu$ is the outer normal vector to
$\Gamma_{\ld_1}^{\e_0}\cap\t\Gamma_{\ld_2}$ at $X_0$. This leads a
contradiction to the assumption $\ld_1<\ld_2$.

Hence, we obtain the uniqueness of $\ld$, and denote
$\ld=\ld_1=\ld_2$.

{\bf Step 2.} In this step, we will show that $\psi_\ld=\t\psi_\ld$.
By virtue of the asymptotic behaviors of $\psi_{\ld_1}$ and
$\t\psi_{\ld_2}$ in Proposition \ref{ld12}, we have
$$\lim_{x\rightarrow+\infty}k_{\ld}(x)=\lim_{x\rightarrow+\infty}\t k_{\ld}(x)=h_\ld.$$
Without loss of generality, we assume that there exists some
$x_0\in(0,+\infty)$, such that \be\label{e3}k_{\ld}(x_0)<\t
k_{\ld}(x_0).\ee

Denote $\psi_{\ld}^{\e}(x,y)=\psi_{\ld}(x,y-\e)$ for $\e\geq0$, and
let $\e_0\geq0$ to be the smallest one, such that
\be\label{e4}\psi_{\ld}^{\e_0}(X)\leq\t\psi_{\ld}(X)\ \text{in $\O$,
and}\ \psi_{\ld}^{\e_0}(X_0)=\t\psi_{\ld}(X_0)\ \text{for some
$X_0\in\overline{\O\cap\{\t\psi_{\ld}<Q\}}$}.\ee It follows from assumption
\eqref{e3} that $\e_0>0$ and $|X_0|<+\infty$.

Similar to the proof of Case 1 in Step 1, we can take $X_0$ be the free
boundary point of $\psi_\ld^{\e_0}$ and $\t\psi_{\ld}$, applying Hopf's
lemma yields that
$$\ld=\f{\p\psi_{\ld}^{\e_0}}{\p\nu}>\f{\p\t\psi_{\ld}}{\p\nu}=\ld\
\ \ \text{at $X_0$,}$$ where $\nu$ is the outer normal vector to
$\Gamma_{\ld}^{\e_0}\cap\t\Gamma_{\ld}$ at $X_0$. This leads a
contradiction.
\end{proof}


\bibliographystyle{plain}

\end{document}